\documentclass[11pt]{amsart}
\usepackage{amsmath, amsthm, amssymb, amsfonts}
\usepackage[normalem]{ulem}
\usepackage{hyperref}
\usepackage{verbatim} 

\usepackage{mathtools}

\usepackage{tikz-cd}
\usetikzlibrary{arrows}
\usepackage{lmodern}
\usetikzlibrary{decorations.pathmorphing}
\tikzset{snake it/.style={decorate, decoration=snake}}

\theoremstyle{plain}
\newtheorem{thm}{Theorem}
\newtheorem{cor}{Corollary}
\newtheorem{lemma}{Lemma}

\newtheorem{corstar}[cor]{Corollary*}
\newtheorem{prop}{Proposition}
\newtheorem{conjecture}{Conjecture}  

\theoremstyle{definition}

\theoremstyle{remark}

\newcommand{\BC}{{\mathbb{C}}}

\newcommand{\BH}{{\mathbb{H}}}

\newcommand{\BQ}{{\mathbb{Q}}}
\newcommand{\BR}{{\mathbb{R}}}

\newcommand{\BZ}{{\mathbb{Z}}}

\newcommand{\CC}{{\mathcal C}}
\newcommand{\CD}{{\mathcal D}}

\newcommand{\CF}{{\mathcal F}}

\newcommand{\CJ}{{\mathcal J}}
\newcommand{\CK}{{\mathcal K}}

\newcommand{\CO}{{\mathcal O}}
\newcommand{\CP}{{\mathcal P}}

\newcommand{\CZ}{{\mathcal Z}}

\newcommand{\blangle}{\big\langle}
\newcommand{\brangle}{\big\rangle}

\newcommand{\pt}{{\mathsf{p}}}
\newcommand{\ch}{{\mathrm{ch}}}

\newcommand{\bw}{{\mathbf{w}}}
\newcommand{\bc}{{\mathbf{c}}}

\newcommand{\aaa}{\mathsf{a}}
\newcommand{\bbb}{\mathsf{b}}
\newcommand{\A}{\mathsf{A}}

\DeclareFontFamily{OT1}{rsfs}{}
\DeclareFontShape{OT1}{rsfs}{n}{it}{<-> rsfs10}{}
\DeclareMathAlphabet{\curly}{OT1}{rsfs}{n}{it}

\newcommand\End{\operatorname{End}}
\newcommand{\p}{\mathbb{P}}

\newcommand*\dd{\mathop{}\!\mathrm{d}}
\newcommand{\Mbar}{{\overline M}}
\newcommand{\vir}{{\text{vir}}}

\newcommand{\Pic}{\mathop{\rm Pic}\nolimits}

\newcommand{\1}{\mathsf{1}}
\newcommand{\kk}{\mathsf{k}}
\newcommand{\C}{\mathsf{C}}
\newcommand{\Aut}{\mathrm{Aut}}
\newcommand{\DR}{\mathsf{DR}}
\newcommand\ev{\operatorname{ev}}
\newcommand{\QMod}{\mathsf{QMod}}
\newcommand{\Mod}{\mathsf{Mod}}

\newcommand{\w}{\mathbf{w}}
\newcommand{\Sb}{\ME}
\newcommand{\E}{\mathsf{E}}
\newcommand{\ME}{\mathsf{ME}}

\begin{document}
\baselineskip=14.6pt

\title[Holomorphic anomaly equations]
{Holomorphic anomaly equations
and the Igusa cusp form conjecture}

\author{Georg Oberdieck}
\address{MIT, Department of Mathematics}
\email{georgo@mit.edu}

\author{Aaron Pixton}
\address{MIT, Department of Mathematics}
\email{apixton@mit.edu}

\date{\today}
\begin{abstract}
Let $S$ be a K3 surface and let $E$ be an elliptic curve.
We solve the reduced Gromov--Witten theory of the Calabi--Yau threefold $S \times E$
for all curve classes which are primitive in the K3 factor.
In particular, we deduce the Igusa cusp form conjecture.

The proof relies on new results in the Gromov--Witten theory of elliptic curves and K3 surfaces.
We show the generating series of
Gromov-Witten classes of an elliptic curve
are cycle-valued quasimodular forms
and satisfy a holomorphic anomaly equation.
The quasimodularity generalizes a result by Okounkov and
Pandharipande, and the holomorphic anomaly equation
proves a conjecture of Milanov, Ruan and Shen.
We further conjecture quasimodularity and holomorphic anomaly
equations for the cycle-valued Gromov-Witten theory
of every elliptic fibration with section.
The conjecture generalizes the holomorphic anomaly equations
for elliptic Calabi--Yau threefolds predicted by
Bershadsky, Cecotti, Ooguri, and Vafa.
We show a modified conjecture holds numerically
for the reduced Gromov-Witten theory of K3 surfaces in primitive classes.
\end{abstract}

\maketitle

\setcounter{tocdepth}{1} 
\tableofcontents
\setcounter{section}{-1}
\newpage
\section{Introduction}
\subsection{Overview}
Let $S$ be a non-singular projective K3 surface and let $E$ be an elliptic curve.
In 1999, Katz, Klemm and Vafa \cite{KKV} predicted 
that the topological string partition function of the Calabi--Yau threefold
\[ X = S \times E \]
is the reciprocal of the Igusa cusp form $\chi_{10}$, a Siegel modular form.
In 2014 a conjecture for the reduced Gromov--Witten theory of $X$ in all curve classes was presented in \cite{K3xE}.
In the primitive case (i.e. for curve classes which are primitive in the K3 factor)
the conjecture matches exactly the earlier physics prediction.
We call the primitive case of the conjecture the Igusa cusp form conjecture.\footnote{
The Katz--Klemm--Vafa conjecture usually refers to the result proven in \cite{PT2}.}
In this paper we solve the reduced Gromov--Witten theory of $X$
in the primitive case and prove the Igusa cusp form conjecture.

The main tool used in the proof is
the correspondence between
Gromov--Witten theory (counting stable maps)
and Pandharipande--Thomas theory (counting sheaves) proven in \cite{PaPix1, PaPix2}.
Both sides yield modular constraints
and taken together, they determine the partition function from a single coefficient.
The sheaf theory side was developed in \cite{O1, OS1}
and yields the elliptic transformation law of Jacobi forms
(proven by derived auto-equivalences and wall-crossing in the motivic Hall algebra).
On the Gromov--Witten side we apply the product formula \cite{B2}
and study the theory for the K3 surface and the elliptic curve separately.
We prove the following new ingredients:
\begin{itemize}
 \item[(i)] A holomorphic anomaly equation for the cycle-valued Gromov--Witten theory of the elliptic curve $E$
 (Sections~\ref{Subsection_intro_Elliptic_curves} and~\ref{Subsection_intro_HAE})
 \item[(ii)] A holomorphic anomaly equation for the numerical reduced Gromov--Witten theory of the K3 surface $S$ in primitive classes
 (Section~\ref{Subsection_intro_K3surfaces}).
\end{itemize}
Part (i) contains a proof of the quasimodularity of the cycle-valued theory.
For both the elliptic curve and the K3 surface the holomorphic anomaly equation is formulated on the cycle-level and motivates
a conjectural holomorphic anomaly equation for elliptic fibrations with section (Section~\ref{Subsection_Elliptic_fibrations}).

\subsection{The Igusa cusp form conjecture} \label{Subsection_intro_Igusa_cusp_form}
Let
\[ \pi_1 : X \to S, \quad \pi_2 : X \to E \]
be the projections to the two factors and let
\[ \iota_S : S \hookrightarrow X, \quad \iota_E : E \hookrightarrow X \]
be inclusions of fibers of $\pi_2$ and $\pi_1$ respectively.

Let $\beta \in H_2(S,\BZ)$ be a non-zero curve class
and let $d$ be a non-negative integer.
The pair $(\beta,d)$ determines a class in $H_2(X,\mathbb{Z})$ by
\[ (\beta,d) = \iota_{S \ast}( \beta ) + \iota_{E \ast}(d [E] ). \]

The moduli space of stable maps
$\Mbar^{\bullet}_{g}(X, (\beta,d))$
from disconnected genus~$g$ curves to $X$
representing the class $(\beta,d)$
carries a reduced\footnote{Since $S$ is holomorphic symplectic
the (ordinary) virtual fundamental class vanishes. The theory is non-trivial only after reduction \cite{K3xE}.}
virtual fundamental class
\[ [ \Mbar^{\bullet}_{g}(X, (\beta,d)) ]^{\text{red}} \]
of dimension $1$.
Let $\pt \in H^2(E,\BZ)$ be the class Poincar\'e dual to a point,
and let $\beta^{\vee} \in H^2(S,\BQ)$ be any class
satisfying 
\[ \langle \beta , \beta^{\vee} \rangle = 1 \]
with respect to the intersection pairing on $S$.
Following \cite{K3xE}, reduced Gromov--Witten invariants of $X$ are defined by
\begin{equation} \label{DefnNghd}
\mathsf{N}_{g,\beta,d}
=
\int_{ [ \Mbar_{g,1}^{\bullet}(X, (\beta,d)) ]^{\text{red}} }
\ev_1^{\ast}\left(
\pi_1^{\ast} ( \beta^{\vee} ) \cup \pi_2^{\ast} ( \pt ) \right).
\end{equation}
By a degeneration argument $\mathsf{N}_{g,\beta,d}$ is independent of the choice of $\beta^{\vee}$. 

The elliptic curve $E$ acts on the moduli space
$\Mbar^{\bullet}_{g}(X, (\beta,d))$ by translation with $1$-dimensional orbits.
The Gromov--Witten invariant $\mathsf{N}_{g,\beta,d}$ is a virtual count of these $E$-orbits,
and hence enumerates (with degeneracies and multiplicities) maps from
algebraic curves to $X$ up to translation.

Let $\beta_h \in H_2(S,\BZ)$
be a primitive class satisfying
\[ \langle \beta_h, \beta_h \rangle = 2h-2. \]
By deformation invariance
$\mathsf{N}_{g,\beta_h,d}$
only depends on $g,h$ and $d$. We write
\[ \mathsf{N}_{g,h,d} = \mathsf{N}_{g,\beta_h,d} \,. \]
The partition function of primitive invariants is defined by
\begin{equation} \label{K3xEpartition}
\CZ(u,q, \tilde{q})
=
\sum_{g=0}^{\infty} \sum_{h=0}^{\infty} \sum_{d=0}^{\infty} 
\mathsf{N}_{g,h,d} u^{2g-2} q^{h-1} \tilde{q}^{d-1}.
\end{equation}

Consider the classical Jacobi theta functions
\[
\theta_2(q) = \sum_{n \in \BZ} q^{(n+\frac{1}{2})^2}, \ \
\theta_3(q) = \sum_{n \in \BZ} q^{n^2}, \ \
\theta_4(q) = \sum_{n \in \BZ} (-1)^n q^{n^2} \,.
\]
Let $c(n) \in \BZ$ be the 
Fourier coefficients of the following
meromorphic modular form for $\Gamma_0(4)$ of weight $-1/2$:
\begin{multline*} \label{hh}
\sum_{n} c(n) q^n = \frac{40 \theta_3( q )^4 - 8 \theta_4( q )^4 }{ \theta_3( q ) \theta_2( q )^4 } \\
= 2 q^{-1} + 20 - 128q^{3} + 216q^{4} - 1026q^{7} + 1616q^{8} + \ldots \,.
\end{multline*}
The Igusa cusp form $\chi_{10}$ is a weight $10$ Siegel modular form
of genus $2$,
defined as the Borcherds lift\footnote{See Gritsenko--Nikulin \cite{GN}.}
\begin{equation} \label{IGUSA}
\chi_{10}(p,q, \tilde{q}) \ = \ p q \tilde{q}
\prod_{k,h,d} (1-p^k q^h \tilde{q}^d)^{c(4hd-k^2)},
\end{equation}
where the product runs over all
$k \in \BZ$ and $h,d \geq 0$ such that
\begin{enumerate}
\item[$\bullet$]
 $h>0$ or $d>0$,
\item[$\bullet$]
  $h = d = 0$ and $k < 0$.
\end{enumerate}
We will assume the variables $p,q,\tilde{q}$
are taken in the non-empty open region defined by
$|p^k q^h \tilde{q}^d|<1$ whenever $4hd-k^2 \geq -1$.

The following result proves the Igusa cusp form conjecture \cite[Conj.A]{K3xE}.

\begin{thm} \label{thmIgusa}
The partition function $\CZ(u,q, \tilde{q})$
is the Laurent expansion of $-1/\chi_{10}$
under the variable change $p=e^{iu}$,
\[
\CZ(u,q, \tilde{q})
\ = \ 
- \frac{1}{\chi_{10}(p, q, \tilde{q})} \,.
\]
\end{thm}
\vspace{8pt}

In genus $0$ and class $(\beta_h, 0)$
the Gromov--Witten invariants
enumerate rational curves on the K3 surface. 
Theorem~\ref{thmIgusa} then specializes to
the Yau--Zaslow formula proven by Beauville \cite{Beauville} and Bryan--Leung \cite{BL}:
\[
\sum_{h = 0}^{\infty} \mathsf{N}_{0,h,0} q^{h-1} =
\frac{1}{\Delta(q)},
\]
where the right hand side is the reciprocal of the modular discriminant
\[ \Delta(q) = q \prod_{m \geq 1} (1-q^m)^{24}. \]
More generally
$\mathsf{N}_{g,h,0}$
are the $\lambda_g$-integrals
in the Gromov-Witten theory of K3 surfaces and
we obtain the Katz--Klemm--Vafa formula proven in \cite{MPT}:
\begin{equation*} \label{KKV}
\sum_{g,h}
\mathsf{N}_{g,h,0} u^{2g-2} q^{h-1}
=
\frac{1}{(p - 2 + p^{-1})} \prod_{m \geq 1}
\frac{1}{(1-p q^m)^2 (1-q^m)^{20} (1-p^{-1} q^m)^{2}}.
\end{equation*}
We list several other known cases.
In case $h=0$ the invariants $\mathsf{N}_{g,h,d}$ were obtained by Maulik in \cite{M_An}.
The cases $h \in \{0,1\}$ were shown by Bryan \cite{Bryan-K3xE}
and a second time in \cite{OS1}.
The cases $d \in \{ 1,2 \}$ can be found in \cite{K3xP1}.

Theorem~\ref{thmIgusa} determines the Gromov--Witten invariants of $S \times E$ in the primitive case.
A conjecture in all curve classes $(\beta,d)$ has been proposed in \cite{K3xE}.
The case $d=0$ corresponds to the
imprimitive Katz--Klemm--Vafa
formula and was proven in \cite{PT2}.
The case $\beta=0$ is proven in \cite{OSK0} on the
sheaf theory side.
The intermediate cases remain open.

\subsection{Elliptic curves} \label{Subsection_intro_Elliptic_curves}
Let $E$ be a non-singular elliptic curve, and let
\[ \Mbar_{g,n}(E,d) \]
be the moduli space of degree $d$ stable maps of connected curves of genus $g$
to $E$ with $n$ markings.
Consider the correspondence\footnote{
We assume here that $g, n$ lie
in the stable range i.e. take only those values
for which the moduli spaces $\Mbar_{g,n}$ and $\Mbar_{g,n}(E,d)$
are Deligne-Mumford stacks.
We follow the same convention throughout the paper.
In all equations or diagrams or sums
we assume $(g,n)$
to lie in the range where all moduli spaces are
Deligne--Mumford stacks.
}
\begin{equation*} \label{correspondence}
\begin{tikzcd}
\Mbar_{g,n}(E,d) \ar{d}{\pi} \ar{rr}{\ev_1 \times \cdots \times \ev_n}
& & E^n \\
\Mbar_{g,n}
\end{tikzcd}
\end{equation*}
defined by the evaluation maps at the markings
$\ev_1, \ldots, \ev_n$,
and the forgetful morphism $\pi$ to the moduli space
of stable curves.
Gromov--Witten classes
of $E$ are defined by
the action of the virtual fundamental class
\[ [ \Mbar_{g,n}(E,d) ]^{\vir} \in H_{\ast}( \Mbar_{g,n}(E,d) ) \]
on cohomology classes $\gamma_1, \ldots, \gamma_n \in H^{\ast}(E)$
via the correspondence: 
\begin{equation} \label{GWclasses}
\CC_{g,d}(\gamma_1, \ldots, \gamma_n)
=
\pi_{\ast}
\left( [ \Mbar_{g,n}(E,d) ]^{\vir} \prod_{i=1}^{n} \ev_i^{\ast}(\gamma_i) \right) \, \in H^{\ast}(\Mbar_{g,n}),
\end{equation}
where we have suppressed an application of Poincar\'e duality on $\Mbar_{g,n}$.

Define the generating series
\[ \CC_g( \gamma_1, \ldots, \gamma_n )
=
\sum_{d=0}^{\infty} \CC_{g,d}(\gamma_1, \ldots, \gamma_n) q^d,
\]
which by definition is an element of $H^{\ast}(\Mbar_{g,n}) \otimes \BQ[[q]]$.

The ring of quasimodular forms is the free polynomial algebra
\[ \QMod = \BQ[C_2, C_4, C_6] \]
where $C_k$ are the weight $k$ Eisenstein series
\begin{equation} \label{EisensteinSeries}
C_{k}(q) = -\frac{B_k}{k \cdot k!} + \frac{2}{k!}
\sum_{n \geq 1} \sum_{d|n} d^{k-1} q^n
\end{equation}
and $B_k$ are the Bernoulli numbers.

\begin{thm}
\label{thm_E_quasimodularity}
For any $\gamma_1, \ldots, \gamma_n \in H^{\ast}(E)$
the series
$\CC_g( \gamma_1, \ldots, \gamma_n )$
is a cycle-valued quasimodular form:
\[
\CC_g( \gamma_1, \ldots, \gamma_n )
\in 
H^{\ast}(\Mbar_{g,n}) \otimes \QMod.
\]
\end{thm}
\vspace{7pt}

The Gromov--Witten invariants of $E$
are obtained from
the Gromov--Witten classes
by integration against the cotangent line classes
$\psi_i \in H^2(\Mbar_{g,n})$,
\[
\sum_{d = 0}^{\infty} \langle \tau_{k_1}(\gamma_1) \ldots \tau_{k_n}(\gamma_n) \rangle^E_{g,d} q^d
=
\int_{\Mbar_{g,n}} \psi_1^{k_1} \cdots \psi_n^{k_n} \cdot \CC_g(\gamma_1, \ldots, \gamma_n).
\]
Hence Theorem~\ref{thm_E_quasimodularity}
generalizes\footnote{Our argument is independent of \cite{OP1, OP3} and in fact yields a new proof.}
the quasimodularity of the
Gromov-Witten invariants of elliptic curves
proven by Okounkov and Pandharipande \cite{OP1, OP3}.

The \emph{double ramification cycle}
\[
\DR_g(\mu, \nu) \in A^{g}(\Mbar_{g,n})
\]
parametrizes curves of genus $g$ admitting a map to $\p^1$
with given ramification profiles $\mu$ over $0 \in \p^1$ and
$\nu$ over $\infty \in \p^1$. A precise definition is given in Section~\ref{Subsection_Double_ramification_Cycle}. The key ingredient in our study of $\CC_g(\gamma_1, \ldots, \gamma_n)$ is the polynomiality of the double ramification cycle in the parts of the ramification profiles. This polynomiality is a difficult result proved by a combinatorial study \cite{PZ} of an explicit formula for the double ramification cycle \cite{JPPZ}.

The proof of Theorem~\ref{thm_E_quasimodularity} proceeds by 
degenerating the elliptic curve to a rational nodal curve.
After degeneration
the Gromov--Witten classes
of the elliptic curve are expressed
as a trace-like sum of
double ramification cycles.
Quasimodularity then follows
from the polynomiality of the double ramification
cycle.

\subsection{Holomorphic anomaly equation} \label{Subsection_intro_HAE}
Let $\iota : \Mbar_{g-1, n+2} \to \Mbar_{g,n}$ be the
gluing map along the last two marked points, and
for any $g = g_1 + g_2$
and $\{ 1 , \ldots, n \} = S_1 \sqcup S_2$
let
\[
j : \Mbar_{g_1, S_1 \sqcup \{ \bullet \}}
\times \Mbar_{g_2, S_2 \sqcup \{ \bullet \}} \to
\Mbar_{g, n}
\]
be the map which glues the points marked by $\bullet$,
where
$\Mbar_{g_i, S_i}$ is the moduli space of stable
curves with markings in the set $S_i$.

\begin{thm} \label{thm_E_HAE} Considering $\CC_g( \gamma_1, \ldots, \gamma_n )$
as a polynomial in $C_2, C_4, C_6$ with coefficients in $H^\ast(\Mbar_{g,n})$,
we have
\begin{align*}
\frac{d}{dC_2} \CC_g( \gamma_1, \ldots, \gamma_n )
\ =\  & 
\iota_{\ast} \CC_{g-1}( \gamma_1, \ldots, \gamma_n, \1 , \1 )
\\
& + \sum_{\substack{ g= g_1 + g_2 \\ \{1,\ldots, n\} = S_1 \sqcup S_2}}
j_{\ast} \left( \CC_{g_1}( \gamma_{S_1}, \1 ) \boxtimes
\CC_{g_2}( \gamma_{S_2}, \1 ) \right) \\
& - 2 \sum_{i=1}^{n} \left( \int_{E} \gamma_i \right) \psi_i \cdot \CC_g( \gamma_1, \ldots, \gamma_{i-1}, \1, \gamma_{i+1}, \ldots, \gamma_n ),
\end{align*}
where $\gamma_{S_i} = ( \gamma_k )_{k \in S_i}$
and $\1 \in H^\ast(E)$ is the unit.
\end{thm}
\vspace{7pt}

Theorem~\ref{thm_E_HAE} measures the dependence of the
modular completion \cite{KZ} of $\CC_g( \ldots )$ 
on the non-holomorphic parameter and is therefore called a \emph{holomorphic anomaly equation}.
Practically it determines the quasimodular form
from lower weight data
up to a purely modular part (involving only $C_4$ and $C_6$) which depends on strictly less parameters.
This will be used in the proof of the Igusa cusp form conjecture in Section~\ref{Section_Igusa_cusp_form_conjecture}.


Milanov, Ruan and Shen have proven a holomorphic anomaly equation
for some elliptic orbifold $\p^1$s (i.e. stack quotients of an elliptic curve
by a non-trivial finite group).
The elliptic curve case was left as a
conjecture in \cite{MRS} and is proven by Theorem~\ref{thm_E_HAE}.
For elliptic orbifold $\p^1$s
the genus~$0$ Gromov--Witten theory
is generically semisimple,
and the holomorphic anomaly equation 
is deduced by Teleman's higher genus
reconstruction theorem.
For the elliptic curve the genus $0$
theory is trivial and this approach breaks down.
Instead our proof
relies on a careful analysis
of the appearance of $C_2$
in the degeneration formula for $\CC_g$
and properties of the double ramification cycle.


The ring $\QMod$ is graded by the weight
of its generators
\[ \QMod = \bigoplus_{k \geq 0} \QMod_{k} \,. \]
In particular, each graded summand
$\QMod_k$ is a finite-dimensional vector space
and knowing the weight of a
quasimodular form yields strong constraints
on its Fourier coefficients.
One immediate consequence of Theorem~\ref{thm_E_HAE} is
the following refinement of Theorem \ref{thm_E_quasimodularity}
by weight.

For any homogeneous $\gamma \in H^{\ast}(E)$ let $\deg_{\BR}( \gamma )$
denote its real cohomological degree.
Hence $\gamma \in H^{\deg_{\BR}(\gamma)}(E)$.

\begin{cor}
\label{CorE}
Let $\gamma_1, \ldots, \gamma_n \in H^{\ast}(E)$ be homogeneous. Then
$\CC_g( \gamma_1, \ldots, \gamma_n )$
is a cycle-valued quasimodular form of weight
$2g - 2 + \sum_{i} \deg_{\BR}(\gamma_i)$,
\[
\CC_g( \gamma_1, \ldots, \gamma_n )
\in 
H^{\ast}(\Mbar_{g,n}) \otimes \QMod_{2g-2+\sum_i \deg_{\BR}(\gamma_i)}.
\]
\end{cor}
\vspace{3pt}

\subsection{Elliptic fibrations}
\label{Subsection_Elliptic_fibrations}
Let $X$ and $B$ be non-singular projective varieties and
consider an elliptic fibration
\[ \pi : X \to B, \]
a flat morphism with fibers connected curves of arithmetic genus $1$.
We assume $\pi$ has integral fibers and 
admits a section
\[ \iota : B \to X \,. \]

For every curve class $\beta \in H_2(X,\BZ)$
with $\pi_{\ast} \beta = \mathsf{k}$
the fibration $\pi$ induces a morphism
\[
\pi : \Mbar_{g,n}(X, \beta) \to \Mbar_{g,n}(B, \mathsf{k}).
\]
Define $\pi$-relative Gromov-Witten classes with insertions
$\gamma_1, \ldots, \gamma_n \in H^{\ast}(X)$,
\[
\CC^{\pi}_{g, \beta}(\gamma_1, \ldots, \gamma_n)
=
\pi_{\ast}\left(
[ \Mbar_{g,n}(X, \beta) ]^{\vir}
\prod_{i=1}^{n} \ev_i^{\ast}(\gamma_i) \right) \,
\in H_{\ast}( \Mbar_{g,n}(B, \kk)).
\]

Let $B_0 \in H^2(X)$ be the class of the section $\iota$ and
let $N_{\iota}$ be the normal bundle of $\iota$. We define the divisor class
\[ W = B_0 - \frac{1}{2} \pi^{\ast} c_1(N_{\iota}). \]
For every curve class $\mathsf{k} \in H_2(B, \BZ)$
we form the generating series
\[
\CC^{\pi}_{g, \kk}(\gamma_1, \ldots, \gamma_n)
=
\sum_{\pi_{\ast} \beta = \mathsf{k}}
q^{W \cdot \beta}
\CC^{\pi}_{g, \beta}(\gamma_1, \ldots, \gamma_n)
\]
where the sum runs over all curve classes
$\beta \in H_2(X,\BZ)$ with $\pi_{\ast} \beta = \mathsf{k}$.



\begin{conjecture} \label{Conj_Quasimodularity}
For any $\gamma_1, \ldots, \gamma_n \in H^{\ast}(X)$
and $\kk \in H_2(B,\BZ)$ we have
\[
\CC^{\pi}_{g, \kk}(\gamma_1, \ldots, \gamma_n)
\, \in \,
H_{\ast}(\Mbar_{g,n}(B, \kk)) \otimes \frac{1}{\Delta(q)^m}\QMod.
\]
where $m = -\frac{1}{2} c_1(N_{\iota}) \cdot \mathsf{k}$.
\end{conjecture}
\vspace{8pt}

A refinement of Conjecture~\ref{Conj_Quasimodularity}
by weight can be found in Appendix~\ref{Section_Elliptic_fibrations}. 

We conjecture a holomorphic anomaly equation.
Consider the diagram
\[
\begin{tikzcd}
\Mbar_{g,n}(B,\mathsf{k}) & M_{\Delta} \ar[swap]{l}{\iota} \ar{d} \ar{r} & \Mbar_{g-1, n+2}(B, \mathsf{k}) \ar{d}{\ev_{n+1} \times \ev_{n+2}} \\
& B \ar{r}{\Delta} & B \times B
\end{tikzcd}
\]
where $\Delta$ is the diagonal, $M_{\Delta}$ is the fiber product
and $\iota$ is the gluing map along the last two points.
Similarly, for every
splitting $g = g_1 + g_2$, $\{ 1, \ldots, n \} = S_1 \sqcup S_2$
and $\mathsf{k} = \mathsf{k_1} + \mathsf{k_2}$
consider
\[
\begin{tikzcd}
\Mbar_{g,n}(B,\mathsf{k}) & M_{\Delta, \kk_1, \kk_2} \ar{d} \ar[swap]{l}{j} \ar{r} & 
\Mbar_{g_1, S_1 \sqcup \{ \bullet \}}(B, \mathsf{k_1})
\times \Mbar_{g_2, S_2 \sqcup \{ \bullet \}}(B, \mathsf{k_2}) \ar{d}{\ev_{\bullet} \times \ev_{\bullet}} \\
& B \ar{r}{\Delta} & B \times B
\end{tikzcd}
\]
where $M_{\Delta, \kk_1, \kk_2}$ is the fiber product and
$j$ is the gluing map along the marked points labeled by $\bullet$.

\begin{conjecture} \label{Conj_HAE} On $\Mbar_{g,n}(B, \mathsf{k})$,
\begin{align*}
\frac{d}{dC_2} \CC^{\pi}_{g,\kk}( \gamma_1, \ldots, \gamma_n )
\ =\  & 
\iota_{\ast} \Delta^{!}
\CC^{\pi}_{g-1,\kk}( \gamma_1, \ldots, \gamma_n, \1, \1 ) \\
&
+ \sum_{\substack{ g= g_1 + g_2 \\
\{1,\ldots, n\} = S_1 \sqcup S_2 \\
\mathsf{k} = \mathsf{k}_1 + \mathsf{k}_2}}
j_{\ast} \Delta^{!} \left(
\CC^{\pi}_{g_1, \kk_1}( \gamma_{S_1}, \1) \boxtimes
\CC^{\pi}_{g_2, \mathsf{k}_2}( \gamma_{S_2}, \1 ) \right) \\
&
- 2 \sum_{i=1}^{n} 
\psi_i \cdot \CC^{\pi}_{g,\kk}( \gamma_1, \ldots, \gamma_{i-1},
 \pi^{\ast} \pi_{\ast} \gamma_i , \gamma_{i+1}, \ldots, \gamma_n ),
\end{align*}
where $\psi_i \in H^2(\Mbar_{g,n}(B,\kk))$
is the cotangent line class at the $i$-th marking.
\end{conjecture}
\vspace{8pt}

If $X$ is a Calabi--Yau threefold the moduli space
of stable maps is of virtual dimension $0$.
The degree of the $\pi$-relative
classes are the genus $g$ Gromov--Witten potentials
\[ F_{g,\kk}(q) = \int \CC^{\pi}_{g,\kk}(). \]
Conjecture~\ref{Conj_Quasimodularity} implies
\[ F_{g,\kk}(q) \in \frac{1}{\Delta(q)^{-\frac{1}{2} K_B \cdot \kk}} \QMod. \]
Conjecture~\ref{Conj_HAE} and a direct calculation yields
\begin{align*}
\frac{d}{dC_2} F_{g,\kk}
& =
\langle \kk + K_S, \kk \rangle F_{g-1, \kk}
+ \sum_{\substack{g=g_1+g_2 \\ \kk = \kk_1 + \kk_2}}
\langle \kk_1, \kk_2 \rangle F_{g_1, \kk_1} F_{g_2, \kk_2}
- \frac{\delta_{g2} \delta_{k0}}{240} e(X),
\end{align*}
where
$\langle \cdot , \cdot \rangle$ is the intersection pairing on
the surface $B$.
We recover the holomorphic anomaly equation
for Calabi--Yau threefolds predicted by
Bershadsky, Cecotti, Ooguri, and Vafa \cite{BCOV}
using mirror symmetry\footnote{See
\cite[Eqns.(3.8) and (3.9)]{KMW}
and \cite{AS} for a discussion in the elliptic case.}.
This example is further discussed in Appendix~\ref{Subsection_CY3example}.

The generating series
$\CC^{\pi}_{g, \kk}( \ldots )$
captures only a slice of the
full $\pi$-relative Gromov--Witten theory of $X$.
For example, there might be distinct curve classes
$\beta_1, \beta_2 \in H_2(X,\BZ)$ with
\[ \pi_{\ast} \beta_1 = \pi_{\ast} \beta_2
 \ \ \text{and} \ \ \langle W, \beta_1 \rangle = \langle W, \beta_2 \rangle, \]
and $\CC^{\pi}_{g, \kk}( \ldots )$ only remembers
the sum of their Gromov--Witten classes.
A holomorphic anomaly equation
for the full relative potentials
will be conjectured in \cite{OPix2}.
There we also prove that Conjectures A and B
hold for
the rational elliptic surface
after specialization to numerical Gromov--Witten invariants.
Here we state the following
Corollary of Theorems \ref{thm_E_quasimodularity} and~\ref{thm_E_HAE}
which follows from Behrend's product formula \cite{B2}.

\begin{cor} \label{Cor_Trivial_elliptic_fibration}
Conjectures \ref{Conj_Quasimodularity} and \ref{Conj_HAE}
hold if $X = B \times E$ and $\pi : X \to B$ is the projection onto the first factor.
\end{cor}

\subsection{K3 surfaces} \label{Subsection_intro_K3surfaces}
Let $S$ be a non-singular projective K3 surface
and let $\beta \in \Pic(S)$ be a non-zero curve class.
Since $S$ carries a holomorphic symplectic form
the virtual class on the moduli space of stable maps vanishes,
\[ [ \Mbar_{g,n}(S, \beta) ]^{\text{vir}} = 0. \]
A non-trivial Gromov--Witten theory of $S$ is
defined by the reduced virtual class \cite{GWNL, KT}
\[ [ \Mbar_{g,n}(S, \beta) ]^{\text{red}} \in A_{\ast}( \Mbar_{g,n}(S, \beta) ). \]

Let $\pi : S \to \p^1$
be an elliptic K3 surface with a section, and let
\[ B, F \in \Pic(S) \]
be the class of a section and a fiber of $\pi$ respectively.
By deformation invariance the Gromov--Witten theory of $S$
in the classes
\[ \beta_h = B + hF, \ h \geq 0 \]
determines the Gromov--Witten theory
of all K3 surfaces in primitive classes.

Given $\gamma_1, \ldots, \gamma_n \in H^\ast(S)$
define the generating series of $\pi$-relative classes
\begin{multline*}
\label{defn_Kg}
\CK_{g}(\gamma_1, \ldots, \gamma_n)
=
\sum_{h = 0}^{\infty} q^{h-1}
\pi_{\ast}\left(
[ \Mbar_{g,n}(S, \beta_h) ]^{\text{red}}
\prod_{i=1}^{n} \ev_i^{\ast}(\gamma_i) \right) \\
\in H_{\ast}( \Mbar_{g,n}(\p^1, 1) ) \otimes \BQ[[q]],
\end{multline*}
where
$\pi : \Mbar_{g,n}(S,\beta_h) \to \Mbar_{g,n}(\p^1, 1)$
is the induced morphism.

\vspace{3pt}
\begin{conjecture} \label{K3_Conjecture1}
$\CK_{g}(\gamma_1, \ldots, \gamma_n)
\in H_{\ast}( \Mbar_{g,n}(\p^1, 1) ) \otimes \frac{1}{\Delta(q)} \QMod$.
\end{conjecture}
\vspace{7pt}

Because we use reduced virtual classes,
the holomorphic anomaly equation of Conjecture~\ref{Conj_HAE}
does not apply to $\CK_g$ and needs to be modified.
We require two additional ingredients.
First, the virtual class
on the moduli space of degree $0$ maps plays a role.
Identifying
$\Mbar_{g,n}(S,0) = \Mbar_{g,n} \times S$
we have
\[
[ \Mbar_{g,n}(S,0) ]^{\text{vir}}
=
\begin{cases}
[\Mbar_{0,n} \times S] & \text{if } g= 0 \\
\mathrm{pr}_2^{\ast} c_2(S) \cap [\Mbar_{1,n}  \times S]
& \text{if } g= 1\\
0 & \text{if } g \geq 2,
\end{cases}
\]
where $\mathrm{pr}_2$ is the projection to the second factor.
We let
\[ \CK_g^{\vir}(\gamma_1, \ldots, \gamma_n)
=
\pi_{\ast}\left( [ \Mbar_{g,n}(S, 0) ]^\vir \prod_i \ev_i^{\ast}(\gamma_i) \right), \]
where $\pi : \Mbar_{g,n}(S,0) \to \Mbar_{g,n}(\p^1,0)$
is the induced map.

Second, let $V$ be the orthogonal complement to
$B, F$ in $H^2(S,\BQ)$ with respect to the
intersection pairing,
\[
H^2(S,\BZ) = \langle B,F \rangle \oplus V,
\]
and let $\Delta_{V} \in V \boxtimes V$
be its diagonal.
Define the endomorphism
\[ \sigma: H^{\ast}(S^2) \to H^{\ast}(S^2) \]
by the following assignments:
\[
\sigma(\gamma \boxtimes \gamma') = 0
\ \text{ whenever } \gamma \text{ or } \gamma' \text{ lie in }
H^0(S) \oplus \BQ F  \oplus H^4(S),
\]
and for all $\alpha, \alpha' \in V$,
\begin{alignat*}{2}
\sigma( B \boxtimes B ) & = \Delta_{V} , \quad \quad \quad
& \sigma( B \boxtimes \alpha ) & = - \alpha \boxtimes F, \\
\sigma( \alpha \boxtimes B ) & = - F \boxtimes \alpha,
& \sigma( \alpha, \alpha' ) & = \langle \alpha, \alpha' \rangle F \boxtimes F.
\end{alignat*}

Define the class
\begin{align*} \mathsf{T}_g(\gamma_1, \ldots, \gamma_n)
& =
\iota_{\ast} \Delta^{!} \CK_{g-1}( \gamma_1, \ldots, \gamma_n, \1, \1 ) \\
& + 2 \sum_{\substack{g= g_1 + g_2 \\ \{1,\ldots, n\} = S_1 \sqcup S_2}}
j_{\ast} \Delta^{!} \left( \CK_{g_1}( \gamma_{S_1}, \1 ) \boxtimes \CK^{\vir}_{g_2}( \gamma_{S_2}, \1 ) \right) \\
& - 2 \sum_{i=1}^{n} \psi_i \cdot \CK_g( \gamma_1, \ldots, \gamma_{i-1}, \pi^{\ast} \pi_{\ast} \gamma_i, \gamma_{i+1}, \ldots, \gamma_n ) \\
& + 20 \sum_{i=1}^{n} \langle \gamma_i, F \rangle \CK_{g}(\gamma_1, \ldots, \gamma_{i-1}, F, \gamma_{i+1}, \ldots, \gamma_n ) \\
& -2 \sum_{i < j} \CK_g(\gamma_1, \ldots, \underbrace{\sigma_{1}(\gamma_i,\gamma_j)}_{i^{\text{th}}},
\ldots, \underbrace{\sigma_{2}(\gamma_i,\gamma_j)}_{j^{\text{th}}}, \ldots, \gamma_n ).
\end{align*}

\begin{conjecture} \label{K3_Conjecture2}
\label{conj_K3HAE}
For any $\gamma_1, \ldots, \gamma_n \in H^{\ast}(S)$,
\begin{align*}
\frac{d}{dC_2} \CK_g( \gamma_1, \ldots, \gamma_n )
\ =\  & \mathsf{T}_g(\gamma_1, \ldots, \gamma_n).
\end{align*} 
\end{conjecture}
\vspace{7pt}

Let $p : \Mbar_{g,n}(\p^1,1) \to \Mbar_{g,n}$
be the forgetful map, and let
\[ R^{\ast}(\Mbar_{g,n}) \subset H^{\ast}(\Mbar_{g,n}) \]
be the tautological subring spanned by push-forwards of
products of $\psi$ and $\kappa$ classes on boundary strata \cite{FP13}.
By 
\cite[Prop.29]{MPT},
for any tautological class
$\alpha \in R^{\ast}(\Mbar_{g,n})$
we have
\begin{equation} \label{454353}
\int_{\Mbar_{g,n}(\p^1,1)} p^{\ast}(\alpha) \cap \CK_g(\gamma_1, \ldots, \gamma_n)
\, \in \frac{1}{\Delta(q)} \QMod.
\end{equation}
Hence Conjecture~\ref{K3_Conjecture1} holds after specialization
to numerical Gromov--Witten invariants, or \emph{numerically}.
Here we show Conjecture~\ref{K3_Conjecture2} holds numerically.
\begin{thm} \label{thm_K3HAE}
For any tautological class $\alpha \in R^{\ast}(\Mbar_{g,n})$,
\[
\frac{d}{dC_2} \int p^{\ast}(\alpha) \cap \CK_g(\gamma_1, \ldots, \gamma_n)
=
\int p^{\ast}(\alpha) \cap \mathsf{T}_g(\gamma_1, \ldots, \gamma_n).
\]
\end{thm}

\subsection{Comments}
1) In Conjecture~\ref{Conj_HAE}
we assumed that the elliptic fibration admits a section.
We expect quasimodularity and holomorphic anomaly equations
also for elliptic fibrations without a section,
with the modification that we use quasimodular forms for the congruence subgroup
\[ \Gamma(N) \subset \mathrm{SL}_2(\BZ), \]
where $N$ is the lowest degree of a multisection over the base.
This prediction is in agreement with computations
for elliptic Calabi--Yau threefolds \cite{AS, Gr2},
and also \cite{MRS} if we 
view an elliptic orbifold $\p^1$ as an elliptic fibration
over an orbifold point without a section\footnote{
The examples in \cite{AS} suggest that the congruence subgroup should be
$\Gamma_1(N)$ in general.
For elliptic orbifold $\p^1$s we have strictly $\Gamma(N)$ modular forms;
however this is not a counterexample since the target is an orbifold.
We leave determining the exact congruence subgroup for elliptic fibrations without a section
to a later date.}.

2) Conjecture~\ref{Conj_HAE} predicts a holomorphic anomaly equation
for $\pi$-relative Gromov--Witten classes
of (well-behaved) Calabi--Yau fibrations\footnote{A Calabi--Yau fibration is a flat connected morphism of
non--singular projective varieties whose
general fiber has trivial canonical class.}
of relative dimension~$1$.
It seems plausible that 
holomorphic anomaly equations
exist for Calabi--Yau fibrations of higher relative dimension,
and that the shape of the formula should be simular to Conjecture~\ref{Conj_HAE}.
Families of lattice polarized K3 surfaces is a particularly interesting case to study
and beyond fiber classes \cite{GWNL, PT2} not much is understood.
Another example is the equivariant theory
of local $\p^2$ which we may view as a local $\p^2$ fibration.
Here, \cite[Sec.8]{LP} proves a holomorphic anomaly equation (after a specialization of variables)
which exactly matches the shape of ours.

3) The virtual class on $\Mbar_{g,n}(E,d)$ can be defined as an algebraic cycle and yields a correspondence
between Chow groups. Hence it is natural to ask if
the Chow-valued generating series
\[ \CC_g(\alpha) = \sum_{d=0}^{\infty} \pi_{\ast}\left( [\Mbar_{g,n}(E,d)]^{\text{vir}} (\ev_1 \times \cdots \times \ev_n)^{\ast}(\alpha)\right) q^d
\]
lies in $A^{\ast}(\Mbar_{g,n}) \otimes \QMod$ for every algebraic cycle $\alpha \in A^{\ast}(E^n)$.

The methods used in the paper unfortunately do not provide any
answer even if $\alpha$ is the class of a point $(z_1, \ldots, z_n) \in E^n$.
The argument fails already in the first step -- 
finding a suitable degeneration of $E$ to a rational nodal curve.
If we work in Chow we require the degeneration to be over $\p^1$
and to admit $n$ sections that specialize to the points $z_i$.
However, if the $z_i$ are chosen to be linearly independent then
such degeneration yields
an elliptic surface over $\p^1$ with Mordell--Weil rank $\geq n$, hence
an elliptic curve over $\BC(t)$ of rank $\geq n$. It is an open question
whether those exist for $n \gg 0$.\footnote{We thank B.~Poonen for discussions on this point.}

4) The holomorphic anomaly equation for the elliptic curve (Theorem~\ref{thm_E_HAE}) can be interpreted in terms of Givental's $R$-matrix action on cohomological field theories as follows. By Theorem~\ref{thm_E_quasimodularity}, we can view $\CC_g$ as a CohFT with coefficients in the ring $\QMod$. Define
\[
\CC_g^{\text{mod}}(\gamma_1,\ldots,\gamma_n) \in H^{\ast}(\Mbar_{g,n}) \otimes \Mod
\]
to be the modular part of $\CC_g$ obtained by setting $C_2$ to zero.
Since the map $\QMod \to \Mod$ sending quasimodular forms to their modular parts is a ring homomorphism, $\CC_g^{\text{mod}}$ is also a CohFT (with coefficients in $\Mod\subset\QMod$). These two CohFTs are identical in genus $0$ since the genus $0$ theory of $E$ vanishes in positive degree, but Teleman's reconstruction theorem does not apply because they are not (generically) semisimple. Thus the two theories need not be related by an $R$-matrix. However, it turns out that they are: Theorem~\ref{thm_E_HAE} is equivalent to the statement
\[
\CC_g = R_E.\CC_g^{\text{mod}}
\]
for the $R$-matrix $R_E \in \End(H^*(E)) \otimes \QMod[[z]]$ defined by
\[
R_E(\gamma) = \gamma + 2C_2\left(\int_E \gamma\right)z\cdot\1.
\]

For an elliptic fibration $\pi: X\to B$, it should be possible to interpret Conjecture~\ref{Conj_HAE} as an $R$-matrix action (on an appropriate generalization of a CohFT that takes values in the moduli space of stable maps to $B$) in a similar way. In this case the $R$-matrix will be given by
\[
R_X(\gamma) = \gamma + 2C_2\pi^*\pi_*\gamma.
\]

\subsection{Plan of the paper}
In Section~\ref{Section_Elliptic_curves_Point_insertions}
we prove quasimodularity and the holomorphic anomaly equation for the elliptic curve
(Theorems~\ref{thm_E_quasimodularity} and~\ref{thm_E_HAE})
if all insertions are point classes.
In Section~\ref{Section_Generalcase}
we prove the general case and Corollary~\ref{CorE}.
In Section~\ref{Section_K3_surfaces} we prove the holomorphic anomaly equation for K3 surfaces numerically.
In Section~\ref{Section_Igusa_cusp_form_conjecture} we prove the Igusa cusp form conjecture.
In Appendix~\ref{Section_Graphs_and_quasimodularforms}
we study the constant term in the Fourier expansion of
certain multivariate elliptic functions appearing in Section~\ref{Section_Elliptic_curves_Point_insertions}.
In Appendix~\ref{Section_Elliptic_fibrations}
we give a refinement of Conjecture~\ref{Conj_Quasimodularity} by weight and we work out an example as evidence.

\subsection{Acknowledgements}
We would like to thank J.~Bryan, F.~Janda, D.~Maulik, R.~Pandharipande, J.~Shen
and Q.~Yin for useful discussions on curve counting on K3 surfaces and elliptic curves.
We are also very grateful to T.~Milanov, Y.~Ruan and Y.~Shen for
discussions about their paper \cite{MRS}.
We would also like to thank the anonymous referees for their comments.

The second author was supported by a fellowship from the Clay Mathematics Institute.

\section{Elliptic curves: Point insertions}
\label{Section_Elliptic_curves_Point_insertions}
\subsection{Overview}
Let $E$ be a non-singular elliptic curve and let
\[ \pt \in H^2(E) \]
be the class of a point. We write
$\pt^{\times n}$ for the $n$-tuple $(\pt, \ldots, \pt)$.
In this section we prove the following special cases of Theorems~\ref{thm_E_quasimodularity} and \ref{thm_E_HAE}.

\begin{thm} \label{thm_E_quasimodularity_point}
$\CC_g( \pt^{\times n} ) \in \QMod$ for every $n \geq 0$.
\end{thm}

\begin{thm} \label{thm_E_HAE_point} For every $n \geq 0$ we have
\begin{align}\label{eq:ppp:HAE}
\begin{split}
\frac{d}{dC_2} \CC_g(\pt^{\times n})
\ =\  & 
\iota_{\ast} \CC_{g-1}( \pt^{\times n}, \1 , \1 )
\\
& + \sum_{\substack{ g= g_1 + g_2 \\ \{1,\ldots, n\} = S_1 \sqcup S_2}}
j_{\ast} \left( \CC_{g_1}( \pt^{\times|S_1|}, \1 ) \boxtimes
\CC_{g_2}( \pt^{\times|S_2|}, \1 ) \right) \\
& - 2 \sum_{i=1}^{n} \psi_i \cdot p_i^*\CC_g( \pt^{\times n-1} ),
\end{split}
\end{align}
where $p_i:\Mbar_{g,n}\to\Mbar_{g,n-1}$ is the map forgetting the $i$th marked point.
\end{thm}
\vspace{4pt}

In Section~\ref{Subsection_Double_ramification_Cycle} we introduce
the double ramification cycles.
In Section~\ref{sec_graphs} we discuss a relationship between
certain graph sums and elliptic functions which is used later in the proof.
In Section~\ref{Subsection_Proof_of_QMod_for_points} we prove Theorem~\ref{thm_E_quasimodularity_point}
and in Section~\ref{Subsection_Proof_E_HAE_point} we prove Theorem~\ref{thm_E_HAE_point}.

\subsection{Double ramification cycles}
\label{Subsection_Double_ramification_Cycle}
Let
\[ A = (a_1, a_2, \ldots, a_n), \ a_i \in \BZ \]
be a vector satisfying $\sum_{i=1}^{n} a_i = 0$.
The $a_i$ are the parts of $A$.
Let $\mu$ be the partition
defined by the positive parts of $A$,
and let $\nu$ be the partition defined by the negatives
of the negative parts of $A$.
Let $I$ be the set of markings corresponding
to the $0$ parts of $A$.

Let $\Mbar_{g,I}(\p^1,\mu,\nu)^{\sim}$
be the moduli space of stable relative maps of
connected curves of genus $g$ to rubber with ramification
profiles $\mu,\nu$ over the points $0,\infty \in \p^1$ respectively.
The moduli space admits a forgetful morphism (but still remembering the relative markings)
\[ \pi : \Mbar_{g,I}(\p^1,\mu,\nu)^{\sim}
\to \Mbar_{g,n}. \]
The double ramification cycle is the push-forward
\[
\DR_g(A)
= \pi_{\ast} 
\left[ \Mbar_{g,I}(\p^1,\mu,\nu)^{\sim} \right]^{\text{vir}}
\ \in A^g(\Mbar_{g,n}).
\]


Consider the double ramification cycle
as a function of integer parameters $(a_1, \ldots, a_n)$
with $\sum_i a_i = 0$, taking values
in the Chow ring of $\Mbar_{g,n}$.
The following result is proven in \cite{JPPZ, PZ}
and forms the basis of our approach.

\begin{prop}[\cite{JPPZ, PZ}] \label{Polynomiality}
$\DR_g(A)$ is polynomial in the $a_i$, that is,
there exists a polynomial $\mathsf{P}_{g,n} \in A^g(\Mbar_{g,n})[x_1, \ldots, x_n]$
such that
\[ \DR_g(A) = \mathsf{P}_{g,n}(a_1, \ldots, a_n) \]
for all $(a_i)_i \in \BZ^n$ with $\sum_i a_i = 0$.
\end{prop}

Since $\DR_g(A)$ is an $S_n$-equivariant function of $A$, we can choose the polynomial $\mathsf{P}_{g,n}$ to be $S_n$-equivariant as well.

\subsection{Graph sums} \label{sec_graphs}
Let $\Gamma$ be a connected finite graph with $n$ vertices $v_1,\ldots,v_n$ and no loops. Let $H(\Gamma)$ be the set of half-edges of $\Gamma$.
If $h\in H(\Gamma)$, let $v(h)$ denote the vertex to which $h$ is attached.
A function 
\[ \w :H(\Gamma) \to \BZ\setminus\{0\} \]
is called \emph{balanced} if it satisfies the following conditions:
\begin{enumerate}
\item $\w(h) + \w(h') = 0$ for every edge $e = \{h, h'\}$,
\item $\sum_{v(h) = v} \w(h) = 0$ for every vertex $v$.
\end{enumerate}

Let $k:H(\Gamma)\to\BZ_{\ge 0}$ be an arbitrary function, and let $\sigma$ be a total ordering of the vertices of $\Gamma$.
We consider the $q$-series
\[
F(\Gamma,k,\sigma) = \sum_{\substack{\w:H(\Gamma)\to\BZ\setminus\{0\} \\ \text{balanced}}}
\prod_{\substack{e = \{h,h'\} \\ v(h) <_\sigma v(h')}}
\frac{\w(h)}{1-q^{\w(h)}}\w(h)^{k(h)}\w(h')^{k(h')},
\]
where the product is over all edges $e = \{h,h'\}$ such that vertex $v(h)$ 
appears earlier than $v(h')$ in the total ordering $\sigma$,
and every $(1-q^m)^{-1}$ factor is expanded in positive powers of $q$.
If $m < 0$ then
\[
\frac{1}{1-q^m} = -q^{-m} -q^{-2m} - \cdots
\]
has leading term $-q^{-m}$, not $1$. This behavior implies
that the sum defining $F$ converges, since a direct check shows
that there are only finitely many terms in the sum
with all values of $\w(h)$ bounded from below. 

We rewrite the series $F(\Gamma,k,\sigma)$ in terms of elliptic functions.
Let $p_v$ be a formal variable for every vertex $v$ in $\Gamma$ and write
\[ p_h = p_{v(h)} \]
for every half-edge $h$.
Then $F(\Gamma,k,\sigma)$ is the coefficient of $\prod_v p_v^0$ of the series
\[
\sum_{\w}
\prod_{\substack{e = \{h,h'\} \\ v(h) <_\sigma v(h')}}\frac{\w(h)}{1-q^{\w(h)}} \w(h)^{k(h)} \w(h')^{k(h')}p_h^{\w(h)}p_{h'}^{\w(h')},
\]
where the sum is over all $\w:H(\Gamma)\to\BZ\setminus\{0\}$ satisfying condition (1).
This series factors as
\begin{equation} \label{4sdkfsdf}
\prod_{\substack{e = \{h,h'\} \\ v(h) <_\sigma v(h')}} \sum_{a \in \BZ\setminus\{0\}}\frac{a}{1-q^a}a^{k(h)}(-a)^{k(h')}\left(\frac{p_h}{p_{h'}}\right)^a.
\end{equation}

Let $z \in \BC$ and $\tau \in \BH$ where $\BH$ is the upper half plane, and let $p=e^{2 \pi i z}$ and $q=e^{2 \pi i \tau}$.
The Weierstra{\ss} elliptic function $\wp(z)$ reads
\[ \wp(z) = \frac{1}{12} + \frac{p}{(1-p)^2} + \sum_{d \geq 1} \sum_{k | d} k (p^k - 2 + p^{-k}) q^{d} \]
when expanded in $p,q$-variables in the region $0 < |q| < |p| < 1$. Hence
\[ \wp(z) + 2 C_2(\tau) = \sum_{a \in \BZ \setminus 0} \frac{a p^a}{1-q^a}, \]
where we consider $C_2(q)$ as a function on $\BH$ via $q=e^{2 \pi i \tau}$.

Consider the operator of differentiation with respect to $z$,
\[ \partial_z = \frac{1}{2 \pi i} \frac{d}{dz} = p \frac{d}{dp}. \]
Let $z_v \in \BC$ be a variable for every vertex $v$ and set $p_v = e^{2 \pi i z_v}$.
We write $z_h = z_{v(h)}$ for every half-edge $h$. 
The individual factors in \eqref{4sdkfsdf} can then be rewritten as
\[
\sum_{a \in \BZ\setminus\{0\}}\frac{a}{1-q^a}a^{k(h)}(-a)^{k(h')}\left(\frac{p_h}{p_{h'}}\right)^a
=
\partial_{z_h}^{k(h)}\partial_{z_{h'}}^{k(h')}(\wp(z_h-z_{h'}) + 2C_2),
\]
where the right hand side is expanded in the region $U_\sigma \subset \BC^n$ defined by
\[ 0 < |q| < |p_h|/|p_{h'}| < 1 \]
whenever $v(h) <_{\sigma} v(h')$.
We conclude the following result.

\begin{prop}\label{prop:graph_to_coeff} Let $\Gamma,k,\sigma$ be as above. Then
\[
F(\Gamma,k,\sigma)
=
\left[
\prod_{\substack{e = \{h,h'\} \\ v(h) <_\sigma v(h')}}
\partial_{z_h}^{k(h)}\partial_{z_{h'}}^{k(h')}(\wp(z_h-z_{h'}) + 2C_2)\right]_{p^0,\sigma}
\]
where we let $\left[ \, \cdot \, \right]_{p^0, \sigma}$ denote
taking the coefficient of $\prod_v p_v^0$ in the expansion in the variables $p_v$ in the region $U_{\sigma}$.
\end{prop}

\subsection{Proof of Theorem~\ref{thm_E_quasimodularity_point}}
\label{Subsection_Proof_of_QMod_for_points}
Since $\CC_g() = 0$ the claim holds if $n=0$, so we may assume $n \geq 1$.
Let $P_1, \ldots, P_n$ be disjoint copies of $\p^1$,
and let
\[ 0, \infty \in P_i \]
be two distinct points on each copy.
Let $C_n$ be the curve obtained by gluing for every $i$
the point $0$ on $P_i$ to the point $\infty$ on $P_{i+1}$,
where the indexing is taken modulo $n$.
In particular, $C_1$ is a $\p^1$ glued to itself along two points.
The curve $C_n$ is called an $n$-cycle of $\p^1$s and its dual graph
is depicted in Figure~\ref{FigureCn} in case $n=5$.

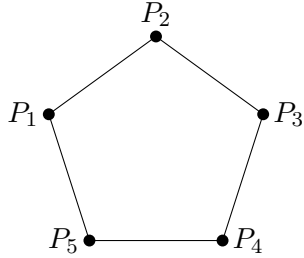
\begin{figure}
\centering
\begin{tikzpicture}[scale=1.5]
\draw[-] (0.95, 0.31) -- (0, 1.0) -- (-0.95, 0.31) -- (-0.59, - 0.81) -- (0.59, - 0.81) -- (0.95, 0.31);
\fill (-0.95,0.31)  circle[radius=1.5pt] node[left] {$P_1$};
\fill (0,1)  circle[radius=1.5pt] node[above] {$P_2$};
\fill (0.95,0.31)  circle[radius=1.5pt] node[right] {$P_3$};
\fill (0.59, -0.81)  circle[radius=1.5pt] node[right] {$P_4$};
\fill (-0.59, -0.81)  circle[radius=1.5pt] node[left] {$P_5$};
\end{tikzpicture}
\caption{The dual graph of an $n$-cycle in case $n=5$.} \label{FigureCn}
\end{figure}

Cconsider a degeneration of the elliptic curve $E$ to an $n$-cycle of $\p^1$s,
\[ E \rightsquigarrow C_n. \]
We apply the degeneration formula of \cite{Junli1, Junli2} to the class 
\[ \CC_{g,d}(\pt, \ldots, \pt) \in H^{\ast}(\Mbar_{g,n}) \]
where we choose the lift of the $i$-th point insertion $\pt \in H^2(E)$
to the total space of the degeneration such that its restriction
to the special fiber is the point class on the $i$-th copy of $\p^1$.
Hence after degeneration 
the $i$-th marking of the relative stable maps must lie on a component which maps to $P_i$.

For partitions
$\mu = (\mu_1, \ldots, \mu_{\ell(\mu)})$ and $\nu = (\nu_1, \ldots, \nu_{\ell(\nu)})$
of equal size let
\[ \Mbar_{g,n}(\p^1, \mu, \nu) \]
be the moduli space parametrizing
relative stable maps
from connected $n$-marked genus $g$ curves to $\p^1$
with (ordered)
ramification profile $\mu, \nu$ over the
relative points $0, \infty$ respectively.
If $2g-2+n+\ell(\mu)+\ell(\nu) > 0$, let 
\[ \pi : \Mbar_{g,n}(\p^1 , \mu, \nu) \to \Mbar_{g,n+\ell(\mu) + \ell(\nu)} \]
be the forgetful map (which remembers the relative markings).

\begin{lemma} If $2g-2+n+\ell(\mu)+\ell(\nu) > 0$, then
$\pi_{\ast} \big[ \Mbar_{g,n}(\p^1 , \mu, \nu) \big]^{\text{vir}} = 0$.
\end{lemma}
\begin{proof}
The $\BC^{\ast}$ action
on $\p^1$ which fixes the points $0, \infty \in \p^1$
induces a $\BC^{\ast}$-action on
$\Mbar_{g,n}(\p^1 , \mu, \nu)$.
The claim follows by virtual localization and a dimension computation.
\end{proof}

By the lemma we find that if a graph
is to contribute in the degeneration formula,
each stable vertex $v$ (those where $2g_v-2+n_v+\ell(\mu_v) + \ell(\nu_v) > 0$)
must contain a marked point. Hence there are at most $n$ stable vertices.
Since the $n$ marked points must map to $n$ different copies of $\p^1$ (by the incidence conditions)
and each lies on a stable vertex,
the graph therefore must have $n$ stable vertices containing a single marking each.

The contribution of each stable vertex is related
to the double ramification cycle by the following.
\begin{lemma} Let $\pt \in H^2(\p^1)$ be the point class. Then
\[ \pi_{\ast} \left( \left[ \Mbar_{g,1}(\p^1 , \mu, \nu) \right]^{\text{vir}} \ev_1^{\ast}(\pt) \right)
 = \DR_g(0, \mu_1, \ldots, \mu_{\ell(\mu)}, - \nu_1, \ldots, -\nu_{\ell(\nu)}).
\]
\end{lemma}
\begin{proof}
This follows from rigidification \cite[Sec.1.5.3]{MP}.
\end{proof}

At unstable vertices of the graph
we must have 
$g_v=n_v=0$ and $\mu_v = \nu_v = (d)$ for some $d$. The corresponding
moduli space $\Mbar_{0,0}(\p^1, (d), (d))$
is of virtual dimension $0$ and
parametrizes a map to $\p^1$ totally ramified at both ends.
We call such a component a \emph{tube}.
The contribution of a degree $d$ tube is
\begin{equation} \deg \, \left[ \Mbar_{0,0}(\p^1, (d), (d)) \right]^{\text{vir}} = \frac{1}{d}.
 \label{jrogwsrfg}
\end{equation}

Considering all possible contributions yields for all $d$ the formula
\begin{equation}
\CC_{g,d}(\pt^{\times n})
=
\sum_{\overline{\Gamma}} \frac{\prod_{h \colon \mathbf{w}(h)>0} \mathbf{w}(h)}{|\Aut( \overline{ \Gamma })|}
\xi_{\Gamma \ast} \left( \prod_{i=1}^{n} \DR_{g_i}\Big( ( \mathbf{w}(h) )_{h \in v_i} \Big) \right)
\label{goskogksf} \end{equation}
with the following notation.
The sum is over tuples $\overline{\Gamma} = (\Gamma, \mathbf{w}, \ell)$ where
\begin{itemize}
 \item $\Gamma$ is a connected stable graph\footnote{See \cite[Sec.0.3.2]{JPPZ} for the definition of a stable graph.}
 $\Gamma$ of genus $g$
 with exactly $n$ vertices $v_1, \ldots, v_n$, where each vertex $v_i$ has genus $g_i$ and exactly one leg with label $i$,
 \item $\mathbf{w} : H(\Gamma) \to \BZ$ is a weight function on the set of half-edges,
 \item $\ell : E(\Gamma) \to \BZ_{\geq 0}$ is a wrapping assignment on the set of edges,
\end{itemize}
satisfying the following conditions:
\begin{enumerate}
\item $\mathbf{w}(h) + \mathbf{w}(h') = 0$ for every edge $e = \{ h,h' \}$,
\item $\mathbf{w}(h) = 0$ if and only if $h$ is a leg,
\item $\sum_{h \in v} \mathbf{w}(h) = 0$.
\item For every edge $e = \{ h, h' \}$ with $\mathbf{w}(h) > 0$ and
$h \in v_i$ and $h' \in v_j$
let
\[ \mathbf{d}(e) =
\begin{cases}
 \mathbf{w}(h) ( \ell(e) + 1 ) & \text{ if } i \geq j \\
 \mathbf{w}(h) \ell(e) & \text{ if } i < j.
 \end{cases}
\]
Then $\sum_{e \in E(\Gamma)} \mathbf{d}(e) = d$.
\end{enumerate}
The group $\Aut( \overline{\Gamma} )$ is the automorphism group of the stable graph $\Gamma$
which preserves the decorations $\mathbf{w}$ and $\mathbf{\ell}$.
The morphism $\xi_{\Gamma}$ is the canonical gluing map
into the boundary stratum of $\Mbar_{g,n}$ determined by $\Gamma$.

We explain the graph data and the summands in \eqref{goskogksf}.
The vertex $v_i$ labels the unique stable component
which maps to $P_i$.
Every edge $e$ between vertices $v_i$ and $v_j$
corresponds to a chain of tubes
between the corresponding stable components
(the chain may have length $0$).
The tubes in the chain have a common degree $r$.
We set $\mathbf{w}(h) = r$ for the half-edge $h$ of $e$ which is glued to the
stable component over the point $0 \in P_i$.
For the opposite half-edge $h'$ 
we let $\mathbf{w}(h') = -r$.\footnote{
This corresponds to the following convention:
Assume the dual graph of the target $C_n$ is depicted in the plane
with labels increasing in clockwise direction as in Figure~\ref{FigureCn}.
Let $e = \{ h,h' \}$ be an edge with $\mathbf{w}(h) > 0$ and $h \in v_i$ and $h' \in v_j$.
Then the chain corresponding to $e$ 'travels' clockwise from $P_i$ to $P_j$ around
the cycle.}
We let $\ell(e)$ be the number of times the chain fully wraps
around the cycle.
If $e$ starts and ends at the same stable component
and traverses the cycle once we let $\ell(e)=0$.

We can read off the degree of the stable map at the intersection point
$x = P_1 \cap P_n$.
Let $e = \{ h,h' \}$ be an edge with $\mathbf{w}(h) > 0$ and $h \in v_i$ and $h' \in v_j$.
We may depict the chain corresponding to $e$ as leaving $P_i$
and traveling clockwise in Figure~\ref{FigureCn}.
If $i < j$ the chain crosses $x$ exactly $\ell(e)$ times
with ramification index $\mathbf{w}(h)$ each. It contributes therefore
$\mathbf{w}(h) \ell(e)$ to the degree of the stable map.
If $i \geq j$ the chain crosses $P$ exactly $(\ell(e)+1)$ times
with degree contribution $(\ell(e)+1) \mathbf{w}(h)$. Summing
up over all edges yields the degree condition (4).

We discuss the contributions coming from the kissing factors and the genus $0$ unstable components.
For every edge $e=\{h,h'\}$ with $\mathbf{w}(h)>0$ the corresponding chain of tubes
has $m$ components and $m+1$ gluing points (with itself or other components) for some $m$.
Each component contributes $1/\mathbf{w}(h)$ by \eqref{jrogwsrfg} and each gluing point contributes
the kissing factor $\mathbf{w}(h)$. The contibution from $e$ is therefore
\[ \frac{1}{\mathbf{w}(h)^m} \cdot \mathbf{w}(h)^{m+1} = \mathbf{w}(h). \]
The product over all these contributions yields $\prod_{h \colon \mathbf{w}(h)>0} \mathbf{w}(h)$.

We turn to the evaluation of \eqref{goskogksf}.
Forming a $q$-series over all $d$ yields
\begin{equation} \label{124235345}
\CC_{g}(\pt^{\times n})
=
\sum_{\Gamma} \frac{1}{|\Aut(\Gamma)|}
\xi_{\Gamma \ast}
\left( 
\sum_{\mathbf{w}} \prod_{e = \{h,h'\}} \frac{\mathbf{w}(h)}{1-q^{\mathbf{w}(h)}}
 \prod_{i=1}^{n} \DR_{g_i}\Big( ( \mathbf{w}(h) )_{h \in v_i} \Big) \right),
\end{equation}
where $(\Gamma, \mathbf{w})$ runs over the
same set as before (satisfying (1), (2) and (3)),
and the first product on the right side
is over the set of edges $e = \{h, h'\}$
where $h \in v_i$ and $h' \in v_j$ such that
\begin{itemize}
 \item $i < j$, or
 \item $i = j$ and $\mathbf{w}(h) < 0$.
\end{itemize}

By Proposition~\ref{Polynomiality} there exist classes
\[ \mathsf{C}_{g,k} \in A^{\ast}(\Mbar_{g,m}) \]
all vanishing except for finitely many $k = (k_1, \ldots, k_m) \in (\BZ_{\geq 0})^m$ and with $S_m$-equivariant dependence on $k$
such that
\begin{equation} \DR_g(a_1, \ldots, a_m) = \sum_{k = (k_1, \ldots, k_m)} \mathsf{C}_{g,k} a_1^{k_1} \cdots a_m^{k_m}.
\label{3513534}
\end{equation}
Plugging into \eqref{124235345} we obtain
\begin{equation}\label{eq:ppp1}
\CC_{g}(\pt^{\times n})
=
\sum_{\Gamma}
\sum_{k_v = (k_h)_{h \in v}}
\frac{\xi_{\Gamma \ast} \left( \prod_{v} \mathsf{C}_{g_v,k_v} \right)}{|\Aut(\Gamma)|}
\left( 
\sum_{\mathbf{w}} \prod_{e=\{h,h'\}} \frac{\mathbf{w}(h)}{1-q^{\mathbf{w}(h)}}
\cdot \prod_{h} \mathbf{w}(h)^{k_h}
\right),
\end{equation}
where the product over edges $e$ is as before and the
last product is over all half-edges $h$.

Suppose $e = \{h,h'\}$ is a loop and let $\widetilde{k}$ be defined by
$\widetilde{k}_h = k_{h'}$ and $\widetilde{k}_{h'} = k_h$, as well as $\widetilde{k}_{h''}=k_{h''}$ for
all all other half-edges $h''$.
Then by $S_n$-equivariance of the double ramification cycle we have
\[
\xi_{\Gamma \ast} \left( \prod_{v} \mathsf{C}_{g_v,k_v} \right) = \xi_{\Gamma \ast} \left( \prod_{v} \mathsf{C}_{g_v,\widetilde{k}_v} \right).
\]
If $k_h + k_{h'}$ is odd then the contribution to \eqref{eq:ppp1} of $k$ and $\bw$
cancels out with the contribution of $\widetilde{k}$ and $\widetilde{\bw}$, where $\widetilde{\bw}$
agrees with $\bw$ at all half-edges other than $h$ and $h'$ and has interchanged values there.
We conclude that we can restrict the sum over all $k = (k_v)$ to only those $k$ satisfying
\begin{equation}
k_h + k_{h'}\text{ is even for every loop $e = \{h,h'\}$}.
\end{equation}

Since the balancing conditions at vertices are independent of the weighting at loops,
we can factor the sum over $\mathbf{w}$ into a contribution from the loops and non-loops respectively.
The sum over loops further splits as a product over each individual loop,
with a loop $e = \{h,h'\}$ contributing the factor
\begin{align*}
L_{k_h, k_{h'}}(q)
& = 2\sum_{d < 0} \frac{d \cdot d^{k_h} \cdot (-d)^{k_{h'}}}{1-q^d} \\
& = 2 (-1)^{k_{h}} \sum_{d > 0} d^{k_h+k_{h'} + 1} \frac{q^d}{1-q^d}.
\end{align*}
In the notation of Section~\ref{sec_graphs} the non-loops contribute exactly
\[
F(\Gamma^{\text{no loops}},k,\sigma_0),
\]
where $\Gamma^{\text{no loops}}$ is the graph formed by deleting all the loops of $\Gamma$
and $\sigma_0$ is the vertex ordering defined by $v_i <_{\sigma_0} v_j \Leftrightarrow i < j$.

Hence we arrive at
\begin{equation} \label{eq:ppp2}
\CC_{g}(\pt^{\times n})
=
\sum_{\Gamma, k}
\frac{\xi_{\Gamma \ast} \left( \prod_{v} \mathsf{C}_{g_v,k_v} \right)}{|\Aut(\Gamma)|}
\Bigg( \prod_{\substack{\text{loops }\\e=\{h,h'\}}} L_{k_h, k_{h'}}(q)\Bigg)
\cdot F(\Gamma^{\text{no loops}},k,\sigma_0).
\end{equation}

For every $m \geq 0$ we have
\begin{equation} \label{4trfsdf}
\sum_{d > 0} d^{2 m + 1} \frac{q^{d}}{1-q^{d}}
= \frac{B_{2m+2}}{4 (m+1)} + \frac{(2m+2)!}{2} C_{2m+2}(q).
\end{equation}
Since we have already removed all terms with $k_h + k_{h'}$ odd from \eqref{eq:ppp2},
we conclude that the loop contribution $L_{k_h, k_{h'}}(q)$ is a quasimodular form.
The quasimodularity of $F(\Gamma^{\text{no loops}},k,\sigma_0)$ follows
from Proposition~\ref{prop:graph_to_coeff} and the first part of Theorem~\ref{THMAPP2} in Appendix~\ref{Section_Graphs_and_quasimodularforms}. This concludes the proof of Theorem~\ref{thm_E_quasimodularity_point}. \qed

\subsection{Proof of Theorem~\ref{thm_E_HAE_point}}
\label{Subsection_Proof_E_HAE_point}
We will begin the proof of \eqref{eq:ppp:HAE} on the left side using the formula \eqref{eq:ppp2}.
Since $\CC_{g}(\pt^{\times n})$ is $S_n$-invariant, we can average the formula above over all $n!$ vertex orderings to get
\begin{align*}
\CC_{g}(\pt^{\times n})
=
\sum_{\Gamma, k}
\frac{\xi_{\Gamma \ast} \left( \prod_{v} \mathsf{C}_{g_v,k_v} \right)}{|\Aut(\Gamma)|}
\Bigg( \prod_{\substack{\text{loops }\\e=\{h,h'\}}} L_{k_h, k_{h'}}(q)\Bigg)
\cdot
\frac{1}{n!}\sum_{\sigma}F(\Gamma^{\text{no loops}},k,\sigma).
\end{align*}
Using Proposition~\ref{prop:graph_to_coeff} we can rewrite this as
\begin{multline*}
\CC_{g}(\pt^{\times n})
=
\sum_{\Gamma, k}
\frac{\xi_{\Gamma \ast} \left( \prod_{v} \mathsf{C}_{g_v,k_v} \right)}{|\Aut(\Gamma)|} \\
\cdot \Bigg( \prod_{\substack{\text{loops }\\e=\{h,h'\}}} L_{k_h, k_{h'}}(q)\Bigg)
\left[\prod_{\substack{\text{non-loops } \\e=\{h,h'\}}}\partial_{z_h}^{k_h}\partial_{z_{h'}}^{k_{h'}}(\wp(z_h-z_{h'}) + 2C_2)\right]_{p^0},
\end{multline*}
where $[ \cdot ]_{p^0}$ is the coefficient
$[ \cdot ]_{p^0, \sigma}$ averaged over all orderings,
see \eqref{CSTCOEFFICIENTAVERAGED}.

We have already seen that the loop factor $L_{k_h, k_{h'}}(q)$ is quasimodular 
and by \eqref{4trfsdf} applying the $\frac{d}{dC_2}$
operator to it gives $2$ if $k_h = k_{h'} = 0$ and $0$ else. 
By Theorem~\ref{THMAPP2} the non-loop factor is
quasimodular and we have a formula for its $C_2$-derivative.
The sums over $\Gamma$ and $k$ are finite, so we can apply the $\frac{d}{dC_2}$ operator to the entire formula.
The result is a sum of three terms
\begin{align*}
& \frac{d}{dC_2}\CC_{g}(\pt^{\times n}) \\
& = \sum_{\Gamma,k_v}
\frac{\xi_{\Gamma \ast} \left( \prod_{v} \mathsf{C}_{g_v,k_v} \right)}{|\Aut(\Gamma)|}
\sum_{\substack{e_0=\{h_0,h'_0\}\text{ a loop}\\ \text{with }k_{h_0}=k_{h'_0}=0}}
2
\Bigg( \prod_{\substack{\text{other loops }\\e=\{h,h'\}}} L_{k_h, k_{h'}}(q)\Bigg)
\\
& \quad \quad \quad \quad \cdot
\left[\prod_{\substack{\text{non-loops } \\e=\{h,h'\}}}\partial_{z_h}^{k_h}\partial_{z_{h'}}^{k_{h'}}(\wp(z_h-z_{h'}) + 2C_2)\right]_{p^0}
\\
&+
\sum_{\Gamma,k_v}
\frac{\xi_{\Gamma \ast} \left( \prod_{v} \mathsf{C}_{g_v,k_v} \right)}{|\Aut(\Gamma)|}
\Bigg( \prod_{\substack{\text{loops }\\e=\{h,h'\}}} L_{k_h, k_{h'}}(q)\Bigg)
\\
& \quad \quad \quad \quad \cdot 
\sum_{\substack{e_0=(h_0,h'_0)\text{ a non-loop}\\ \text{with }k_{h_0}=k_{h'_0}=0}}
2
\left[\prod_{\substack{\text{other non-loops } \\e=\{h,h'\}}}\partial_{z_h}^{k_h}\partial_{z_{h'}}^{k_{h'}}(\wp(z_h-z_{h'}) + 2C_2)\right]_{p^0}
\\
&+
\sum_{\Gamma,k_v}
\frac{\xi_{\Gamma \ast} \left( \prod_{v} \mathsf{C}_{g_v,k_v} \right)}{|\Aut(\Gamma)|}
\Bigg( \prod_{\substack{\text{loops }\\e=\{h,h'\}}} L_{k_h, k_{h'}}(q)\Bigg) \cdot
\\
&(-1)\sum_{1 \le i \ne j \le n}
\left[(2\pi i)^2\underset{z_i = z_j}{\mathrm{Res}}\left((z_i-z_j)\prod_{\substack{\text{non-loops } \\e=\{h,h'\}}}\partial_{z_h}^{k_h}\partial_{z_{h'}}^{k_{h'}}(\wp(z_h-z_{h'}) + 2C_2)\right)\right]_{p^0}.
\end{align*}

The first two of these three terms agree with the first two of the three terms on the right of the holomorphic anomaly equation \eqref{eq:ppp:HAE} that we are trying to prove. To see this, commute the sum over $e_0$ out past the sum over $k_v$ in each of these terms. After doing so,
the conditions $k_{h_0}=k_{h'_0}=0$ are conditions on the $k_v$-sum.
Then taking $k_{h_0}=k_{h'_0}=0$ in the double ramification cycle coefficients $\mathsf{C}_{g_v,k_v}$
is the same thing as setting those parameters to be zero in the double ramification cycle
and thus is the same thing as computing the $\mathsf{C}_{g_v,k_v}$ for $\Gamma$ with edge $e_0$ deleted and then pulling back by forgetful maps.
In the case where $\Gamma$ is still connected after deleting the edge $e_0$, these contributions give precisely\footnote{The factor of $2$ in the first term above cancels with the factor of $2$ from the deleted loop's contribution to $\Aut(\Gamma)$.} the first term on the right of \eqref{eq:ppp:HAE}. For the second term of the above formula (deleting a non-loop) the graph might become disconnected after deleting edge $e_0$; this gives precisely the second term on the right of \eqref{eq:ppp:HAE}.

Thus it remains to show that the third term above agrees with the third term on the right of \eqref{eq:ppp:HAE}.
Removing a factor of $-1$, we want to show that
\begin{equation}
\begin{aligned}
\label{eq:ppp:HAE2}
&\sum_{\Gamma}
\sum_{k_v = (k_h)_{h\in v}}
\frac{\xi_{\Gamma \ast} \left( \prod_{v} \mathsf{C}_{g_v,k_v} \right)}{|\Aut(\Gamma)|}
\Bigg( \prod_{\substack{\text{loops }\\e=\{h,h'\}}} L_{k_h, k_{h'}}(q)\Bigg)
\\
&
\cdot\sum_{1 \le i \ne j \le n}
\left[(2\pi i)^2\underset{z_i = z_j}{\mathrm{Res}}\left((z_i-z_j)\prod_{\substack{\text{non-loops } \\e=\{h,h'\}}}\partial_{z_h}^{k_h}\partial_{z_{h'}}^{k_{h'}}(\wp(z_h-z_{h'}) + 2C_2)\right)\right]_{p^0}
\\
&= 2 \sum_{i=1}^{n} \psi_i \cdot p_i^*\CC_g( \pt^{\times n-1} ).
\end{aligned}
\end{equation}

We now move to the right side of \eqref{eq:ppp:HAE2}.
In this discussion $\Gamma'$ will always denote a graph with $n-1$ vertices and $\Gamma$ a graph with $n$ vertices.
We start with \eqref{124235345} with $n$ replaced by $n-1$:
\[
\CC_{g}(\pt^{\times n-1})
=
\sum_{\Gamma'} \frac{1}{|\Aut(\Gamma')|}
\xi_{\Gamma' \ast}
\left( 
\sum_{\bw} \prod_{e = \{h,h'\}} \frac{\bw(h)}{1-q^{\bw(h)}}
 \prod_{i=1}^{n-1} \DR_{g_i}\Big( ( \bw(h) )_{h \in v_i} \Big) \right),
\]
where the half-edges $h,h'$ satisfy $v(h) \le v(h')$ with respect to the vertex ordering
$v_a \le v_b$ for $1 \le a \le b \le n-1$ 
and if $v(h) = v(h')$ then $\bw(h)$ must be negative.
As before, we can average over all possible vertex orderings:
\begin{multline*}
\CC_{g}(\pt^{\times n-1})
=
\sum_{\Gamma'} \frac{1}{|\Aut(\Gamma')|} \\
\cdot \xi_{\Gamma' \ast}
\left( 
\frac{1}{(n-1)!}\sum_\sigma\sum_{\bw} \prod_{e = \{h,h'\}} \frac{\bw(h)}{1-q^{\bw(h)}}
 \prod_{i=1}^{n-1} \DR_{g_i}\Big( ( \bw(h) )_{h \in v_i} \Big) \right),
\end{multline*}
where $\sigma$ runs over all $(n-1)!$ orderings of the vertices and the sum over edges $e$ now uses $\sigma$ to choose which half-edge will be $h$.

We apply the pullback map $p_i^*$ for some $i\in\{1,\ldots,n\}$:
\begin{multline*}
p_i^*\CC_{g}(\pt^{\times n-1})
=
\sum_{\substack{1\le j \le n \\ j \ne i}}\sum_{\substack{\Gamma' \\ \text{legs $i,j$ on same vertex}}} \frac{1}{|\Aut(\Gamma')|}
\\
\cdot
\xi_{\Gamma' \ast}
\left( 
\frac{1}{(n-1)!}\sum_\sigma\sum_{\bw} \prod_{e = \{h,h'\}} \frac{\bw(h)}{1-q^{\bw(h)}}
 \prod_{\substack{1\le k\le n\\ k\ne i}} \DR_{g_k}\Big( ( \bw(h) )_{h \in v_k} \Big) \right),
\end{multline*}
where now $\Gamma'$ has $(n-1)$ vertices but $n$ legs and the legs $i,j$ are on the same vertex $v_i = v_j$ (of genus $g_i = g_j$).
Everything inside the first sum is symmetric in $i$ and $j$,
so after multiplying by $\psi_i$ and summing over $i$ we can write the entire formula more symmetrically as
\begin{multline}\label{eq:nminusone}
\sum_{i=1}^n\psi_i\cdot p_i^{\ast}\CC_{g}(\pt^{\times n-1})
=
\frac{1}{2}\sum_{1\le i \ne j \le n}(\psi_i+\psi_j)\sum_{\substack{\Gamma' \\ \text{legs $i,j$ on same vertex}}} \frac{1}{|\Aut(\Gamma')|}
\\
\cdot 
\xi_{\Gamma' \ast}
\left( 
\frac{1}{(n-1)!}\sum_\sigma\sum_{\bw} \prod_{e = \{h,h'\}} \frac{\bw(h)}{1-q^{\bw(h)}}
 \prod_{\substack{1\le k\le n\\ k\ne i}} \DR_{g_k}\Big( ( \bw(h) )_{h \in v_k} \Big) \right).
\end{multline}

We will need a formula for the product of a $\psi$ class with a double ramification cycle.
We use the following variant of the basic identity \cite[Cor.2.2]{BSSZ}.

\begin{lemma}\label{lemma_psi_times_DR}
For any $g\ge 0$ and $a_1,\ldots,a_n\in\BZ$ with sum zero, we have
\begin{multline*}
(\psi_1 + \psi_2)\DR_g(0,0,a_1,\ldots,a_n) = 
\\
\Bigg[ \sum_{m, g_i, S_i, c_i} \frac{c_1\cdots c_m}{m!}\DR_{g_1}(X,a_{S_1},-c_1,\ldots,-c_m) \boxtimes \DR_{g_2}(-X,a_{S_2},c_1,\ldots,c_m)\Bigg]_{X^1},
\end{multline*}
where $\boxtimes$ denotes gluing along the $m$ pairs of marked points with weights $\pm c_i$,
the bracket $[P]_{X^1}$ denotes taking the coefficient of $X^1$ in a polynomial function\footnote{Polynomiality here follows from the polynomiality of the double ramification cycle \cite{JPPZ,PZ}.} $P$ of $X$ (in this case defined for all sufficiently large integers $X$) and the sum $\sum_{m, g_i, S_i, c_i}$ signifies
\[
\sum_{1 \leq m \leq g+1} \
\sum_{\substack{ g = g_1 + g_2 + m - 1 \\ \{1,\ldots,n\} = S_1 \sqcup S_2 }} \ 
\sum_{ \substack{ c_1,\ldots,c_m > 0 \\ c_1 + \cdots + c_m = X + \sum_{i \in S_1}a_i }}.
\]
\end{lemma}

\begin{proof}
Use the basic identity \cite[Corollary 2.2]{BSSZ} to compute 
\[ (X\psi_1 - (-X) \psi_2)\DR_g(X,-X,a_1,\ldots,a_n) \]
and take the coefficient of $X^1$ of both sides.
\end{proof}

Using this lemma to expand the $(\psi_i+\psi_j)\DR_{g_i}$ factor in \eqref{eq:nminusone} effectively splits the vertex carrying legs $i$ and $j$ in $\Gamma'$ into two vertices each with one of the legs, connected by some positive number of edges.
We obtain
\begin{multline*}
\sum_{i=1}^n\psi_i\cdot p_i^*\CC_{g}(\pt^{\times n-1})
=
\frac{1}{2}\sum_{1\le i \ne j \le n}\sum_{\Gamma} \frac{1}{|\Aut(\Gamma)|}
\sum_{\substack{S\text{ a nonempty set of edges}\\ \text{between $v_i$ and $v_j$}}} \\
\cdot
\xi_{\Gamma \ast}
\Bigg( \frac{1}{(n-1)!}\sum_\sigma\sum_{\bw}
\prod_{\substack{e = \{h,h'\} \\ e \notin S}} \frac{\bw(h)}{1-q^{\bw(h)}} \\
\left[\sum_{\bc}
\left(\prod_{\substack{e = \{h,h'\}\in S \\ v(h)=v_i, v(h')=v_j}}\bc(h')\right)
\prod_{i=1}^{n} \DR_{g_i}\Big( ( (\bw \sqcup \bc)(h) )_{h \in v_i} \Big)\right]_{X^1} \Bigg).
\end{multline*}
Here $\Gamma$ is now a stable graph on $n$ vertices with one leg on each vertex. The set $S$ is a nonempty set of edges between $v_i$ and $v_j$ in $\Gamma$; these are the edges created by the $(\psi_i+\psi_j)\DR_{g_i}$ formula. To explain the later sums in the formula, let $\Gamma'$ be the graph formed from $\Gamma$ by contracting the edges of $S$, so $\Gamma'$ has $n-1$ vertices and legs $i$ and $j$ are on a single vertex.
In particular, edges between $v_i$ and $v_j$ that are not in $S$ become loops in $\Gamma'$. Then the remaining sums are over the following data:
\begin{itemize}
\item $\sigma$ is an ordering of the vertices of $\Gamma'$, and is used in the usual way to determine orientations $e=\{h,h'\}$;
\item $\bw$ is a balanced weighting of the non-leg half-edges of $\Gamma'$, which are naturally identified with the non-leg half-edges of $\Gamma$ that do not belong to edges in $S$;
\item $\bc$ is a weighting of the remaining half-edges of $\Gamma$ (those belonging to edges in $S$, and legs) such that $\bc(l_i) = X, \bc(l_j) = -X$ for some large integer variable $X$, $\bc$ vanishes on all the other legs, $\bc(h) < 0$ for any half-edge $h$ between $v_i$ and $v_j$ with $v(h) = v_i$, and $\bw\sqcup\bc$ forms a balanced weighting of $\Gamma$.
\end{itemize}

Using the polynomiality of the double ramification cycle the expression
inside $[ \, \cdot \, ]_{X^1}$ is polynomial in $X$ for fixed $\bw$ and sufficiently large $X\in\BZ$; we take the coefficient of $X^1$ in that polynomial, as in Lemma~\ref{lemma_psi_times_DR}.

Expanding the double ramification cycles as sums of monomials with coefficients $\mathsf{C}_{g_v,k_v}$, this becomes
\begin{multline*}
\sum_{i=1}^n\psi_i\cdot p_i^*\CC_{g}(\pt^{\times n-1}) \\
=
\frac{1}{2}\sum_{1\le i \ne j \le n}\sum_{\substack{\Gamma \\ \text{at least one edge} \\ \text{between $v_i$ and $v_j$}}}\sum_{k_v = (k_h)_{h\in v}} \frac{\xi_{\Gamma \ast}\left( \prod_{v} \mathsf{C}_{g_v,k_v} \right)}{|\Aut(\Gamma)|}\sum_{\substack{S\text{ a nonempty set of edges}\\ \text{between $v_i$ and $v_j$}}}
\\
\cdot
\frac{1}{(n-1)!}\sum_\sigma\sum_{\bw} \prod_{\substack{e = \{h,h'\} \\ e \notin S}} \frac{\bw(h)}{1-q^{\bw(h)}}\bw(h)^{k_h}\bw(h')^{k_{h'}}
\\
\left[\sum_{\bc}
\left(\prod_{\substack{e = \{h,h'\}\in S \\ v(h)=v_i, v(h')=v_j}}\bc(h')\bc(h)^{k_h}\bc(h')^{k_{h'}}\right)\right]_{X^1}.
\end{multline*}

We can multiply by $2$ and rearrange the sums slightly to make this look more
similar to the left side of \eqref{eq:ppp:HAE2}:
\begin{multline*}
2\sum_{i=1}^n\psi_i\cdot p_i^*\CC_{g}(\pt^{\times n-1}) \\
=
\sum_\Gamma\sum_{k_v = (k_h)_{h\in v}} \frac{\xi_{\Gamma \ast}\left( \prod_{v} \mathsf{C}_{g_v,k_v} \right)}{|\Aut(\Gamma)|}
\Bigg( \prod_{\substack{\text{loops in $\Gamma$}\\e=\{h,h'\}}} L_{k_h, k_{h'}}(q)\Bigg)
\\
\cdot
\sum_{1\le i \ne j \le n}\, 
\sum_{\substack{S\text{ a nonempty set of edges}\\ \text{between $v_i$ and $v_j$}}}
\frac{1}{(n-1)!}\sum_\sigma\sum_{\bw}
\prod_{\substack{\text{non-loops in $\Gamma$} \\ e = \{h,h'\}\notin S}} 
\frac{\bw(h)^{k_h+1}\bw(h')^{k_{h'}}}{1-q^{\bw(h)}} \\
\cdot \left[\sum_{\bc}
\left(\prod_{\substack{e = \{h,h'\}\in S \\ v(h)=v_i, v(h')=v_j}}\bc(h')\bc(h)^{k_h}\bc(h')^{k_{h'}}\right)\right]_{X^1}
\end{multline*}

Thus it is sufficient to show that
\begin{multline} \label{eq:ppp:HAE3}
\left[(2\pi i)^2\underset{z_i = z_j}{\mathrm{Res}}\left((z_i-z_j)\prod_{\substack{\text{non-loops } \\e=\{h,h'\}}}\partial_{z_h}^{k_h}\partial_{z_{h'}}^{k_{h'}}(\wp(z_h-z_{h'}) + 2C_2)\right)\right]_{p^0}
\\
=
\sum_{\substack{S\text{ a nonempty set of edges}\\ \text{between $v_i$ and $v_j$}}}
\frac{1}{(n-1)!}\sum_\sigma\sum_{\bw}
\prod_{\substack{\text{non-loops in $\Gamma$} \\ e = \{h,h'\}\notin S}} 
\frac{\bw(h)^{k_h+1}\bw(h')^{k_{h'}}}{1-q^{\bw(h)}} \\
\left[\sum_{\bc}
\left(\prod_{\substack{e = \{h,h'\}\in S \\ v(h)=v_i, v(h')=v_j}}\bc(h')\bc(h)^{k_h}\bc(h')^{k_{h'}}\right)\right]_{X^1}.
\end{multline}
We will expand both sides of \eqref{eq:ppp:HAE3} more explicitly and show they are equal.
On the left side we will expand the residue, while on the right side we will express the last line as a polynomial in the $\bw(h)$ and then apply Proposition~\ref{prop:graph_to_coeff} to interpret the right side as the $p^0$-coefficient of a multivariate elliptic function.

For simplicity, we will assume that $k_{h'} = 0$ for all $h'$ between $v_i$ and $v_j$ with $v(h') = v_j$. This reduction is justified because reducing $k_{h'}$ by one and increasing its partner $k_{h}$ by one just multiplies both sides by $-1$.

Let $e_1,\ldots,e_m$ be the edges in $\Gamma$ between $v_i$ and $v_j$. Let $e_a = (h_a,h'_a)$ with $v(h_a) = v_a$, and let $k_a = k_{h_a}$. On the right side of \eqref{eq:ppp:HAE3}, we will think of $S$ as a subset of $\{1,\ldots,m\}$. We write $c_a = \bc(h'_a)$ for $a\in S$ and $w_a = \bw(h_a)$ for $a\notin S$.

We first compute the residue on the left side of \eqref{eq:ppp:HAE3}. The only terms in the product with a pole along $z_i = z_j$ are $\prod_a\partial_{z_i}^{k_a}(\wp(z_i-z_j)+2C_2)$, so
using \eqref{RFEREas} and setting $w = 2\pi i z$ the residue is equal to
\begin{multline*}
\sum_{l \ge 0}\left[\prod_{a=1}^m\partial_z^{k_a}(\wp(z)+2C_2)\right]_{w^{-2-l}}\\
\cdot \left(\frac{\partial_{z_i}^l}{l!}\prod_{\substack{\text{non-loops in $\Gamma'$} \\e=\{h,h'\}}}\partial_{z_h}^{k_h}\partial_{z_{h'}}^{k_{h'}}(\wp(z_h-z_{h'}) + 2C_2)\right)\Bigg\vert_{z_i=z_j}.
\end{multline*}

We expand the first product. We start with the $w$-expansion
\[
\wp(z) + 2C_2 = \frac{1}{w^2} + \sum_{r\ge 0} (2r+1)(2r+2)C_{2r+2}w^{2r},
\]
so
\[
\partial_z^k(\wp(z)+2C_2) = \frac{(-1)^k(k+1)!}{w^{k+2}} + \sum_{r\ge \frac{k}{2}}\frac{(2r+2)!}{(2r-k)!}C_{2r+2}w^{2r-k}.
\]
Also, it will be convenient to use the notation
\[
D_{2k+2} = 2\sum_{d > 0}\frac{d^{2k+1}q^d}{1-q^d},
\]
so
\[
(2k+2)!\cdot C_{2k+2} = D_{2k+2} + \zeta(-1-2k)
\]
and
\begin{multline*}
\partial_z^k(\wp(z)+2C_2) \\
= \frac{(-1)^k(k+1)!}{w^{k+2}} + \sum_{r\ge \frac{k}{2}}D_{2r+2}\frac{w^{2r-k}}{(2r-k)!} + \sum_{r\ge \frac{k}{2}}\zeta(-1-2r)\frac{w^{2r-k}}{(2r-k)!}.
\end{multline*}

The residue is then equal to
\begin{multline}
\label{eq:ppp:residue}
\sum_{\substack{\{1,\ldots,m\} = A\sqcup B\sqcup C\\ r_a\in\BZ, r_a\ge \frac{k_a}{2}\text{ for }a\in B\sqcup C\\l\ge 0}} \prod_{a\in A}(k_a+1)!\prod_{a\in B}\frac{D_{2r_a+2}}{(2r_a-k_a)!}\prod_{a\in C}\frac{\zeta(-1-2r_a)}{(2r_a-k_a)!}
\\
\cdot \left(\frac{(-1)^l\partial_{z_i}^l}{l!}\prod_{\substack{\text{non-loops in $\Gamma'$} \\e=\{h,h'\}}}\partial_{z_h}^{k_h}\partial_{z_{h'}}^{k_{h'}}(\wp(z_h-z_{h'}) + 2C_2)\right)\Bigg\vert_{z_i=z_j},
\end{multline}
where $l$ is defined by
\[
l = -2 + \sum_{a\in A} (k_a+2) + \sum_{a\in B\sqcup C}(k_a-2r_a)
\]
and the constraint $l\ge 0$ in the sum should be viewed as a constraint on the variables used to define $l$.

We now switch to the right side of \eqref{eq:ppp:HAE3} and show that it is equal to the $p^0$-coefficient of \eqref{eq:ppp:residue}.
We will need the following combinatorial identity (a multivariate version of Euler-Maclaurin summation):

\begin{prop}\label{prop:multiEM}
Let $m$ and $X$ be positive integers. Let $P(x_1,\ldots,x_m)$ be a polynomial in $m$ variables. Then
\begin{multline*}
\sum_{\substack{x_1,\ldots,x_m \in\BZ_{>0} \\ x_1+\cdots+x_m = X}} P(x_1,\ldots,x_m) \\
= \sum_{\substack{I \subseteq \{1,\ldots,m\} \\ I \ne \emptyset}}\, 
\left[
\underset{\substack{x_i \ge 0\text{ for } i\in I\\ \sum_{i\in I}x_i = X - \sum_{i\notin I}x_i}}{\int}P(x_1,\ldots,x_m)
\right]
\Bigg|_{\substack{x_i^k\mapsto\zeta(-k)\\ \text{for $i\notin I$}}},
\end{multline*}
where the integral is over a $(|I|-1)$-dimensional simplex in the variables $(x_i)_{i\in I}$ and if $i_1 < \ldots < i_l$ are the elements of $I$ then we use the convention
\[
\underset{\substack{x_i \ge 0\text{ for } i\in I\\ \sum_{i\in I}x_i = X - \sum_{i\notin I}x_i}}{\int}P := \underset{\substack{x_i \ge 0\text{ for } i\in I\\ \sum_{i\in I}x_i = X - \sum_{i\notin I}x_i}}{\int}P\, dx_{i_1}\cdots dx_{i_{l-1}}.
\]
The value of the integral in the region $\sum_{i\notin I} x_i \le X$ is a polynomial in the variables $(x_i)_{i\notin I}$, and we extract a number by replacing each $x_i^k$ by the negative zeta value $\zeta(-k)$.
\end{prop}
\begin{proof}
When $m=1$ this just says that $P(X) = P(X)$. Assume $m\ge 2$. We may also assume that $P(x_1,\ldots,x_m) = x_1^{a_1}\cdots x_m^{a_m}$ is a monomial. Then the integral on the right side is a beta integral and evaluates as
\begin{align*}
& \int_{\substack{x_i \ge 0\text{ for } i\in I\\ \sum_{i\in I}x_i = X - \sum_{i\notin I}x_i}}P(x_1,\ldots,x_m) \\
=\ &
\left(\prod_{i\notin I}x_i^{a_i}\right)\frac{\prod_{i\in I}a_i!}{(\sum_{i\in I}(a_i+1)-1)!}(X-\sum_{i\notin I}x_i)^{\sum_{i\in I}(a_i+1)-1}
\\
=\ &
\prod_{i\in I}a_i! \sum_{\substack{n\in\BZ_{\ge 0}\\ b_i\in\BZ_{\ge 0}\text{ for $i\notin I$}\\ n+\sum_{i\notin I}b_i = \sum_{i\in I}(a_i+1)-1}}\frac{X^n}{n!}\prod_{i\notin I}\frac{(-1)^{b_i}x_i^{a_i+b_i}}{b_i!}.
\end{align*}
Replacing powers $x_i^k$ by the negative zeta values
\[
\zeta(-k) = (-1)^k\frac{B_{k+1}}{k+1}
\]
(where $B_k$ are the Bernoulli numbers) then gives
\[
\prod_{i\in I}a_i! \prod_{i\notin I}(-1)^{a_i}\sum_{\substack{n\in\BZ_{\ge 0}\\ b_i\in\BZ_{\ge 0}\text{ for $i\notin I$}\\ n+\sum_{i\notin I}b_i = \sum_{i\in I}(a_i+1)-1}}
\frac{X^n}{n!}\prod_{i\notin I}\frac{B_{a_i+b_i+1}}{b_i!\cdot(a_i+b_i+1)}.
\]
Using the generating function
\[
\sum_{b\ge 0}\frac{B_{a+b+1}}{b!\cdot(a+b+1)}z^b = \left(\frac{d}{dz}\right)^a\left(\frac{1}{e^z-1}-\frac{1}{z}\right),
\]
we can rewrite this as
\begin{align*}
&\prod_{i\in I}a_i! \prod_{i\notin I}(-1)^{a_i}\left[
e^{Xz}\prod_{i\notin I}\left(\frac{d}{dz}\right)^{a_i}\left(\frac{1}{e^z-1}-\frac{1}{z}\right)
\right]_{z^{\sum_{i\in I}(a_i+1)-1}}
\\
&=
\prod_{i=1}^m(-1)^{a_i}\left[
e^{Xz}\prod_{i\notin I}\left(\frac{d}{dz}\right)^{a_i}\left(\frac{1}{e^z-1}-\frac{1}{z}\right)\prod_{i\in I}\left(\frac{d}{dz}\right)^{a_i}\left(\frac{1}{z}\right)
\right]_{z^{-1}}.
\end{align*}
If $I=\emptyset$ this expression is $0$, since in this case the expression inside $[]_{z^{-1}}$ has no pole at $z=0$. Hence we can safely enlarge our sum over $I\subseteq \{1,\ldots,m\}$ to include an empty $I$. The result is that the right side of the identity to be proved is equal to
\[
\left[e^{Xz}\prod_{i=1}^m\left(-\frac{d}{dz}\right)^{a_i}\left(\frac{1}{e^z-1}\right)\right]_{z^{-1}}.
\]
If we interpret this as a residue at $z=0$ and change variables by $w = e^z-1$ we obtain
\begin{multline*}
\mathrm{Res}_{z = 0}e^{Xz}\prod_{i=1}^m\left(-\frac{d}{dz}\right)^{a_i}\left(\frac{1}{e^z-1}\right)dz
\\ = \mathrm{Res}_{w = 0}(w+1)^{X-1}\prod_{i=1}^m\left(-(w+1)\frac{d}{dw}\right)^{a_i}\left(\frac{1}{w}\right)dw.
\end{multline*}
This only has poles at $w=0$ and $w=\infty$, so the residue at $w=0$ is $-1$ times the residue at $w=\infty$. We then change variables by $w = \frac{1}{u}-1$ to get
\begin{multline*}
-\mathrm{Res}_{w = \infty}(w+1)^{X-1}\prod_{i=1}^m\left(-(w+1)\frac{d}{dw}\right)^{a_i}\left(\frac{1}{w}\right)dw \\=
\mathrm{Res}_{u = 0}\frac{1}{u^{X+1}}\prod_{i=1}^m\left(u\frac{d}{du}\right)^{a_i}\left(\frac{u}{1-u}\right)du.
\end{multline*}
This is equal to the coefficient of $u^X$ in
\[
\prod_{i=1}^m\left(u\frac{d}{du}\right)^{a_i}\left(\frac{u}{1-u}\right)
= \prod_{i=1}^m\left(\sum_{x_i\in\BZ_{>0}}x_i^{a_i}u^{x_i}\right),
\]
which is the sum we were trying to compute.
\end{proof}

The conditions on the sum over $\bc$ on the right side of \eqref{eq:ppp:HAE3} are that $(c_a)_{a\in S}$ are positive integers with fixed sum
\[
X + \sum_{a\notin S}w_a + \sum_{\substack{\text{$h$ not part of a loop in $\Gamma'$}\\ v(h) = v_i}}\bw(h).
\]
Hence we may apply Proposition~\ref{prop:multiEM} above. The result is
\begin{multline} \label{eq:ppp:long}
\sum_{S\subseteq\{1,\ldots,m\}}\sum_{I\subseteq S, I \ne\emptyset}\frac{1}{(n-1)!}\sum_\sigma\sum_{\bw\text{ balanced on $\Gamma'$}}
\\
\cdot \prod_{\substack{\text{non-loops in $\Gamma'$} \\ e = \{h,h'\}}}\frac{\bw(h)}{1-q^{\bw(h)}}\bw(h)^{k_h}\bw(h')^{k_{h'}}
\cdot
\prod_{a\notin S}\left(\sum_{w_a\in\BZ\setminus\{0\}}\frac{|w_a|q^{|w_a|}}{1-q^{|w_a|}}(-|w_a|)^{k_a}\right) \\
\cdot \left[\underset{\substack{x_a \ge 0\text{ for } a\in I\\ \sum_{a\in I}x_a = X + \sum_{a\notin S}w_a + W - \sum_{a\in S\setminus I}x_a}}{\int}\, \prod_{a\in S}(-1)^{k_a}x_a^{k_a+1}\right]_{X^1}\Bigg|_{\substack{x_a^k\mapsto\zeta(-k)\\ \text{for $a\in S\setminus I$}}},
\end{multline}
where
\[
W := \sum_{\substack{\text{$h$ not part of a loop in $\Gamma'$}\\ v(h) = v_i}}\bw(h).
\]
The integral is a beta integral and evaluates to
\begin{multline*}
\prod_{a\in S}(-1)^{k_a}\prod_{a\in S\setminus I}x_a^{k_a+1}
\frac{\prod_{a\in I}(k_a+1)!}{(-1+\sum_{a\in I}(k_a+2))!} \\
\cdot (X + \sum_{a\notin S}w_a + W - \sum_{a\in S\setminus I}x_a)^{-1+\sum_{a\in I}(k_a+2)},
\end{multline*}
which has $X^1$ coefficient 
\[
\sum_{\substack{r_a\in\BZ,r_a\ge\frac{k_a}{2}\text{ for }a\notin{I}\\ l\ge 0}}
\prod_{a\notin I}\frac{(-1)^{k_a}}{(2r_a-k_a)!}\prod_{a\in I}(k_a+1)!\prod_{a\in S\setminus I}x_a^{2r_a+1}\prod_{a\notin S}w_a^{2r_a-k_a}\frac{W^l}{l!},
\]
where $l$ is defined by
\[
l = -2 + \sum_{a\in I}(k_a+2) + \sum_{a\notin I}(k_a-2r_a).
\]

Substituting this into \eqref{eq:ppp:long} and setting $A = I$, $B = \{1,\ldots,m\}\setminus S$, and $C = S\setminus I$ gives
\begin{multline*}
\sum_{\substack{\{1,\ldots,m\} = A\sqcup B\sqcup C\\ r_a\in\BZ, r_a\ge \frac{k_a}{2}\text{ for }a\in B\sqcup C\\l\ge 0}} \prod_{a\in A}(k_a+1)!\prod_{a\in B}\frac{D_{2r_a+2}}{(2r_a-k_a)!}\prod_{a\in C}\frac{\zeta(-1-2r_a)}{(2r_a-k_a)!}
\\
\cdot \frac{1}{(n-1)!}\sum_\sigma\sum_{\bw\text{ balanced on $\Gamma'$}}\prod_{\substack{\text{non-loops in $\Gamma'$} \\ e = \{h,h'\}}}\frac{\bw(h)}{1-q^{\bw(h)}}\bw(h)^{k_h}\bw(h')^{k_{h'}} \frac{(-1)^lW^l}{l!}.
\end{multline*}
Applying Proposition~\ref{prop:graph_to_coeff} to the second line
gives the desired equality with the $p^0$-coefficient of the residue \eqref{eq:ppp:residue}.
This completes the proof of Theorem~\ref{thm_E_HAE_point}. \qed

\section{Elliptic curves: The general case}
\label{Section_Generalcase}
\subsection{Overview}
In Section~\ref{Section_Elliptic_curves_Point_insertions} we proved
the quasimodularity and holomorphic anomaly equation for
\[ \CC_g(\gamma_1, \ldots, \gamma_n) \in H^{\ast}(\Mbar_{g,n}) \otimes \BQ[[q]] \]
if all $\gamma_i$ are point classes.
Here we show the point case implies the general case by
showing quasimodularity and the holomorphic anomaly equation are preserved by the following operations:
\begin{itemize}
 \item Pull-back under the map $\Mbar_{g,n+1} \to \Mbar_{g,n}$ forgetting a point,
 \item Pull-back under the map $\Mbar_{g-1,n+2} \to \Mbar_{g,n}$ gluing two points,
 \item Monodromy invariance.
\end{itemize}
In Section~\ref{Subsection_ProofofCorE} we also present the proof of Corollary~\ref{CorE}.

\subsection{Cohomology}
Let $E$ be a non-singular elliptic curve and let
\[ \1, \aaa,\bbb, \pt \]
be a basis of $H^{\ast}(E)$ with the following properties:
\begin{enumerate}
\item[(a)] $\1 \in H^0(E)$ is the unit,
\item[(b)] $\aaa \in H^{1,0}(E)$ and $\bbb \in H^{0,1}(E)$ determine
a symplectic basis of $H^1(E)$,
\[ \int_E \aaa \cup \bbb = 1, \]
\item[(c)] $\pt \in H^2(E)$ is the class Poincar\'e dual to a point.
\end{enumerate}

\subsection{Monodromy invariance}
For any $\phi = \binom{\alpha\ \beta}{\gamma\ \delta}\in \mathrm{SL}_2(\BZ)$ there exists
a monodromy of the elliptic curve $E$
whose action on cohomology
\[ \phi : H^{\ast}(E) \to H^{\ast}(E) \]
is the identity on $H^0(E)$ and $H^2(E)$ and satisfies
\[
\phi : \begin{pmatrix} \aaa \\ \bbb \end{pmatrix} \mapsto \begin{pmatrix} \alpha & \beta \\ \gamma & \delta \end{pmatrix} \begin{pmatrix} \aaa \\ \bbb \end{pmatrix}.
\]
By deformation invariance it follows that
\begin{equation} \CC_g(\gamma_1, \ldots, \gamma_n) = \CC_g( \phi(\gamma_1), \ldots, \phi(\gamma_n) ) \label{MONODROMY} \end{equation}
where the right hand side is defined by multilinearity.

This implies the following balancing condition,
which can be found in \cite[Sec.4]{Jtaut}
and is proven by an adaption of \cite[Sec.5]{OP3}.
\begin{lemma}[Janda \cite{Jtaut}] \label{Balanced}
If $\gamma_1, \ldots, \gamma_n \in \{ \1, \aaa, \bbb, \pt \}$ are non-balanced, i.e. if
\[ | \{ i \colon \gamma_i = \aaa \} | \neq | \{ i \colon \gamma_i = \bbb \} |, \]
then $\CC_g(\gamma_1, \ldots, \gamma_n) = 0$.
\end{lemma}

\subsection{Proof of Theorem~\ref{thm_E_quasimodularity}}
\label{Subsection_Proof_of_Theorem_Quasimodularity}
For every $g$ and $n$ consider the homomorphism
\[ \CC_g : H^{\ast}(E)^{\otimes n} \to H^{\ast}(\Mbar_{g,n}) \otimes \BQ[[q]],
\ \gamma_1 \otimes \cdots \otimes \gamma_n \mapsto \CC_g(\gamma_1, \ldots, \gamma_n). \]
Define the subspace
\[ K_n \subset H^{\ast}(E)^{\otimes n} \]
to be the set of all $\gamma$ such that for all $g$ the series $\CC_g(\gamma)$
lies in $H^{\ast}(\Mbar_{g,n}) \otimes \QMod$.
We need to show that the inclusion
\[ K = \bigoplus_{n \geq 0} K_n \, \subset \, T(E) = \bigoplus_{n \geq 0} H^{\ast}(E)^{\otimes n}. \]
is an equality. 
Consider any element
\begin{equation} \gamma = \gamma_1 \otimes \cdots \otimes \gamma_n, \quad \quad
\gamma_1, \ldots, \gamma_n \in \{ \1, \aaa, \bbb, \pt \} \label{54524} \end{equation}
If $\gamma$ is non-balanced, then $\gamma \in K$ by Lemma~\ref{Balanced}.
Hence we may assume $\gamma$ is balanced.
We will show $\gamma \in K$ by induction on
\[ m(\gamma) = | \{ i \colon \gamma_i = \aaa \} |, \]
the number of factors in $\gamma$ equal to $\aaa$.

\vspace{7pt}
\noindent \textbf{Base.} If $m(\gamma)=0$, then every $\gamma_i$ is either the point class or the unit.
Since every $K_n$ is invariant under permutation we may assume
\[ \gamma = \underbrace{\pt \otimes \cdots \otimes \pt}_{k} \otimes \underbrace{1 \otimes \cdots \otimes 1}_{n-k} \]
for some $k$. 
If $2g-2+k \leq 0$ 
then $\CC_g(\gamma)$ is a constant in $q$ and hence quasimodular.
Otherwise we have
\[ p^{\ast} \CC_g(\pt, \ldots, \pt) = \CC_g(\pt, \ldots, \pt, \underbrace{\1, \ldots, \1}_{n-k}), \]
where $p : \Mbar_{g,n} \to \Mbar_{g,k}$ is the map that forgets the last $n-k$ points.
We conclude $\gamma \in K$ by
Theorem~\ref{thm_E_quasimodularity_point}.

\vspace{7pt}
\noindent \textbf{Induction.} Let $m \geq 1$ and
assume $\gamma \in K$ for any $\gamma$ with $m(\gamma) < m$.

\begin{lemma} \label{FDFSDFS} Let $\gamma$ be of the form \eqref{54524} with $m(\gamma) = m$
and let $\gamma_i = \aaa$ and $\gamma_j = \bbb$ for some $i<j$.
Then, under the induction hypothesis,
\[
\cdots \otimes
\underbrace{\aaa}_{i^\text{th}} \otimes \cdots \otimes
\underbrace{\bbb}_{j^\text{th}} \otimes \cdots
\, \ = \  \,
\cdots \otimes \underbrace{\bbb}_{i^\text{th}} \otimes \cdots \otimes \underbrace{\aaa}_{j^\text{th}} \otimes \cdots
\]
in $T(E) / K$, with all factors except the $i$-th and $j$-th the same on both sides.
\end{lemma}
\begin{proof}
Let $\iota : \Mbar_{g,n} \to \Mbar_{g+1, n-2}$ be the map that glues the $i$-th and $j$-th markings.
For every $g$ we have by induction
\[
\CC_{g+1}(\gamma_1, \ldots, \widehat{\gamma_i}, \ldots, \widehat{\gamma_j}, \ldots, \gamma_n )
\in H^{\ast}(\Mbar_{g+1, n-2}) \otimes \QMod,
\]
where $\widehat{\gamma_i}$ and $\widehat{\gamma_j}$ denotes that we omitted the $i$-th and $j$-th entry.
Hence also
\begin{equation}
\begin{aligned}
\label{54624524}
\iota^{\ast} \CC_{g+1}(\gamma_1, \ldots, \widehat{\gamma_i}, \ldots, \widehat{\gamma_j}, \ldots, \gamma_n )
= & \ \, \phantom{+} \CC_{g}(\gamma_1, \ldots, \1, \ldots, \pt, \ldots, \gamma_n ) \\
& + \CC_{g}(\gamma_1, \ldots, \pt, \ldots, \1, \ldots, \gamma_n ) \\
& - \epsilon \cdot \CC_{g}(\gamma_1, \ldots, \aaa, \ldots, \bbb, \ldots, \gamma_n ) \\
& + \epsilon \cdot \CC_{g}(\gamma_1, \ldots, \bbb, \ldots, \aaa, \ldots, \gamma_n )
\end{aligned}
\end{equation}
lies in $H^{\ast}(\Mbar_{g,n}) \otimes \QMod$,
where $\epsilon = \prod_{i<k<j} (-1)^{\deg_{\BR}(\gamma_k)}$ and we used the diagonal splitting
\[ \Delta_E = \1 \otimes \pt + \pt \otimes \1 - \aaa \otimes \bbb + \bbb \otimes \aaa. \]
By induction the first two terms on the right hand side in \eqref{54624524}
lie in $H^{\ast}(\Mbar_{g,n}) \otimes \QMod$, hence so does the difference
\[
- \CC_{g}(\gamma_1, \ldots, \aaa, \ldots, \bbb, \ldots, \gamma_n )
+ \CC_{g}(\gamma_1, \ldots, \bbb, \ldots, \aaa, \ldots, \gamma_n ).
\qedhere
\]
\end{proof}

Let $\gamma$ be of the form \eqref{54524} with $m(\gamma) = m$. We show $\gamma \in K_n$ and complete the induction step.
By $S_n$ invariance of $K_n$ we may assume
\[ \gamma = \underbrace{\aaa \otimes \cdots \otimes \aaa}_{m} \otimes \underbrace{\bbb \otimes \cdots \otimes \bbb}_{m}
\otimes \gamma_{2m+1} \otimes \cdots \otimes \gamma_n, \]
where $\gamma_i$ are even for all $i>2m$.
Consider the element
\[ \gamma' = \underbrace{\aaa \otimes \cdots \otimes \aaa}_{2m} 
\otimes \gamma_{2m+1} \otimes \cdots \otimes \gamma_n. \]
Since $\gamma'$ is non-balanced it lies in $K$ by Lemma~\ref{Balanced}.
Hence by monodromy invariance
with respect to $\phi = \binom{1\ 0}{1\ 1}$ also
\begin{equation} \label{34150999}
\phi(\gamma') =
\underbrace{(\aaa+\bbb) \otimes \cdots \otimes (\aaa + \bbb)}_{2m}
\otimes \gamma_{2m+1} \otimes \cdots \otimes \gamma_n
\end{equation}
lies in $K$. But by Lemmas~\ref{Balanced} and \ref{FDFSDFS} we have
\[ \phi(\gamma') = \binom{2m}{m} \gamma + \, \ldots \]
where $( \ldots )$ stands for elements in $K$.
It follows that $\gamma \in K$. \qed

\subsection{Holomorphic anomaly equation}
Define
\[ \mathsf{H}_g(\gamma_1, \ldots, \gamma_n) \in H^{\ast}(\Mbar_{g,n}) \otimes \BQ[[q]] \]
to be the right hand side in the equality of Theorem~\ref{thm_E_HAE}, and let
\[ \mathsf{H}_g : H^{\ast}(E)^{\otimes n} \to H^{\ast}(\Mbar_{g,n}) \otimes \BQ[[q]],
\ \gamma_1 \otimes \cdots \otimes \gamma_n \mapsto \mathsf{H}_g(\gamma_1, \ldots, \gamma_n) \]
be the induced homomorphism.
We show several compatibilities of $\mathsf{H}_g$.

\begin{lemma} \label{545i2345}
Let $p : \Mbar_{g,n+1} \to \Mbar_{g,n}$ be the map that forgets the $(n+1)$-th marked point.
Then for any $g$ and $\gamma_1, \ldots, \gamma_n$ we have
\[ \iota^{\ast} \mathsf{H}_g(\gamma_1, \ldots, \gamma_n)
 = \mathsf{H}_g(\gamma_1, \ldots, \gamma_n, \1).
\]
\end{lemma}
\begin{proof}
This follows by a direct calculation from the following.
For any $g$ and $\gamma_1, \ldots, \gamma_n$ we have
\[ p^{\ast} \CC_g(\gamma_1, \ldots, \gamma_n) = \CC_g(\gamma_1, \ldots, \gamma_n, \1). \]
And for every $i \in \{ 1, \ldots, n\}$ the cotangent classes $\psi_i \in H^2(\Mbar_{g,n})$ and
$\psi_i \in H^2(\Mbar_{g,n+1})$ are related by
\begin{equation} \psi_i = p^{\ast} \psi_i + D_{(i,n+1)}, \label{4514} \end{equation}
where $D_{(i,n+1)}$ is the boundary divisor whose generic point parametrizes
the union of a genus~$0$ curve carrying the markings $i$ and $n+1$, and
a genus $g$ curve carrying the remaining markings.
\end{proof}

\begin{lemma} \label{Lemma02292}
$\mathsf{H}_g(\gamma_1, \ldots, \gamma_n) = \mathsf{H}_g( \phi(\gamma_1), \ldots, \phi(\gamma_n) )$
for every $\phi \in \mathrm{SL}_2(\BZ)$.
\end{lemma}
\begin{proof}
This follows from the monodromy invariance \eqref{MONODROMY}.
\end{proof}

\begin{lemma} \label{33089834}
Let $\iota : \Mbar_{g-1, n+2} \to \Mbar_{g,n}$ be the
gluing map along the last two marked points. Then
\[ \iota^{\ast} \mathsf{H}_g(\gamma_1, \ldots, \gamma_n)
=
\mathsf{H}_{g-1}(\gamma_1, \ldots, \gamma_n, \Delta_E). \]
\end{lemma}
\begin{proof}
This follows from a direct calculation
of $\iota^{\ast} \mathsf{H}_g(\gamma_1, \ldots, \gamma_n)$
using
the description of the
intersection of boundary strata in $\Mbar_{g,n}$ in \cite[App.A]{GP}.
For example, by \cite[Sec.A4]{GP}
the pullback of the first term is
\begin{multline*}
\iota^{\ast} \iota_{\ast} \CC_{g-1}(\gamma_1, \ldots, \gamma_n, \1, \1)
=
\iota_{34 \ast} \CC_{g-2}(\gamma_1, \ldots, \gamma_n, \Delta_E, \1, \1) \\
 + \sum_{\substack{g-1 = g_1 + g_2 \\ \{ 1, \ldots, n \} = S_1 \sqcup S_2 }} \sum_{\ell}
j_{\ast} \CC_{g_1}(\gamma_{S_1}, \Delta_{E,\ell}, \1) \times \CC_{g_2} ( \gamma_{S_2}, \Delta_{E, \ell}^{\vee}, \1) \\
 - 2 \CC_{g-1}(\gamma_1, \ldots, \gamma_n, \1, \1) (\psi_{n+1} + \psi_{n+2}),
\end{multline*}
where $\iota_{34} : \Mbar_{g-2, n+4} \to \Mbar_{g-1, n+2}$
is the map gluing the $(n+3)$-th and $(n+4)$-th point,
$j$ is the map gluing the last marking on each factor, 
and $\Delta_E = \sum_{\ell} \Delta_{E,\ell} \otimes \Delta_{E, \ell}^{\vee}$ is the diagonal.
The calculation of the other terms is straightforward.
\end{proof}

We prove the holomorphic anomaly equation in full generality.

\begin{proof}[Proof of Theorem~\ref{thm_E_HAE}]
By Theorem~\ref{thm_E_quasimodularity} the homomorphism $\CC_g$
takes values in $H^{\ast}(\Mbar_{g,n}) \otimes \QMod$.
For any $g$ and $n$ we may therefore consider
\[
\mathsf{T}_{g,n} = \left( \frac{d}{dC_2} \CC_g - \mathsf{H}_g \right)
\colon H^{\ast}(E)^{\otimes n} \to
H^{\ast}(\Mbar_{g,n}) \otimes \QMod.
\]
Consider the subspace of vectors which lies in the kernel of $\mathsf{T}_{g,n}$ for every $g$,
\[ K_n = \bigcap_{g} \mathrm{Ker}( \mathsf{T}_{g,n} ) \subset H^{\ast}(E)^{\otimes n}. \]
We need to show the inclusion
\[ K = \bigoplus_{n \geq 0} K_n \  \subset \ \bigoplus_{n \geq 0} H^{\ast}(E)^{\otimes n}. \]
is an equality. We have the following list of properties of $K$.
\begin{itemize}
 \item All unbalanced $\gamma$ are in $K$ (by Lemma \ref{Balanced}).
 \item Every $K_n$ is invariant under permutations.
 \item All $\gamma = \pt \otimes \ldots \otimes \pt$ are in $K$ (by Section~\ref{Section_Elliptic_curves_Point_insertions}).
 \item If $\gamma \in K$, then $\gamma \otimes \1 \in K$ (by Lemma~\ref{545i2345}).
 \item If $\gamma \in K$, then $\gamma \otimes \Delta_E \in K$ (by Lemma~\ref{33089834}).
 \item $\mathsf{T}_{g,n}(\gamma) = \mathsf{T}_{g,n}(\phi(\gamma))$ for every $\phi \in \mathrm{SL}_2(\BZ)$ and $\gamma$
 (by Lemma~\ref{Lemma02292}).
\end{itemize}
The claim of the Theorem follows from the properties above
and the same induction argument used in Section~\ref{Subsection_Proof_of_Theorem_Quasimodularity}.
\end{proof}

\subsection{Proof of Corollary~\ref{CorE}} \label{Subsection_ProofofCorE}
The ring of quasimodular forms admits the derivations
\[ q \frac{d}{dq} : \QMod \to \QMod \quad \text{ and } \quad \frac{d}{dC_2} : \QMod \to \QMod. \]
A verification on generators of $\QMod$ proves the commutator relation
\begin{equation} \left[ \frac{d}{dC_2}, q \frac{d}{dq} \right] \bigg|_{ \QMod_k} \, = \, -2k \cdot \mathrm{id}_{\QMod_k}. \label{COMMUTATOR} \end{equation}
for every $k$.
In particular, $\QMod_k$ is
the $-2k$-eigenspace of $[ \frac{d}{dC_2}, q \frac{d}{dq} ]$.

We prove the Corollary by calculating the commutator
\[ \left[ \frac{d}{dC_2}, q \frac{d}{dq} \right] \CC_g(\gamma_1, \ldots, \gamma_n). \]

Let $p : \Mbar_{g,n+1} \to \Mbar_{g,n}$ be the map that forgets the $(n+1)$-th marked point.
By the divisor equation we have
\[ p_{\ast} \CC_g(\gamma_1, \ldots, \gamma_n, \pt) = q \frac{d}{dq} \CC_g(\gamma_1, \ldots, \gamma_n). \]
Since $p_{\ast}$ and $\frac{d}{dC_2}$ commute we therefore find
\begin{align*}
\left[ \frac{d}{dC_2}, q \frac{d}{dq} \right] \CC_g(\gamma_1, \ldots, \gamma_n)
& = p_{\ast} \frac{d}{dC_2} \CC_g(\gamma_1, \ldots, \gamma_n, \pt)
- q \frac{d}{dq} \frac{d}{dC_2} \CC_g(\gamma_1, \ldots, \gamma_n).
\end{align*}
A direct evaluation of the right hand side
using Theorem~\ref{thm_E_HAE},
relation \eqref{4514} and $p_{\ast} \psi_{n+1} = 2g-2+n$
yields
\begin{multline*}
\left[ \frac{d}{dC_2}, q \frac{d}{dq} \right] \CC_g(\gamma_1, \ldots, \gamma_n)
=
-2 ( 2g-2+n - m_0 + m_2 ) \CC_g(\gamma_1, \ldots, \gamma_n),
\end{multline*}
where $m_0$ and $m_2$ are the number of $\gamma_i$ of degree $0$ and $2$ respectively.
\qed

\section{K3 surfaces}
\label{Section_K3_surfaces}
\subsection{Overview}
Let $\pi : S \to \p^1$ be an elliptic K3 surface with section,
let $B, F \in \Pic(S)$
be the class of a section and a fiber respectively, and define
\[ \beta_h = B + h F, \ \  h \geq 0. \]

The quasimodularity \eqref{454353} is proven in \cite{MPT} by induction
on the genus and the number of markings
using the following reduction steps:
\begin{itemize}
\item Degeneration to the normal cone of an elliptic fiber $S \cup_{E} (\p^1 \times E)$.
\item Restriction to boundary divisors in $\Mbar_{g,n}$. 
\end{itemize}
We show both steps are compatible with the holomorphic anomaly equation 
(Sections~\ref{4524294} and~\ref{Subsection_Relative_K3_geometry} respectively).
This implies Theorem~\ref{thm_K3HAE} (Section~\ref{Subsection_Proof_of_K3HAE}).

\subsection{Convention} \label{Subsection_CONVENTION}
If $2g-2+n>0$ recall the forgetful morphism
\[ p : \Mbar_{g,n}(\p^1, 1) \to \Mbar_{g,n} \]
and the tautological subring
\[ R^{\ast}(\Mbar_{g,n}) \subset H^{\ast}(\Mbar_{g,n}). \]
For Section~\ref{Section_K3_surfaces} we extend both definitions to the unstable case.
If $g,n \geq 0$ but $2g-2+n \leq 0$ we define
$\Mbar_{g,n}$ to be a point, $p$ to be the canonical map to the point, and
$R^{\ast}(\Mbar_{g,n}) = \BQ$ spanned by the identity class.

This will allow us to treat unstable cases consistently throughout. We will point
out when the convention is applied.

\subsection{Boundary divisors} \label{4524294}
For any $2g-2+n>0$ consider the pushforwards
\begin{align*}
\widetilde{\CK}_{g}(\gamma_1, \ldots, \gamma_n)
& =
p_{\ast} \CK_g(\gamma_1, \ldots, \gamma_n), \\
\widetilde{\CK}^{\text{vir}}_g(\gamma_1, \ldots, \gamma_n)
& = p_{\ast} \CK^{\text{vir}}_g(\gamma_1, \ldots, \gamma_n), \\
\widetilde{\mathsf{T}}_g(\gamma_1, \ldots, \gamma_n)
& =
p_{\ast} \mathsf{T}_g(\gamma_1, \ldots, \gamma_n).
%
\end{align*}
where $p: \Mbar_{g,n}(\p^1,k) \to \Mbar_{g,n}$ is the forgetful map and $\CK_g,\CK^{\text{vir}}_g,\mathsf{T}_g$ were defined in Section~\ref{Subsection_intro_K3surfaces}.

Let $\iota : \Mbar_{g-1, n+2} \to \Mbar_{g,n}$ be the
gluing map along the last two marked points. 
We have the splitting formula \cite[7.3]{MPT}
\[ \iota^{\ast} \widetilde{\CK}_g(\gamma_1, \ldots, \gamma_n)
 =
\widetilde{\CK}_{g-1}(\gamma_1, \ldots, \gamma_n, \Delta_S),
\]
where $\Delta_S \in H^{\ast}(S \times S)$ is the diagonal.

\begin{lemma} $\iota^{\ast} \widetilde{\mathsf{T}}_g(\gamma_1, \ldots, \gamma_n)
=
\widetilde{\mathsf{T}}_{g-1}(\gamma_1, \ldots, \gamma_n, \Delta_S)$.
\end{lemma}
\begin{proof}
By a direct calculation, similar to Lemma~\ref{33089834}.
\end{proof}

For any $g = g_1 + g_2$ and $\{ 1 , \ldots, n \} = S_1 \sqcup S_2$ let
\[
j : \Mbar_{g_1, S_1 \sqcup \{ \bullet \}}
\times \Mbar_{g_2, S_2 \sqcup \{ \bullet \}} \to
\Mbar_{g, n}
\]
be the gluing map along the points marked '$\bullet$'.
We have
\begin{multline*}
j^{\ast} \widetilde{\CK}_g(\gamma_1, \ldots, \gamma_n)
= \\
\sum_{\ell} \bigg( \widetilde{\CK}_{g_1}(\gamma_{S_1}, \Delta_{S, \ell})
\boxtimes
\widetilde{\CK}^{\text{vir}}_{g_2}( \gamma_{S_2}, \Delta_{S, \ell}^{\vee} )
+
\widetilde{\CK}^{\text{vir}}_{g_1}( \gamma_{S_1}, \Delta_{S, \ell} )
\boxtimes
\widetilde{\CK}_{g_2}(\gamma_{S_2}, \Delta_{S, \ell}^{\vee} ) \bigg).
\end{multline*}

\begin{lemma}
\begin{multline*} j^{\ast} \widetilde{\mathsf{T}}_{g}(\gamma_1, \ldots, \gamma_n)
 = 
 \sum_{\ell} \bigg( \widetilde{\mathsf{T}}_{g_1}(\gamma_{S_1}, \Delta_{S, \ell})
 \boxtimes
\widetilde{\CK}^{\text{vir}}_{g_2}( \gamma_{S_2}, \Delta_{S, \ell}^{\vee} ) \\
+
\widetilde{\CK}^{\text{vir}}_{g_1}( \gamma_{S_1}, \Delta_{S, \ell} )
 \boxtimes \widetilde{\mathsf{T}}_{g_2}( \gamma_{S_2}, \Delta_{S, \ell}^{\vee}) \bigg).
\end{multline*}
\end{lemma}
\begin{proof}
The pullback under $j$ of the
first three terms of $\widetilde{\mathsf{T}}_g(\gamma_1, \ldots, \gamma_n)$
match the corresponding three terms of the right hand side.
The respective last two terms also agree by a careful matching of all the cases.
\end{proof}

\subsection{Relative geometry on $\p^1 \times E$}
Consider the trivial elliptic fibration
\[ \pi \colon \p^1\times E \rightarrow \p^1. \]
We denote the section class by $B$ and the fiber class by $E$, and write
\[ (k,d) = k B + d E \]
for the corresponding class in $H_2(\p^1 \times E,\BZ)$.
Let also 
\[ \gamma_1, \ldots, \gamma_n \in H^{\ast}(\p^1 \times E) \]
be cohomology classes. We define several Gromov--Witten classes.

\subsubsection{Absolute classes}
Recall the absolute Gromov--Witten classes
\[ \CP_{g,k}(\gamma_1, \ldots, \gamma_n) = \CC^\pi_{g,k}(\gamma_1, \ldots, \gamma_n)
 \in H_{\ast}(\Mbar_{g,n}(\p^1,k)) \otimes \BQ[[q]].
\]

\subsubsection{Relative classes}
Consider the moduli space of stable maps
 \[
\overline{M}_{g,n}\left( (\p^1 \times E) / \{0\} ,\,  (k,d), \underline{\mu} \right)
\]
to $\p^1 \times E$ in class $(k,d)$ relative to the fiber over $0 \in \p^1$
with ramification profile over $0$ specified by the partition $\underline{\mu}$
of size $k$. We let
\[
\pi :
\overline{M}_{g,n}\big( (\p^1 \times E) / \{0\} ,\,  (1,d), \underline{\mu} \big)
\to \overline{M}_{g,n}( \p^1/\{0\} , \underline{\mu} )
\]
be the morphism induced by the projection $\p^1 \times E \to \p^1$.

We are here interested only in the cases $k \in \{ 0, 1 \}$
where the partition $\underline{\mu}$ is uniquely determined.
Hence we omit it in the notation.
If $k=0$ define
\[
\CP^{\text{rel}}_{g,0}( \gamma_1, \ldots, \gamma_n ; )
=
\sum_{d=0}^{\infty} q^d
\pi_{\ast}\left(
\left[ \overline{M}_{g,n}( (\p^1 \times E)/ \{0\} , (0,d)) \right]^{\text{vir}}
\prod_{i=1}^{n} \ev_i^{\ast}(\gamma_i) \right),
\]
where $\ev_1, \ldots, \ev_n$ are the
evaluation maps at the non-relative markings.

In degree $1$
consider
the evaluation map at the unique relative marking,
\[ \ev_0 : \overline{M}_{g,n}( (\p^1 \times E)/ \{0\} ,\,  (1,d)) \to E. \]
Let $\mu \in H^{\ast}(E)$ be a relative insertion. We define
\begin{multline*}
\CP^{\text{rel}}_{g,1}( \gamma_1, \ldots, \gamma_n ; \mu ) \\
=
\sum_{d=0}^{\infty} q^d
\pi_{\ast}\left(
\left[ \overline{M}_{g,n}( (\p^1 \times E)/ \{0\} , (1,d)) \right]^{\text{vir}}
\ev_{0}^{\ast}(\mu) \prod_{i=1}^{n} \ev_i^{\ast}(\gamma_i) \right).
\end{multline*}

\subsubsection{Rubber classes}
Consider the moduli space of stable maps
\begin{equation} \Mbar_{g,n}((\p^1 \times E) / \{ 0, \infty \},\, (1,d)) \label{rkgsrgfg} \end{equation}
relative to fibers over both $0 \in \p^1$ and $\infty \in \p^1$,
and let
\[ \Mbar^{\sim}_{g,n}((\p^1 \times E) / \{ 0, \infty \},\, (1,d)) \]
denote the corresponding stable maps space to a rubber target \cite{OP3, MP}.
We have an induced morphism 
\[
\pi : \Mbar^{\sim}_{g,n}((\p^1 \times E) / \{ 0, \infty \},\, (1,d)) \to \Mbar^{\sim}_{g,n}(\p^1 / \{ 0, \infty \},\, (1))
\]
and interior evaluation maps
\[
\ev_1, \ldots, \ev_n : \Mbar^{\sim}_{g,n}((\p^1 \times E) / \{ 0, \infty \},\, (1,d)) \to E.
\]
which are the descents of the composition of the interior evaluation maps of \eqref{rkgsrgfg} with the projection $\p^1 \times E \to E$.
For any $\gamma_1, \ldots, \gamma_n \in H^{\ast}(E)$ and relative insertions $\mu, \nu \in H^{\ast}(E)$ define the rubber class
\begin{multline*}
\CP^{\text{rubber}}_{g}( \gamma_1, \ldots, \gamma_n ; \mu, \nu) \\
=
\sum_{d=0}^{\infty} q^d
\pi_{\ast}\left( 
\left[ \Mbar^{\sim}_{g,n}((\p^1 \times E) / \{ 0, \infty \},\, (1,d)) \right]^{\text{vir}}
\ev_{0}^{\ast}(\mu) \ev_{\infty}^{\ast}(\nu)
\prod_{i=1}^{n} \ev_i^{\ast}(\gamma_i) \right).
\end{multline*}
We identify the insertion $\gamma_i \in H^{\ast}(E)$ with its pullback to $H^{\ast}(\p^1 \times E)$ by
the projection to the second factor.

\subsubsection{Holomorphic anomaly equation}
By Corollary~\ref{Cor_Trivial_elliptic_fibration}, the class
\[ \CP_{g,k}(\gamma_1, \ldots, \gamma_n)
\]
lies in $H_{\ast}(\Mbar_{g,n}(\p^1,k)) \otimes \QMod$
and satisfies a holomorphic anomaly equation.
We obtain a parallel result for the relative classes
by the relative product formula \cite{LQ}.
The argument is similar to Corollary~\ref{Cor_Trivial_elliptic_fibration}
and yields
\[
\CP^{\text{rel}}_{g,1}( \gamma_1, \ldots, \gamma_n ; \mu )
\in
H_{\ast}(\overline{M}_{g,n}( \p^1/\{0\} , (1))) \otimes \QMod
\]
as well as the holomorphic anomaly equation
\begin{equation} \label{PrelHAE}
\begin{aligned}
& \frac{d}{dC_2} \CP^{\text{rel}}_{g,1}( \gamma_1, \ldots, \gamma_n ; \mu ) \\
& =
\iota_{\ast} \Delta^{!} \CP^{\text{rel}}_{g-1,1}( \gamma_1, \ldots, \gamma_n, \1, \1 ; \mu ) \\
& + 2 \sum_{\substack{g = g_1 + g_2 \\ \{ 1, \ldots, n \} = S_1 \sqcup S_2 }}
j_{\ast} \Delta^{!} \left( \CP_{g_1,1}^{\text{rel}}( \gamma_{S_1}, \1 ; \mu )
\boxtimes \CP^{\text{rel}}_{g_2,0}( \gamma_{S_2}, \1 ; \, ) \right) \\
& + 2  \sum_{\substack{g = g_1 + g_2 \\ \{ 1, \ldots, n \} = S_1 \sqcup S_2 \\ \forall i \in S_1: \gamma_i \in H^{\ast}(E) }}
j_{\ast} \left( \CP_{g_1}^{\text{rubber}}(  \gamma_{S_1}; \mu, \1 ) \boxtimes \CP_{g_2,1}^{\text{rel}}( \gamma_{S_2}; \1) \right) \\
& - 2 \sum_{i=1}^{n} \CP_{g,1}^{\text{rel}}(\gamma_1, \ldots, \gamma_{i-1}, \pi^{\ast} \pi_{\ast} \gamma_i, \gamma_{i+1}, \ldots, \gamma_n ; \mu) \psi_i \\
& - 2 \left( \int_{E} \mu \right) \CP_{g,1}^{\text{rel}}( \gamma_1, \ldots, \gamma_n ; \1 ) \Psi_0,
\end{aligned}
\end{equation}
where $\Delta : \p^1 \to \p^1 \times \p^1$ is the diagonal,
the product $\Delta^{!}$ is taken both times with respect to the evaluation maps of the two extra markings,
$\iota, j$ are the gluing maps along the extra markings,
and
\[ \Psi_0 \in H^2( \overline{M}_{g,n}( \p^1/\{0\} , (1))) \]
is the cotangent line class at the relative marking.

\subsection{Relative K3 geometry} \label{Subsection_Relative_K3_geometry}
\subsubsection{Relative classes} Let
\[ E \subset S \]
be a fixed non-singular fiber of $\pi : S \to \p^1$ over the point $\infty \in \p^1$.

For any $\gamma_1, \ldots, \gamma_n \in H^{\ast}(S)$ and $\mu \in H^{\ast}(E)$ define the relative class
\begin{equation*} \label{45245234}
\CK^{\text{rel}}_g(\gamma_1, \ldots, \gamma_n ; \mu)
=
\sum_{h = 0}^{\infty} q^{h-1}
\pi_{\ast} \left( [\Mbar_{g,n}(S / E, \beta_h)]^{\text{red}} \ev_{\infty}^{\ast}(\mu)
\prod_{i=1}^{n} \ev_{i}^{\ast}(\gamma_i) \right),
\end{equation*}
where $\pi : \Mbar_{g,n}(S / E, \beta_h) \to \Mbar_{g,n}(\p^1 / \{\infty\}, (1))$
is the induced morphism.

Since every curve on $S$ in class $\beta_h$ is of the form $B + D$ for some vertical divisor $D$ we have 
\begin{equation} \CK^{\text{rel}}_g(\gamma_1, \ldots, \gamma_n ; \mu) = 0
\ \text{ if } \ \deg_{\BR}(\mu) > 0. \label{4540830vv}
\end{equation}
Hence we will usually take the relative insertion to be $\mu = \1$.

By \cite[Lem.31]{MPT} and with the convention of Section~\ref{Subsection_CONVENTION} we have 
\[
\int p^{\ast}(\alpha) \cap \CK^{\text{rel}}_g(\gamma_1, \ldots, \gamma_n ; \1) \in \frac{1}{\Delta(q)} \QMod.
\]
for every $\alpha \in R^{\ast}(\Mbar_{g,n})$.

\subsubsection{Relative fiber classes}
Consider the moduli space
\begin{equation} \Mbar_{g,n}(S / E, dF) \label{40543434} \end{equation}
of stable maps to $S$ relative $E$ in class $dF$. Since $F \cdot E = 0$
we do not need to specify a ramification profile here.
The moduli space \eqref{40543434} carries a non-zero (non-reduced)
virtual class\footnote{The log canonical class of the pair $(S, E)$ is non-zero.}.
We define
\[ \CK_{g}^{\text{vir-rel}}(\gamma_1, \ldots, \gamma_n ;\,  )
 =
\sum_{d = 0}^{\infty} q^d \pi_{\ast}\left( [ \Mbar_{g,n}(S/E, dF) ]^{\text{vir}} \prod_{i=1}^{n} \ev_i^{\ast}(\gamma_i) \right),
\]
where $\pi : \Mbar_{g,n}(S/E, dF) \to \Mbar_{g,n}(\p^1/\{0\} , 0)$ is the induced morphism.

We have the following description of the virtual class.
\begin{lemma} \label{Lemma_Relfibbb} Let $i : E \to S$ be the inclusion
and let $p$ be the forgetful map to $\Mbar_{g,n}$.
With the convention of Section \ref{Subsection_CONVENTION} in the unstable case,
\begin{equation*}
p_{\ast} \CK_{g}^{\text{vir-rel}}(\gamma_1, \ldots, \gamma_n ;\,)
=
p_{\ast} \CK_g^{\text{vir}}(\gamma_1, \ldots, \gamma_n)
- p_{\ast} \CP_{g,0}^{\text{rel}}(i^{\ast}(\gamma_1), \ldots, i^{\ast}(\gamma_n))
\end{equation*}
\end{lemma}
\begin{proof}
Since $S$ carries a holomorphic symplectic form we have
\[
\sum_{d=0}^{\infty}
p_{\ast}\left( [ \Mbar_{g,n}(S, dF) ]^{\text{vir}} \prod_{i=1}^{n} \ev_i^{\ast}(\gamma_i) \right)
q^d
\, = \, 
\CK_g^{\text{vir}}(\gamma_1, \ldots, \gamma_n).
\]
The statement follows by applying the degeneration formula
for the degeneration $S \rightsquigarrow S \cup_E (\p^1 \times E)$
to the left hand side.
\end{proof}


\subsubsection{Relative holomorphic anomaly equation}
We define a candidate for the $\frac{d}{dC_2}$-derivative of
the relative class
\[ \CK^{\text{rel}}_g(\gamma_1, \ldots, \gamma_n ; \1) \in H_{\ast}(\Mbar_{g,n}(\p^1/\{0\}, 1))\otimes \BQ[[q]]. \]
Consider the class in $H_{\ast}(\Mbar_{g,n}(\p^1/\{0\}, 1)) \otimes \BQ[[q]]$ defined by
\begin{align*}
& \mathsf{T}^{\text{rel}}_g(\gamma_1, \ldots, \gamma_n) \\
& = \iota_{\ast} \Delta^{!} \CK^{\text{rel}}_{g-1}( \gamma_1, \ldots, \gamma_n, \1 , \1 ; \1) \\
%
%
& + 2 \sum_{\substack{g = g_1 + g_2 \\ \{ 1, \ldots, n \} = S_1 \sqcup S_2 }}
j_{\ast} \Delta^{!} \left( \CK_{g_1}^{\text{rel}}( \gamma_{S_1}, \1 ; \1 ) \times \CK_{g_2}^{\text{vir-rel}}(\gamma_{S_2}, \1 ;\,  ) \right)\\
& + 2  \sum_{\substack{g = g_1 + g_2 \\ \{ 1, \ldots, n \} = S_1 \sqcup S_2 \\ \forall i \in S_1: \gamma_i \in H^{\ast}(E) }}
j_{\ast} \left( \CK_{g_1}^{\text{rel}}( \gamma_{S_1}; \1) \times \CP_{g_2}^{\text{rubber}}(  \gamma_{S_2}; \1, \1 ) \right) \\
& - 2 \sum_{i=1}^{n} \CK_g^{\text{rel}}( \gamma_1, \ldots, \gamma_{i-1}, \pi^{\ast} \pi_{\ast} \gamma_i, \gamma_{i+1}, \ldots, \gamma_n ; \1) \psi_i\\
& + 20 \sum_{i=1}^{n} \langle \gamma_i, F \rangle \CK_g^{\text{rel}}( \gamma_1, \ldots, \gamma_{i-1}, F, \gamma_{i+1}, \ldots, \gamma_n ; \1) \\
& - 2 \sum_{i < j} \CK^{\text{rel}}( \gamma_1, \ldots,
\underbrace{\sigma_{1}(\gamma_i,\gamma_j)}_{i^{\text{th}}},
\ldots, \underbrace{\sigma_{2}(\gamma_i,\gamma_j)}_{j^{\text{th}}}, \ldots, \gamma_n ; \1 ),
\end{align*}
where $\Delta : \p^1 \to \p^1 \times \p^1$ is the diagonal,
in the first and second term on the right hand side
the intersection $\Delta^{!}$ and the gluing maps $j$ are taken with respect to the extra interior marked points,
and in the third term $j$ is the gluing map between the relative point on the K3 and one of the markings on the rubber class.

\subsubsection{Compatibility with the degeneration formula I}
\label{foigsdfgdfg}
Assuming quasimodularity we expect the holomorphic anomaly equations
\begin{align}
\frac{d}{dC_2} \CK_g(\gamma_1, \ldots, \gamma_n)
& = \mathsf{T}_g(\gamma_1, \ldots, \gamma_n) \label{eq:abscase} \\
\frac{d}{dC_2} \CK^{\text{rel}}_g(\gamma_1, \ldots, \gamma_n ; \1)
& = \mathsf{T}^{\text{rel}}_g(\gamma_1, \ldots, \gamma_n) \label{eq:relcase}
\end{align}

The degeneration formula yields a compatibility check for these equations.
Consider the degeneration
\begin{equation} S \rightsquigarrow S \cup_E ( \p^1 \times E ) \label{sdfsdfsd} \end{equation}
and apply the degeneration formula to $\CK_{g}(\gamma_1, \ldots, \gamma_n)$
for any choice of lift of $\gamma_1, \ldots, \gamma_n$. The result is
\begin{equation} \label{RFGRFG222}
p_{\ast} \CK_{g}(\gamma_1, \ldots, \gamma_n)
=
\sum_{\substack{ g = g_1 + g_2 \\ \{ 1, \ldots, n \} = S_1 \sqcup S_2 }}
p_{\ast} \iota_{\ast} \left( \CK_{g_1}^{\text{rel}}(\gamma_{S_1}; \1) \boxtimes \CP^{\text{rel}}_{g_2}( \gamma_{S_2} ; \pt) \right)
\end{equation}
where $\iota$ is the gluing map along the relative point and we used the vanishing \eqref{4540830vv}.
Assuming \eqref{eq:relcase} and using \eqref{PrelHAE}
we can calculate the $C_2$ derivative of the right hand side of \eqref{RFGRFG222}.
A direct check shows the result coincides with applying the degeneration formula to
the right hand side of \eqref{eq:abscase}.
Hence the formulas \eqref{eq:abscase} and \eqref{eq:relcase} are compatible
with the degeneration formula.

\subsubsection{Compatibility with the degeneration formula II}
Let $\leq$ be the lexicographic order on the set of pairs $(g,n)$, i.e.
\begin{equation} \label{lexigraph}
(g_1, n_1) \leq (g_2, n_2) \quad \Longleftrightarrow \quad g_1 < g_2 \text{ or } \big( g_1=g_2 \text{ and } n_1 \leq n_2 \big).
\end{equation}

The following proposition shows the holomorphic anomaly equation
in the absolute case implies the relative case.
Here and in the proof we use the convention of Section~\ref{Subsection_CONVENTION}.

\begin{prop} \label{DEGENERATIONPROP} Let $G,N$ be fixed. Assume
\[
\frac{d}{dC_2} \int p^{\ast}(\alpha) \cdot \CK_g(\gamma_1, \ldots, \gamma_n)
=
\int p^{\ast}(\alpha) \cdot \mathsf{T}_g(\gamma_1, \ldots, \gamma_n)
\]
for every $(g,n) \leq (G,N)$, $\alpha \in R^{\ast}(\Mbar_{g,n})$
and $\gamma_1, \ldots, \gamma_n \in H^{\ast}(S)$.
Then
\[
\frac{d}{dC_2} \int p^{\ast}(\alpha) \cdot \CK^{\text{rel}}_g(\gamma_1, \ldots, \gamma_n)
=
\int p^{\ast}(\alpha) \cdot \mathsf{T}^{\text{rel}}_g(\gamma_1, \ldots, \gamma_n)
\]
for every $(g,n) \leq (G,N)$, $\alpha \in R^{\ast}(\Mbar_{g,n})$
and $\gamma_1, \ldots, \gamma_n \in H^{\ast}(S)$.
\end{prop}
\begin{proof}
Let $\pt_S \in H^4(S)$ be the point class and assume
\[ \gamma_i \in \{ \1, \pt_S , F, W \} \cup V \]
for every $i$.
We apply the degeneration formula for the degeneration \eqref{sdfsdfsd}
where we choose all $\gamma_i$ with $\gamma_i \notin \{ \1, W \}$
to specialize to the component $S$.
Writing out \eqref{RFGRFG222} we find
\begin{equation} \label{fgodjfgfdg}
p_{\ast} \CK_{g}(\gamma_1, \ldots, \gamma_n) = p_{\ast} \CK_{g}^{\text{rel}}(\gamma_1, \ldots, \gamma_n; \1) + \ldots,
\end{equation}
where '$\ldots$' stands for terms of lower order (i.e. for which $(g_1, n_1) < (g,n)$).

We argue now by induction over $(g,n)$. Let $(g,n)$ be given and assume the claim holds for all
$(g',n')$ with $(g',n') < (g,n)$.
After integration against any tautological class $\alpha$
both sides of \eqref{fgodjfgfdg} are quasimodular forms.
Hence after integration we may apply $\frac{d}{dC_2}$.
By the induction hypothesis, the assumption in the Proposition, and \eqref{PrelHAE},
all terms except for
\[ \frac{d}{dC_2} \int_{\Mbar_{g,n}} \alpha \cdot p_{\ast} \CK_{g}^{\text{rel}}(\gamma_1, \ldots, \gamma_n; \1) \]
are determined. By the compatibility check
of Section~\ref{foigsdfgdfg} the claim follows also in the case $(g,n)$.
\end{proof}

\subsection{Proof of Theorem~\ref{thm_K3HAE}} \label{Subsection_Proof_of_K3HAE}
With the convention of Section~\ref{Subsection_CONVENTION}, we show
\begin{equation} \label{904goijf}
\frac{d}{dC_2} \int p^{\ast}(\alpha) \cap \CK_g(\gamma_1, \ldots, \gamma_n)
=
\int p^{\ast}(\alpha) \cap  \mathsf{T}_g(\gamma_1, \ldots, \gamma_n)
\end{equation}
for all $\gamma_1, \ldots, \gamma_n \in H^{\ast}(E)$ and $\alpha \in R^{\ast}(\Mbar_{g,n})$.

Assume the classes $\gamma_1, \ldots, \gamma_n \in H^{\ast}(S)$ and $\alpha$ are homogenenous and consider the dimension constraint
\begin{equation} \label{dimensionconstraint}
g + n = \deg(\alpha) + \sum_{i=1}^{n} \deg(\gamma_i), 
\end{equation}
where $\deg( )$ denotes half the real cohomological degree.
The left hand side in \eqref{dimensionconstraint} is the reduced virtual dimension of $\Mbar_{g,n}(S, \beta_h)$.
If the dimension constraint is violated, both sides of \eqref{904goijf} are zero and the claim holds.
Hence we may assume \eqref{dimensionconstraint}.

We argue by induction on $(g,n)$ with respect to the ordering \eqref{lexigraph}.
If $(g,n) = (0, 0)$, then by the Yau-Zaslow formula \cite{BL}
\[ \int \CK_{0}() = \frac{1}{\Delta(q)}. \]
Hence the left hand side in \eqref{904goijf} vanishes, and by inspection also the right hand side.
We may therefore assume $(g,n) > (0,0)$ and
the claim holds for any
$(g',n') < (g,n)$.
We have four cases.

\vspace{5pt}
\noindent \textbf{Case (i):} $g=0$ and $\deg(\gamma_i) = 1$ for all $i$.

\vspace{5pt}
By the dimension constraint $\deg(\alpha) = 0$.
Let 
\[ p_n : \Mbar_{g,n}(\p^1,1) \to \Mbar_{g,n-1}(\p^1,1) \]
be the map that forgets the last point and for
any $D \in H^2(S)$ let
\[ \frac{d}{dD} = \langle D, F \rangle q \frac{d}{dq} + \langle D,W \rangle. \]
Then we have
\begin{align*}
& \frac{d}{dC_2} \int \CK_g(\gamma_1, \ldots, \gamma_n)\\
= & \frac{d}{dC_2} \int p_{n \ast} \CK_g(\gamma_1, \ldots, \gamma_n) \\
= & \frac{d}{dC_2} \frac{d}{d \gamma_{n}} \int \CK_{g}( \gamma_1, \ldots, \gamma_{n-1} ) \\
= & 
\left[ \frac{d}{dC_2}, \frac{d}{d \gamma_n} \right] \int \CK_{g}( \gamma_1, \ldots, \gamma_{n-1} ) +
\frac{d}{d \gamma_n} \int \mathsf{T}_{g}(\gamma_1, \ldots, \gamma_{n-1}),
\end{align*}
where we used the divisor equation in the second last and
the induction hypothesis in the last step.
By a direct computation using \eqref{COMMUTATOR}
and the weight statement \cite[Thm.9]{BOPY}
the last term equals precisely
\[ \int p_{n \ast} \mathsf{T}_g(\gamma_1, \ldots, \gamma_n). \]

\vspace{5pt}
\noindent \textbf{Case (ii):} $\deg(\gamma_i) = 2$ for some $i$.

\vspace{5pt}
We may assume $\deg(\gamma_1)=2$. 
We apply the degeneration formula to the degeneration
$S \rightsquigarrow S \cup_E ( \p^1 \times E )$
and specialize $\gamma_1$ to the $\p^1 \times E$ component, while
choosing an arbitrary lift for the other insertions.
The result is
\[
p_{\ast} \CK_{g}(\gamma_1, \ldots, \gamma_n)
=
\sum_{\substack{ g = g_1 + g_2 \\ \{ 2, \ldots, n \} = S_1 \sqcup S_2}}
p_{\ast} \iota_{\ast} \left( \CK_{g_1}^{\text{rel}}(\gamma_{S_1}; \1) \boxtimes \CP^{\text{rel}}_{g_2}( \gamma_1, \gamma_{S_2} ; \pt) \right).
\]
Since every $(g_1, |S_1|) < (g,n)$ the $\frac{d}{dC_2}$ derivative of 
(the integral against any tautological class of) the right hand side is determined by induction, Proposition~\ref{DEGENERATIONPROP} and \eqref{PrelHAE}.
By Section~\ref{foigsdfgdfg}
it matches the output of the degeneration formula applied to
\[ \int_{\Mbar_{g,n}} \alpha \cap p_{\ast} \mathsf{T}_g(\gamma_1, \ldots, \gamma_n) \]
which completes the step.

\vspace{5pt}
\noindent \textbf{Case (iii):} $g>0$ and $\deg(\gamma_i) \leq 1$ for all $i$.

\vspace{5pt}
By the dimension constraint we must have $\deg(\alpha) \geq g$.
By a strong form of Getzler's vanishing \cite[Prop.2]{FPrel} we have
\[ \alpha = \iota_{\ast} \alpha' \]
for some $\alpha'$, where $\iota : \partial \Mbar_{g,n} \to \Mbar_{g,n}$ is the inclusion of the boundary.
By the compatibilities of Section~\ref{4524294}
we are reduced to lower order.

\vspace{5pt}
\noindent \textbf{Case (iv):} $g = 0$, $\det(\gamma_i) \leq 1$ for all $i$ and $\deg(\gamma_i) = 0$ for at least one $i$.

\vspace{5pt}
By the dimension constraint we have $\deg(\alpha) > 0$ 
and $\alpha$ is the pushforward of a class on the boundary.
The case follows again by Section~\ref{4524294}. \qed

\subsection{An example}
We use the bracket notation
\[
\blangle \tau_{k_1}(\gamma_1) \cdots \tau_{k_n}(\gamma_n) \rangle_g
= \sum_{h = 0}^{\infty} q^{h-1} \int_{[\Mbar_{g,n}(S, \beta_h)]^{\text{red}}} \prod_{i=1}^{n} \ev_{i}^{\ast}(\gamma_i) \psi_i^{k_i}.
\]
We give an example of the holomorphic anomaly equation in genus $1$.
Consider the series $\langle \tau_1(W) \rangle_1$.
By a monodromy argument and a direct evaluation\footnote{
The evaluation $\langle \tau_1(F) \rangle_1 = \frac{2 C_2}{\Delta}$
follows also from the holomorphic anomaly equation.}
following \cite[App.A]{K3xP1} we have
\[
\blangle \tau_1(W) \brangle_1 = q \frac{d}{dq} \blangle \tau_1(F) \brangle_1 = q \frac{d}{dq} \left( \frac{2 C_2(q)}{\Delta(q)} \right).
\]
Hence using the commutator relation \eqref{COMMUTATOR} we calculate
\begin{equation} \frac{d}{dC_2} \blangle \tau_1(W) \brangle_1 = 40 \frac{C_2(q)}{\Delta(q)} + q \frac{d}{dq} \frac{2}{\Delta(q)}.
\label{45943}
\end{equation}

On the other hand the holomorphic anomaly equation yields
\begin{equation} \label{529544}
\frac{d}{dC_2} \blangle \tau_1(W) \brangle_1
= \blangle \tau_1(W) \tau_0(\Delta_{\p^1}) \brangle_0 - 2 \blangle \tau_2(\1) \brangle_1 + 20 \blangle \tau_1(F) \brangle_1.
\end{equation}
A direct calculation shows
\begin{align*}
\blangle \tau_1(W) \tau_0(\Delta_{\p^1}) \brangle_0 & = 2 \blangle \tau_1(W) \tau_0(\1) \tau_0(F) \brangle_0 = 2 q \frac{d}{dq} \frac{1}{\Delta(q)} \\
\blangle \tau_2(\1) \brangle_1 & = 0.
\end{align*}
Plugging everything into \eqref{529544} we arrive exactly at \eqref{45943}.

\section{The Igusa cusp form conjecture}
\label{Section_Igusa_cusp_form_conjecture}
\subsection{Overview}
Let $S$ be a non-singular projective K3 surface,
let $E$ be a non-singular elliptic curve, and let
\[ X = S \times E. \]
We present the proof of the Igusa cusp form conjecture
(Theorem~\ref{thmIgusa}).
In Section~\ref{Subsection_PTtheory}
we introduce reduced Pandharipande--Thomas invariants.
In Section~\ref{Subsection_JacobiForms}
we recall properties of Jacobi forms.
Sections~\ref{Subsection_Constraints},
\ref{Subsection_Comparision}
and \ref{Subsection_ProofConstraintsII} are the heart of the proof.
We first state a list of constraints on three-variable generating series
and prove they determine the series from initial data.
Then we show both $\CZ(u,q,\tilde{q})$
and $\chi_{10}^{-1}$ satisfy these constraints.
In Section~\ref{Subsection_Completeproof}
we put the pieces together and complete the proof.

\subsection{Pandharipande--Thomas theory}
\label{Subsection_PTtheory}
Let $\beta \in H_2(S,\BZ)$ be a curve class and let $d \geq 0$.
Following \cite{PT1} let
\[
P_n(X, (\beta,d))
\]
be the moduli space of stable pairs $(F,s)$ on $X$
with numerical invariants
\[
\chi(F)=n\in \mathbb{Z}
\  \ \ \text{and} \ \ \ \ch_2(F) = (\beta,d) \in H_2(X,\mathbb{Z})\,.
\]
For any non-zero $\beta$ 
the
group $E$ acts on the moduli space 
by translation with finite stabilizers. 
\emph{Reduced Pandharipande--Thomas invariants} are defined
by integrating the Behrend function \cite{B}
\[ \nu : P_n(X, (\beta,d))/E \to \BZ \]
with respect to the
orbifold topological Euler characteristic $e( \cdot )$,
\[
\mathsf{P}_{n, (\beta,d)} = \int_{ P_n(X, (\beta,d))/E } \nu \dd{e}
= \sum_{k \in \BZ} k \cdot e\left( \nu^{-1}(k) \right) \,.
\]
The definition is equivalent to
integrating the reduced virtual class against insertions \cite{O1}.
In particular, $\mathsf{P}_{n, (\beta,d)}$ is
deformation invariant.

Let $\beta_h \in H_2(S,\BZ)$
be a primitive curve class satisfying
$\langle \beta_h, \beta_h \rangle = 2h-2$.
By deformation invariance
$\mathsf{P}_{n,(\beta_h,d)}$ only depends on $n, h$ and $d$.
We write
\[
\mathsf{P}_{n,h,d} = \mathsf{P}_{n,(\beta_h,d)}.
\]
By \cite[Prop.5]{K3xE}
every $\sum_{n \in \BZ} \mathsf{P}_{n, h, d} y^n$
is the Laurent expansion of a rational function
and we have the Gromov--Witten/Pairs correspondence
\begin{equation} \label{GWPTcorr}
\sum_{n \in \BZ} \mathsf{P}_{n, h,d} y^n
=
\sum_{g} \mathsf{N}_{g,h,d} u^{2g-2}
\end{equation}
under the variable change $y = -e^{iu}$.

\subsection{Jacobi forms}\label{Subsection_JacobiForms}
Jacobi forms are generalizations
of modular forms which depend on an elliptic parameter
$u \in \BC$ and a modular parameter
$q$, see \cite{EZ} for an introduction\footnote{
The variables $z\in \BC$ and $\tau \in \BH$ of \cite{EZ}
are related to $(u,q)$ by $u = 2 \pi z$ and $q = e^{2 \pi i \tau}$.
}.
We will also use the variables
\begin{equation*} p = e^{iu}, \quad y = -p \label{varchange} \end{equation*}
and make the convention to identify
a function in $u$ with the corresponding
function in $y$ or $p$. 
The $q^{k}$-coefficient
in the expansion of a function $f(u,q)$ is denoted by
$[ f(u,q) ]_{q^k}$ and similarly for the other variables.

Consider the Jacobi theta function
\begin{equation}  \label{Theta}
\begin{aligned}
\Theta(u,q)
& = u \exp \Big( \sum_{k \geq 1} (-1)^{k-1} C_{2k} u^{2k} \Big)  \\
& = (p^{1/2} - p^{-1/2}) \prod_{m \geq 1} 
\frac{ (1-pq^m) (1-p^{-1}q^m)}{ (1-q^m)^2 } 
\end{aligned}
\end{equation}
and the Weierstra{\ss} elliptic function
\begin{equation} \label{WP}
\begin{aligned} \wp(u,q) & =
-\frac{1}{u^2} - \sum_{k \geq 2} (-1)^k (2k-1) 2k C_{2k} u^{2k-2} \\
& =
\frac{1}{12} + \frac{p}{(1-p)^2} + \sum_{d \geq 1}
\sum_{k | d} k (p^k - 2 + p^{-k}) q^{d}.
\end{aligned}
\end{equation}
Define
\begin{equation} \phi_{-2,1}(u,q) = \Theta(u,q)^2,
\quad
\phi_{0,1}(u,q) = 12 \Theta(u,q)^2 \wp(u,q).
\label{phik1}
\end{equation}
The ring of weak Jacobi forms of even weight
is the free polynomial algebra
\[ \CJ
= \BQ[ C_4, C_6, \phi_{-2,1}, \phi_{0,1}]. \]
We assign the functions $\phi_{k,1}$ weight $k$ and index $1$,
and the Eisenstein series $C_k$ weight $k$ and index $0$.
We let
\[ \CJ = \bigoplus_{k,m} \CJ_{k,m} \]
denote the induced bi-grading by weight $k$ and index $m$.

Recall also the ring of modular forms
\[ \Mod = \bigoplus_k \Mod_k = \BQ[C_4, C_6] \]
graded by weight.
The following fact is well-known.

\begin{lemma} \label{Modconstraint}
Let $f \in \Mod_k$.
If $[f(q)]_{q^{\ell}} = 0$ for all
$\ell \leq \lfloor \frac{k}{12} \rfloor$,
then $f(q) = 0$.
\end{lemma}

For Jacobi forms we have
the following analog.

\begin{lemma} \label{Jacconstraint}
Let $\phi \in \CJ_{k,m}$.
If $[\phi]_{q^{\ell}} = 0$ for all
$\ell \leq \lfloor \frac{k+2m}{12} \rfloor$,
then $\phi=0$. 
\end{lemma}
\begin{proof}
Let $\phi \in \CJ_{k,m}$ and
let $\phi = \sum_{n,r} c(n,r) q^n p^r$ be its Fourier expansion.
By \cite[Thm.3.1]{EZ} for every $\nu \geq 0$
the series
\begin{equation} \label{Duf}
\CD_{\nu} f = \sum_{n=0}^{\infty}
\left( \sum_{r \in \BZ} p_{2 \nu}^{(k-1)}(r,nm) c(n,r) \right) q^n
\end{equation}
is a modular form of weight $k+2 \nu$;
here $p_{2\nu}^{(k-1)}$ is a certain explicit polynomial.
Moreover, by \cite[Thm.9.2]{EZ} the mapping
\[ \CD = \CD_0 \oplus \ldots \oplus \CD_m \colon
\CJ_{k,m} \to \Mod_k \oplus \ldots \oplus \Mod_{k+2 m} \]
is an isomorphism.

If $[\phi]_{q^{\ell}} = 0$
for all $\ell \leq \lfloor \frac{k+2m}{12} \rfloor$,
then
$[ \CD_{\nu} \phi ]_{q^{\ell}} = 0$
for all $\nu$ and $\ell \leq \lfloor \frac{k+2m}{12} \rfloor$
by \eqref{Duf}.
Applying Lemma~\ref{Modconstraint} we find
$\CD_{\nu} \phi = 0$ for all $\nu \leq m$, so $\phi = 0$.
\end{proof}

We require the following property of the $u$-expansion of Jacobi forms.

\begin{lemma}
\label{Lemma_JacddC2}
Let $\phi \in \CJ_{k,m}$ and let
$\phi(u,q) = \sum_{\ell \geq 0} f_{\ell}(q) u^{\ell}$ be its $u$-expansion.
 Then every $f_\ell(q)$ is a quasimodular form of weight $\ell+k$ and
 \[ \frac{d}{dC_2} f_{\ell} \, = \, 2m \cdot f_{\ell-2}. \]
\end{lemma}
\begin{proof}
Consider $\Theta$ and $\wp$ as power series
in $u$ with quasimodular form coefficients.
By \eqref{Theta} and \eqref{WP}
we have
\begin{equation} \frac{d}{dC_2} \Theta = u^2 \Theta, \quad
\frac{d}{dC_2} \wp = 0. 
\label{dddc2} \end{equation}
Every $\phi \in \CJ_{k,m}$ can be written as
\[ \phi = \Theta^{2m} P(C_4,C_6, \wp) \]
for some (weighted homogeneous) polynomial $P$.
Hence $\phi$ is a power series in $u$ with quasimodular form coefficients 
and
\[ \frac{d}{dC_2} \phi = 2m u^2 \phi . \]
This proves the quasimodularity of the $f_{\ell}$ and
the second claim. The weight statement
follows by an inspection
of the weights of the quasimodular forms entering the definition
of $\Theta$ and $\wp$.
\end{proof}

\begin{lemma}
\label{Lemmaconstraint2}
Let
$\phi \in \frac{1}{\phi_{-2,1} \Delta} \CJ_{k.m}$.
Then there exist $c_{g,d} \in \BQ$ such that
\[ \phi(u,q) =
\sum_{d \geq 0}
\left( \sum_{g=0}^{m+d}
c_{g,d} (y^{\frac{1}{2}} + y^{-\frac{1}{2}})^{2g-2} \right)
q^{d-1}
\]
under the variable change $y=-e^{iu}$
and in the region $0 < |q| < |y| < 1$.
\end{lemma}

\begin{proof}
Let $y = -p = -e^{iu}$ throughout.
The $q$-coefficients of functions in $\CJ$
are Laurent polynomials in
$y$ that are invariant under $y \mapsto y^{-1}$,
and hence can be written as
a linear combination of even non-negative powers of
$y^{\frac{1}{2}} + y^{-\frac{1}{2}}$.
By inspection of the $y$-expansion of $\Theta$ we find
\begin{equation}
\label{4omrfgdfg}
\phi(u,q) =
\sum_{d \geq 0}
\left( \sum_{g=0}^{N_d}
c_{g,d} (y^{\frac{1}{2}} + y^{-\frac{1}{2}})^{2g-2} \right)
q^{d-1}
\end{equation}
for some $c_{g,d} \in \BQ$ and $N_d$.
We need to show $N_d \leq m+d$.

Let
$\phi = \psi / ( \phi_{-2,1} \Delta )$
for some $\psi \in \CJ_{k,m}$ and consider
the functions $\phi, \psi, \phi_{-2,1}^{-1}$ as formal power series
in $y$ and $q$ expanded in the region $0 < |q| < |y| < 1$.
Since $\psi$ satisfies the elliptic transformation law \cite{EZ}
the power series $\psi(y,q)$ satisfies
\[ \psi(y^{-1} q,q) = y^{2m} q^{-m} \psi(y,q). \]
By \eqref{Theta}
the series $1/\phi_{-2,1}(y,q)$
satisfies the equality of power series
\begin{equation} \frac{1}{\phi_{-2,1}}(y^{-1} q, q)
= y^{-2} q^{1} \frac{1}{\phi_{-2,1}}(y,q). \label{poeldas}
\end{equation}
Combining both equations
we obtain the identity of power series
\begin{equation}
\phi(y^{-1} q, q) = y^{2(m-1)} q^{-(m-1)} \phi(y,q). \label{5452646}\end{equation}
Let
$\phi = \sum_{d,r} b(d,r) y^r q^{d-1}$.
Then \eqref{5452646} is equivalent to
\[ b(d,r) = b(d+r+m-1, -2m-r+2). \]
In particular, $b(d,r) = 0$ if
$d+r+m-1 < -1$ or equivalently $r < -(m+d)$.
This shows $N_{d} \leq m+d$ in \eqref{4omrfgdfg}.
\end{proof}

\subsection{Constraints} 
\label{Subsection_Constraints}
Let $\CF(u,q,\tilde{q})$
be a formal power series
in the variables $u, q, \tilde{q}$
which satisfies the following properties:

\vspace{4pt}
\noindent \textbf{Property 1.} There exist $a_{g,h,d} \in \BQ$ such that
\[ \CF(u,q,\tilde{q})
=
\sum_{h,d \geq 0} \sum_{g \geq 0} a_{g,h,d} u^{2g-2} q^{h-1} \tilde{q}^{d-1}. \]

\vspace{4pt}
\noindent \textbf{Property 2.} For every $h$ we have
\[ \big[ \, \CF \, \big]_{q^{h-1}} \in \frac{1}{ \phi_{-2,1} \Delta}
\CJ_{0,h} \]
where the right side denotes Jacobi forms in the variables $(u, \tilde{q})$.

\vspace{4pt}
\noindent \textbf{Property 3.} For every $g$ and $d$
the series
\begin{equation} \CF_{g,d}(q) = \big[ \, \CF \, \big]_{u^{2g-2} \tilde{q}^{d-1}}
\label{Cgd} \end{equation}
satisfies
\begin{enumerate}
\item[(a)] $\CF_{g,d}(q) \in \frac{1}{\Delta(q)} \QMod_{2g}$,
\item[(b)] $\frac{d}{dC_2} \CF_{g,d} = (2d-2) \CF_{g-1,d}$.
\end{enumerate}
\vspace{9pt}

We show Properties 1--3 determine the series $\CF$ up to a single coefficient.

\begin{prop} \label{PropUniqueness}
Assume the series $\CF(u,q,\tilde{q})$ satisfies Properties 1,2,3 above.
If moreover $a_{0,0,0} = 0$, then $\CF = 0$.
\end{prop}

\begin{proof}
Let $\CF$ be a series which satisfies Properties 1--3 and $a_{0,0,0} = 0$.
We show by induction that
$[ \CF ]_{q^{h-1}} = 0$ for every $h \geq 0$.

\vspace{4pt}
\noindent \textbf{Base case:}
By Property 2, Lemma~\ref{Lemmaconstraint2}
and $a_{0,0,0} = 0$
we have
$[\CF]_{q^{-1} \tilde{q}^{-1}} = 0$,
hence $[\phi_{-2,1} \Delta \cdot \CF]_{q^{-1} \tilde{q}^{0}} = 0$,
and since $\CJ_{0,0} = \BQ$ therefore
$\phi_{-2,1} \Delta [ \CF ]_{q^{-1}} = 0$.

\vspace{4pt}
\noindent \textbf{Induction:}
Let $N \geq 0$ and assume $[ \CF ]_{q^{h-1}} = 0$ for all $h \leq N$.
Then
for all $g$ and $d$
the series $\CF_{g,d}(q)$ defined in \eqref{Cgd} satisfies
\begin{equation}
[ \CF_{g,d}(q) ]_{q^{\ell}} = 0 \ \text{ for all } \ell < N.
\label{Cgdvanish} \end{equation}

\noindent \emph{Claim:} $\CF_{g,d}(q) = 0$ whenever $\lfloor g/6 \rfloor \leq N$.

\noindent \emph{Proof of Claim:}
We use a second induction over all $g$
such that $\lfloor g/6 \rfloor \leq N$. If $g = 0$, then
by Property 3 and $\QMod_0 = \BQ$
we have $\CF_{0,d} = a/\Delta(q)$ for some $a \in \BQ$,
so the claim follows from \eqref{Cgdvanish}.
Assume the claim holds for $g-1$.
We show it holds for $g$.
By Property 3 and induction we have
\[ \frac{d}{dC_2} \CF_{g,d} = (2d-2) \CF_{g-1,d} = 0. \]
We conclude
$\CF_{g,d} \Delta \in \Mod_{2g}$.
By \eqref{Cgdvanish} we have
$[ \CF_{g,d} \Delta ]_{q^{\ell}} = 0$
for all $\ell \leq N$,
hence in particular
for all $\ell \leq \lfloor g/6 \rfloor$
(since $g$ lies in the range $\lfloor g/6 \rfloor \leq N$).
Using Lemma~\ref{Modconstraint} we conclude $\CF_{g,d} = 0$.
\qed

We continue the proof of the Proposition.
By Property~2 and Lemma~\ref{Lemmaconstraint2},
for every $h,d$ there exist $c_{g,h,d} \in \BQ$
such that
\begin{equation} \label{399i8}
\big[ \, \CF  \, \big]_{q^{h-1} \tilde{q}^{d-1}}
=
\sum_{g=0}^{h+d} c_{g,h,d} (y^{\frac{1}{2}} + y^{-\frac{1}{2}})^{2g-2}
\end{equation}
under the variable change $y = -e^{iu}$.
We have
\[ y^{1/2} + y^{-1/2} =
- 2 \sin\left( \frac{u}{2} \right) = - u + \frac{1}{24} u^3 + \ldots \ . \]
Hence for every $h,d$
we have an invertible
and upper-triangular relation
between the coefficients $\{ a_{g,h,d} \}_{g \geq 0}$
and the coefficients $\{ c_{g,h,d} \}_{g \geq 0}$.
By the Claim we have
$a_{g,h,d} = 0$
whenever $\lfloor g/6 \rfloor \leq N$.
Therefore
$c_{g,h,d} = 0$ for all $\lfloor g/6 \rfloor \leq N$.
Since $g \leq h+d$ in
the sum in \eqref{399i8}
we thus find
\begin{equation} \big[ \, \CF  \, \big]_{q^{h-1} \tilde{q}^{d-1}} = 0
\, \text{ for all } h,d \text{ such that }\,
\left\lfloor \frac{h+d}{6} \right\rfloor \leq N.
\label{vanish3} \end{equation}

Let $\phi(u,\tilde{q}) = [ \CF ]_{q^{N}}$.
We show $\phi = 0$ and conclude the induction step.
By Property 2 we have $\Theta^2 \Delta \phi \in \CJ_{0,N+1}$.
On the other hand
specializing to $h = N+1$ in \eqref{vanish3} and shifting by
$\Theta^2 \Delta$ yields
\begin{equation} \label{vanish4}
\left[ \Theta^2 \Delta \phi \right]_{\tilde{q}^{\ell}} = 0
\end{equation}
for all $\ell$ such that $\lfloor \frac{1}{6} (N+\ell+2) \rfloor \leq N + 1$,
or equivalently, such that
$\ell < 5N+10$.
Since $N \geq 0$
this implies
the vanishing of \eqref{vanish4}
for all $\ell \leq \lfloor (N+1)/6 \rfloor$.
An application of Lemma~\ref{Jacconstraint} yields $\phi = 0$.
\end{proof}

\subsection{Proof of constraints I}
\label{Subsection_Comparision}
Recall the Igusa cusp form $\chi_{10}$ and let
\begin{equation}
\CF(u,q, \tilde{q}) = -\frac{1}{\chi_{10}(p,q,\tilde{q})}
\label{CFdef} \end{equation}
be the Laurent expansion in $u$ under the variable
change $p=e^{iu}$.

\begin{prop}
\label{ConstraintProp1}
$\CF(u,q, \tilde{q})$ satisfies
Properties 1--3 of Section~\ref{Subsection_Constraints}.
\end{prop}

\begin{proof}
Let $V_\ell$ be the $\ell^{\text{th}}$ Hecke operator on
Jacobi forms defined in \cite[\S 4]{EZ}.
Definition \eqref{IGUSA} is equivalent to
\begin{equation}
\chi_{10} = - \tilde{q} \Theta(u,q)^2 \Delta(q)
\exp\left( - \sum_{\ell = 1}^{\infty}
\tilde{q}^\ell \cdot (Z|_{0,1}V_{\ell})(u,q) \right)
\label{102} \end{equation}
where $Z = 2 \phi_{0,1}(u,q) \in \CJ_{0,1}$.
By \cite[\S 4]{EZ} for every $\ell \geq 1$ we have
\[ Z|_{0,1}V_{\ell} \in \CJ_{0, \ell} \]
from which we obtain for all $d$
\begin{equation} \phi_d = \big[ \, \CF \, \big]_{\tilde{q}^{d-1}}
\in \frac{1}{\Theta^2 \Delta} \CJ_{0,d}.
\label{phidddd} \end{equation}
Using the $(u,q)$-expansions
of $\Theta$, $\Delta$ and the generators of $\CJ$
we conclude Property 1.
By \eqref{IGUSA}
the series $\CF$ is invariant
under
interchanging $q$ and $\tilde{q}$,
\begin{equation*}
\CF(u,q, \tilde{q}) = \CF(u, \tilde{q}, q).
\end{equation*}
Hence \eqref{phidddd} implies also Property 2.

By an argument similar to the proof of Lemma~\ref{Lemma_JacddC2}
the $u^{2g-2}$-coefficient of $\Delta(q) \phi_d(u,q)$
is a quasimodular form of weight $2g$ and
\[ \frac{d}{dC_2} \phi_d = (2d-2) u^2 \phi_d. \]
This shows Property 3.
\end{proof}

\subsection{Proof of constraints II}
\label{Subsection_ProofConstraintsII}
Recall from \eqref{K3xEpartition} the three-variable
generating series of Gromov--Witten invariants
\[
\CZ(u,q, \tilde{q})
=
\sum_{h=0}^{\infty} \sum_{d=0}^{\infty} \sum_{g=0}^{\infty}
\mathsf{N}_{g,h,d} u^{2g-2} q^{h-1} \tilde{q}^{d-1}.
\]

\begin{prop}
\label{ConstraintProp2}
$\CZ(u,q, \tilde{q})$ satisfies
Properties 1-3 of Section~\ref{Subsection_Constraints}.
\end{prop}

We begin the proof with two Lemmas.
\begin{lemma} \label{Lemma9457}
For all $g$ and $h$ the series
$f_{g,h}(q) = \sum_{d \geq 0} \mathsf{N}_{g,h,d} q^{d-1}$
satisfies
\begin{enumerate}
\item[(a)] $f_{g,h}(q) \in \frac{1}{\Delta(q)} \QMod_{2g}$,
\item[(b)] $\frac{d}{dC_2} f_{g,h} = (2h-2) f_{g-1,h}$.
\end{enumerate}
\end{lemma}
\begin{proof}
Let $\beta_h \in \Pic(S)$ be a primitive curve class
satisfying $\langle \beta_h, \beta_h \rangle = 2h-2$.
With the same notation as in \eqref{DefnNghd} let
\[ \mathsf{N}'_{g,h,d} = 
\int_{ [ \Mbar_{g,1}(X, (\beta_h,d)) ]^{\text{red}} }
\ev_1^{\ast}\left(
\pi_1^{\ast} ( \beta_h^{\vee} ) \cup \pi_2^{\ast} ( \pt ) \right)
\]
be the \emph{connected} reduced Gromov--Witten invariant
in class $(\beta_h,d)$. Define
\[
f'_{g,h}(q) = \sum_{d \geq 0} \mathsf{N}'_{g,h,d} q^d.
\]
By
\cite[Prop.1]{K3xE}
the connected and disconnected invariants are related by
\[
f_{g,h}(q) = \frac{1}{\Delta(q)} \cdot f'_{g,h}(q).
\]

An application of Behrend's product formula\footnote{The arguments
of \cite{B2} carry over to the reduced virtual class.
Alternatively, we may use a degeneration argument
similar to \cite[Prop.5]{K3xE} to reduce to the
standard case.} \cite{B2} yields
\[
f'_{g,h}(q)
=
\int_{\Mbar_{g,1}}
\CC_{g}(\pt) \cup
\pi_{\ast}\left( [\Mbar_{g,1}(S,\beta_h)]^{\text{red}} \ev_1^{\ast}(\beta_h^{\vee}) \right),
\]
where $\pi : \Mbar_{g,n}(S,\beta_h) \to \Mbar_{g,n}$
is the forgetful map.
By Corollary~\ref{CorE} therefore 
\[ f'_{g,h}(q) \in \QMod_{2g}. \]

By Theorem~\ref{thm_E_HAE} we have further
\[
\frac{d}{dC_2} \CC_g(\pt)
= \iota_{\ast} p^{\ast} \CC_{g-1}(\pt)
\]
where $p : \Mbar_{g-1,3} \to \Mbar_{g-1,1}$ is the map
forgetting the last two points. Hence
\begin{align} \label{132435}
\frac{d}{dC_2} f'_{g,h}
& = 
\int_{\Mbar_{g-1,1}}
\CC_{g-1}(\pt) \cup
p_{\ast} \iota^{\ast} \pi_{\ast}\left( [\Mbar_{g,1}(S,\beta_h)]^{\text{red}} \ev_1^{\ast}(\beta_h^{\vee}) \right).
\end{align}
By the compatibility
of the reduced virtual class under gluing
and by the divisor equation
we have
\begin{align*}
& 
p_{\ast} \iota^{\ast} \pi_{\ast}\left( [\Mbar_{g,1}(S,\beta_h)]^{\text{red}} \ev_1^{\ast}(\beta_h^{\vee}) \right) \\
& =
p_{\ast} \pi_{\ast}\left( [\Mbar_{g-1,3}(S,\beta_h)]^{\text{red}} \ev_1^{\ast}(\beta_h^{\vee}) (\ev_2 \times \ev_3)^{\ast}(\Delta_S) \right)  \\
& = 
\blangle \beta_h, \beta_h \brangle 
\pi_{\ast}\left( [\Mbar_{g-1,1}(S,\beta_h)]^{\text{red}} \ev_1^{\ast}(\beta_h^{\vee}) \right),
\end{align*}
where $\Delta_S \in H^{\ast}(S\times S)$ is the class
of the diagonal. Plugging into \eqref{132435}
and using the product formula again we conclude that
\[ \frac{d}{dC_2} f'_{g,h}
= \langle \beta_h, \beta_h \rangle f'_{g-1,h}(q). \qedhere \]
\end{proof}

\begin{lemma} \label{Lemma9458}
For all $g$ and $d$ the series
$\CZ_{g,d}(q) = \sum_{h \geq 0} \mathsf{N}_{g,h,d} q^{h-1}$
satisfies
\begin{enumerate}
\item[(a)] $\CZ_{g,d}(q) \in \frac{1}{\Delta(q)} \QMod_{2g}$,
\item[(b)] $\frac{d}{dC_2} \CZ_{g,d} = (2d-2) \CZ_{g-1,d}$.
\end{enumerate}
\end{lemma}
\begin{proof}
Let $S \to \p^1$ be an elliptic surface with section $B$
and fiber class $F$, let $\beta_h = B + hF$ and define
\[ \widetilde{\CK}_{g}(\gamma_1, \ldots, \gamma_n)
=
\sum_{h \geq 0} q^{h-1} \pi_{\ast} 
\left( [\Mbar_{g,n}(S,\beta_h)]^{\text{red}} \prod_i \ev_i^{\ast}(\gamma_i) \right),
\]
where $\pi : \Mbar_{g,n}(S,\beta_h) \to \Mbar_{g,n}$
is the forgetful map.

Consider the generating series
of connected invariants
\[
\CZ'_{g,d}(q) = \sum_{h=0}^{\infty} \mathsf{N}'_{g,h,d} q^{h-1},
\]
and recall the classes $\CC_{g,d}(\ldots)$
from \eqref{GWclasses}.
By the product formula we have
\[
\CZ'_{g,d}(q)
=
\int_{\Mbar_{g,1}}
\widetilde{\CK}_{g}(F) \cup \CC_{g,d}(\pt).
\]
By a result of Faber and Pandharipande \cite{FPrel}
and a degeneration argument,
$\CC_{g,d}(\pt)$ is a tautological class.
Hence \cite[Prop.29]{MPT} resp. \cite[Sec.4.6]{BOPY} yields
\[ \CZ'_{g,d}(q) \in \frac{1}{\Delta(q)} \QMod_{2g}. \]

If Conjecture~\ref{conj_K3HAE} would hold we have
\begin{equation} \label{53134}
\frac{d}{dC_2}
\widetilde{\CK}_g(F)
=
\iota_{\ast} \widetilde{\CK}_{g-1}(F, \Delta_{\p^1})
+
2 \cdot 24 \cdot j_{\ast} \big( \widetilde{\CK}_{g-1}(F,F) \times [\Mbar_{1,1}] \big)
\end{equation}
where $\Delta_{\p^1}$ is the pullback
of the diagonal under $S^2 \to \p^1$.
By Theorem~\ref{thm_K3HAE} equation \eqref{53134}
holds after integration against any tautological class.
Hence
\begin{multline*}
\frac{d}{dC_2} \CZ'_{g,d}
=
\int_{\Mbar_{g-1,3}}
\widetilde{\CK}_{g-1}(F, \Delta_{\p^1})
\cup \CC_{g-1,d}(\pt, \Delta_E) \\
+
48 \sum_{d = d_1 + d_2}
\left( \int_{\Mbar_{g-1,2}} \widetilde{\CK}_{g-1}(F, F) \cup
\CC_{g-1,d_1}(\pt, \1) \right)
\times
\int_{\Mbar_{1,1}} \CC_{1,d_2}(\pt).
\end{multline*}
Rewriting in terms of the Gromov--Witten theory of $X$,
using the divisor equation and
\[ \int_{\Mbar_{1,1}} \CC_{1,d_2}(\pt) = [ C_2(q) ]_{q^{d_2}} \]
we find
\begin{equation} \label{53315645} \frac{d}{dC_2} \CZ'_{g,d}
=
2d \CZ'_{g-1,d}
+ 48 \sum_{d=d_1 + d_2} \CZ'_{g-1,d_1} \cdot [ C_2(q) ]_{q^{d_2}}.
\end{equation}

Consider the generating series
\[ \CZ'_g(q,\tilde{q}) = \sum_d \CZ'_{g,d}(q) \tilde{q}^{d},
\quad \CZ_g(q,\tilde{q}) = \sum_d \CZ_{g,d}(q) \tilde{q}^{d-1}.
\]
By \cite[Prop.1]{K3xE} we have
\[ \CZ_g(q,\tilde{q}) = \CZ'_g(q,\tilde{q}) \cdot \Delta^{-1}(\tilde{q}). \]
Rewriting \eqref{53315645} yields
\[
\frac{d}{dC_2(q)} \CZ_g' = 2 D_{\tilde{q}} \CZ_{g-1}'
+ 48 \CZ_{g-1}' C_2(\tilde{q})
\]
where $D_{\tilde{q}} = \tilde{q} \frac{d}{d \tilde{q}}$.
Using the identity
\[ C_2 = \frac{1}{24} q \frac{d}{dq} \log( \Delta^{-1} ),\]
we conclude that
\[ \frac{d}{dC_2(q)} \CZ_g =2 D_{\tilde{q}} \CZ_{g-1}. \qedhere \]
\end{proof}

\begin{proof}[Proof of Proposition~\ref{ConstraintProp2}]
Property 1 holds by definition, and Property 3 holds
by Lemma~\ref{Lemma9458}.
By \cite[Thm.4]{OS1} and the Gromov--Witten/Pairs correspondence \eqref{GWPTcorr} we have
\begin{equation} \label{393i3}
\CZ_h(u,\tilde{q}) = [ \CZ ]_{q^{h-1}}
=
\frac{\Theta(u,\tilde{q})^{2h-2}}{\Delta(\tilde{q})}
\sum_{i=0}^{h}
f_i(\tilde{q}) \wp^{h-i}(u,\tilde{q})
\end{equation}
for some $f_i \in \QMod_{2i}$.
By Lemma~\ref{Lemma9457} we have 
\[ \frac{d}{dC_2} \CZ_h = (2h-2) u^2 \CZ_h. \]
Applying $d/dC_2$ to \eqref{393i3}
and using the last equation and \eqref{dddc2} we find
\[
\frac{\Theta^{2h-2}}{\Delta} \sum_{i=0}^{h}
\left( \frac{d}{dC_2} f_i\right) \wp^{h-i} = 0. \]
This implies
$\frac{d}{dC_2} f_i = 0$,
hence that $f_i \in \Mod_{2h-2i}$,
and therefore Property~2 holds.
\end{proof}

\subsection{Proof of Theorem~\ref{thmIgusa}}
\label{Subsection_Completeproof}
By Propositions \ref{ConstraintProp1} and \ref{ConstraintProp2}
respectively
the series $\CZ(u,q,\tilde{q})$
and $\CF(u,q,\tilde{q})$ both
satisfy Properties 1--3 of Section~\ref{Subsection_Constraints}, so their difference does as well.
Moreover the Gromov--Witten invariant
$\mathsf{N}_{0,0,0} = 1$
matches the $u^{-2} q^{-1} \tilde{q}^{-1}$-coefficient in $\CF$.
We conclude that $\CZ = \CF$
by Proposition~\ref{PropUniqueness}. \qed

\appendix
\section{Elliptic functions and quasimodular forms}
\label{Section_Graphs_and_quasimodularforms}
\subsection{Overview}
We prove that for certain multivariate elliptic functions $F$, the constant term of the Fourier expansion of $F$ (in the elliptic parameter) is a quasimodular form. We also calculate the $C_2$-derivative of these quasimodular forms.
In Section~\ref{Subsection_SinglevarEllfun}
we treat the single variable case as a warm-up
for the general case which appears in Section~\ref{sec:multi}.
The main result of this appendix is Theorem~\ref{THMAPP2}.


\subsection{Preliminaries}
Let $z \in \BC$ and $\tau \in \BH$, where $\BH = \{ z \in \BC | \mathrm{Im}(z)>0 \}$ is the upper half plane.
We will use the auxiliary variables
\[ w = 2 \pi i z, \quad p = e^{2 \pi i z}, \quad q = e^{2 \pi i \tau}. \]
The operator of differentiation with respect to $z$ is denoted
\[ \partial_z = \frac{1}{2 \pi i} \frac{d}{dz} = \frac{d}{dw} = p \frac{d}{dp} \]
and for the $k$-th derivative of a function $f(z)$ we write 
\[ f^{(k)}(z) = \partial_z^k f(z). \]
For any meromorphic function $f(z)$
we let $[ f(z) ]_{(z-a)^{\ell}}$ denote the coefficient of $(z-a)^{\ell}$
in the Laurent expansion around $a$. The residue at $a$ is
\[ \mathrm{Res}_{z=a} f(z) \, = \, \big[ f(z) \big]_{(z-a)^{-1}}. \]
If $f(z) = g(z) h(z)$ where $h(z)$ is regular at $a$ we have
\begin{equation} \mathrm{Res}_{z=a} f(z) = \sum_{k \geq 1} \big[ g(z) \big]_{(z-a)^{-k}}
\frac{(2 \pi i)^{k-1} h^{(k-1)}(a)}{(k-1)!}. \label{RFEREas}
\end{equation}

\subsection{Elliptic functions} \label{Subsection_SinglevarEllfun}
Consider the Eisenstein series $C_{2k}(\tau)$
defined in \eqref{EisensteinSeries}
as functions on $\BH$ under the change of variables $q= e^{2 \pi i \tau}$.
Consider also the Weierstra{\ss} function $\wp(z)$ which
has Laurent expansion
\begin{equation}\label{eq:wp}
\wp(z) = \frac{1}{12} + \frac{p}{(1-p)^2} + \sum_{d \geq 1} \sum_{k | d} k (p^k - 2 + p^{-k}) q^{d}
\end{equation}
in the region $0 < |q| < |p| < 1$, and has Laurent expansion
\[ \wp(z) = \frac{1}{w^2} + \sum_{k \geq 2} (2k-1) 2k C_{2k}(\tau) w^{2k-2} \]
at $w=0$.

Let $\E$ be the ring generated by quasimodular forms and derivatives of the Weierstra{\ss} function,
\[ \E = \BQ\left[ C_2(\tau), C_4(\tau), C_6(\tau), \wp^{(k)}(z) \middle| \, k \geq 0 \, \right]. \]
The ring is graded by weight:
\[ \E = \bigoplus_{k \geq 0} \E_{k}, \]
where $C_{k}$ has weight $k$ and $\wp^{(k)}(z)$ has weight $2+k$.
We also let
\[ \frac{d}{dC_2} : \E \to \E \]
be the formal differentiation with respect to the generator $C_2$.\footnote{There exist relations among the generators of $\E$ but they
do not involve $C_2$. The ring $\E$ is free over $\BQ[ C_4, C_6, \wp^{(k)}(z)|k\geq 0]$ and the derivative
with respect to $C_2$ is well-defined.}

Every $F(z) \in \E$ admits a Fourier expansion in the region $0 < |q|< |p|<1$,
\[ F(z) = \sum_{n \in \BZ} a_n(\tau) p^n. \]
The constant term in the expansion is denoted by
\[ \big[ F(z) \big]_{p^0} = a_0(\tau). \]
As a warm-up for the general case we prove the following proposition.

\begin{prop} \label{4dji9isfg}
For every $F \in \E_k$
the series $\big[ F \big]_{p^0}$ is a quasimodular form of weight $k$
and we have
\[ \frac{d}{dC_2} \Big[ \, F \, \Big]_{p^0}
=
\left[ \frac{d}{dC_2} F \right]_{p^0} - 2 \big[ \, F \, \big]_{w^{-2}}.
\] 
\end{prop}
\vspace{7pt}

Consider the function
\begin{align*}
\mathsf{A}(z) 
& = - \frac{1}{2} - \sum_{m \neq 0} \frac{p^m}{1-q^m} \\
& = \frac{1}{w} - \sum_{\ell \geq 1} 2\ell C_{2\ell}(q) w^{2 \ell - 1},
\end{align*}
where the expansion in $p,q$ is taken in the region $0 < |q| < |p| < 1$.
For the proof of the Proposition we require the following Lemma.

\begin{lemma}\label{LemmaA}
$\mathsf{A}(z + \lambda \tau + \mu) = \mathsf{A}(z) - \lambda$ for every $\lambda, \mu \in \BZ$.
\end{lemma}
\begin{proof}
We have
$\mathsf{A}(z) = \partial_z \log \Theta(z)$
where $\Theta$ is the Jacobi theta function
\[ \Theta(z) =  (p^{1/2} - p^{-1/2}) \prod_{m \geq 1} 
\frac{ (1-pq^m) (1-p^{-1}q^m)}{ (1-q^m)^2 }. \]
A direct check using this definition shows
\[ \Theta(z + \lambda \tau + \mu) = (-1)^{\lambda + \mu} p^{-\lambda} q^{-\lambda^2/2} \Theta(z) \]
for all $\lambda, \mu \in \BZ$ which implies the claim.
\end{proof}

\begin{proof}[Proof of Proposition~\ref{4dji9isfg}] We have
\[
\big[ F(z) \big]_{p^0} = \int_{\C_a} F(z) \dd{z}
\]
where $\C_a$ is the line segment from $a$ to $a+1$ for some
$a \in \BC$ with $0 < \mathrm{Im}(a) < \mathrm{Im}(\tau)$.
Since $F(z)$ is periodic, i.e.
\[ F(z + \lambda \tau + \mu) = F(z) \]
for every $\lambda, \mu \in \BZ$, we may instead assume
$-\mathrm{Im}(\tau) < \mathrm{Im}(a) < 0$.

By Lemma~\ref{LemmaA} the function $f(z) = F(z) \cdot \A(z)$ satisfies
\[ f(z+1) = f(z), \quad f(z+\tau) = f(z) - F(z). \]
Hence we may replace the integral of $F$ over $\C_a$ by the integral
of $F \cdot \A$ over the boundary of the fundamental domain $B_a$ depicted in Figure~\ref{FigureBOXB_i},
\[
\big[ F(z) \big]_{p^0} = \oint_{B_a} F(z) \cdot \A(z) \dd{z}.
\]
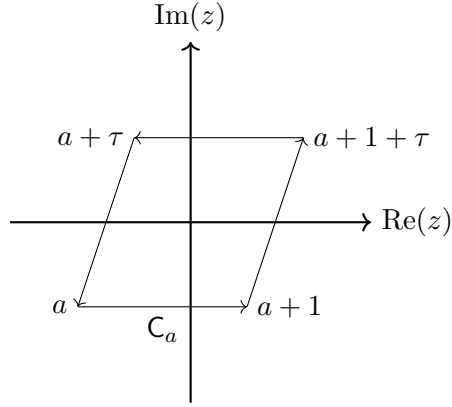
\begin{figure}
\centering
\begin{tikzpicture}[
    scale=1.5,
    axis/.style={thick, ->},
    ]
\draw[axis] (-0.1,1.5)  -- (3.1,1.5) node(xline)[right] {$\mathrm{Re}(z)$};
\draw[axis] (1.5,-0.1) -- (1.5,3.1) node(yline)[above] {$\mathrm{Im}(z)$};
\draw[->] (0.5, 0.75) node[left]{$a$} -- (1.25, 0.75) node[below]{$\C_a$} -- (2.0,0.75) node[right]{$a + 1$};
\draw[->] (2.0, 0.75) -- (2.5,2.25) node[right]{$a+1+\tau$};
\draw[->] (2.5, 2.25) -- (1,2.25) node[left]{$a+\tau$};
\draw[->] (1,2.25) -- (0.5, 0.75);
\end{tikzpicture}
\caption{The closed path $B_{a}$.}
\label{FigureBOXB_i}
\end{figure}
Since both $F$ and $\A$ have poles inside $B_a$ only at $0$,
an application of the residue theorem gives
\begin{align*}
\big[ F(z) \big]_{p^0}
& = \big[ F(z) \cdot \A(z) \big]_{w^{-1}} \\
& = [ F ]_{w^0} - \sum_{\ell \geq 1} 2 \ell C_{2\ell}(\tau) [ F(z) ]_{w^{-2 \ell}}.
\end{align*}
An inspection of the Laurent series of $F(z)$ yields now both claims.
\end{proof}

\subsection{Multiple variables}\label{sec:multi}
Let $n \geq 2$ and $z = (z_1, \ldots, z_n) \in \BC^n$, and denote
\[ w_a = 2 \pi i z_a, \quad p_a = e^{2 \pi i z_a}, \quad a \in \{ 1, \ldots, n \}. \]
Every permutation $\sigma \in S_n$ determines a region $U_{\sigma} \subset \BC^n$ by requiring
\begin{equation} \mathrm{Im}(\tau) > \mathrm{Im}(z_a - z_b) > 0 \label{49gjfdfg} \end{equation}
whenever $\sigma(a) > \sigma(b)$, or equivalently by
\begin{equation} \mathrm{Im}(z_{\sigma^{-1}(n)}) > \ldots > \mathrm{Im}(z_{\sigma^{-1}(1)}) > \mathrm{Im}(z_{\sigma^{-1}(n)} - \tau).
 \label{dsoijsdfjio}
\end{equation}

Consider the ring of multivariate elliptic functions
\[ \Sb = \BQ\left[\, C_2(\tau), C_4(\tau), C_6(\tau), \wp^{(k)}( z_a - z_b )\ \middle|\ k \geq 0,\, 1 \leq a < b \leq n \, \right]. \]
We assign $\wp^{(k)}$ and $C_k$ the weights $2+k$ and $k$ respectively and let
\[ \Sb = \bigoplus_{k \geq 0} \Sb_k \]
be the induced grading by weight $k$.
Let
\[ \frac{d}{dC_2} : \Sb \to \Sb \]
be the formal differentiation with respect to the generator $C_2$.

Every $F \in \ME$
has a well-defined Fourier expansion in the region $U_{\sigma}$,
\[ F = \sum_{k_1, \ldots, k_n \in \BZ} a_{k_1 \ldots k_n}(\tau) p_1^{k_1} \ldots p_n^{k_n}, \quad (z_1, \ldots, z_n) \in U_{\sigma}. \]
The constant coefficient in this expansion,
i.e. the coefficient of $\prod_i p_i^0$, is denoted
\[ \left[ F \right]_{p^0, \sigma} = a_{0\ldots 0}(\tau). \]
Define the constant coefficient of $F$ averaged over all permutation $\sigma$,
\begin{equation} \label{CSTCOEFFICIENTAVERAGED}
\left[ F(z_1, \ldots, z_n) \right]_{p^0}
= \frac{1}{n!} \sum_{\sigma \in S_n} \left[ F(z_1, \ldots, z_n) \right]_{p^0, \sigma}.
\end{equation}

The following is the main result of this appendix\footnote{
The first part of Theorem~\ref{THMAPP2}
can also be found in work of Goujard and M\"oller \cite{GM}.
Our argument gives a new proof of their result.
We thank M.~Raum for pointing out this connection.
}.

\begin{thm} \label{THMAPP2} Let $F \in \ME_k$. Then the following holds.
\begin{enumerate}
 \item $\big[ F(z) \big]_{p^0, \sigma} \in \QMod_{\leq k}$ for every permutation $\sigma$. 
 \item $\big[ F(z) \big]_{p^0} \in \QMod_k$.
 \item We have
\[ \frac{d}{dC_2} \Big[ F(z) \Big]_{p^0}
=
\left[ \frac{d}{dC_2} F \right]_{p^0} - \sum_{\substack{a,b = 1 \\ a \neq b}}^{n} \left[ (2 \pi i)^2 \mathrm{Res}_{z_a = z_b}\Big(  (z_a - z_b) \cdot F \Big) \right]_{p^0}.
\]
\end{enumerate}
\end{thm}

\subsection{Preparations for the proof}
We prove a series of results leading up to the proof of Theorem~\ref{THMAPP2} in Section~\ref{Subsection_Proof_of_Thm_Appendix}.

\begin{lemma} \label{Lemma_cyclic_permutation}
Let $F \in \Sb$ and $\sigma \in S_n$.
Then 
\[ [ F ]_{p^0, \sigma} = [ F ]_{p^0, \widetilde{\sigma}} \]
for every cyclic permutation $\widetilde{\sigma}$ of $\sigma$.
\end{lemma}
\begin{proof}
Let $(a_1, \ldots, a_n) \in U_{\sigma}$ and let $\C_{a_i}$ be the
line segment from $a_i$ to $a_i + 1$ in the $z_i$-plane. Then
\[ [ F ]_{p^0, \sigma} = \int_{\C_{a_1}} \cdots \int_{\C_{a_n}}  F(z_1, \ldots, z_n) \dd{z_n} \cdots \dd{z_{1}}.  \]
Since $F$ is periodic, i.e. $F(z + \lambda \tau + \mu) = F(z)$ for every $\lambda, \mu \in \BZ^n$,
we may replace the integral over $\C_{a_{\sigma^{-1}(n)}}$ by the integral over $\C_{a_{\sigma^{-1}(n)}-\tau}$.
But comparing with \eqref{dsoijsdfjio} this corresponds to taking the constant coefficient of $F$
with respect to a cyclic permutation of $\sigma$.
\end{proof}

For every $a \neq b$ let $R_{ab}$ denote the operation of taking the residue in $z_a = z_b$ written as a right operator,
\[ f(z_1, \ldots, z_n) R_{ab} := 2 \pi i \cdot \mathrm{Res}_{z_a = z_b} f(z_1, \ldots, z_n). \]
We also write
\[ \A_{ab} = \A(z_a - z_b). \]

\begin{lemma} \label{erkfprogdf}
Let $F(z) \in \Sb_k$, and let $i_1, \ldots, i_m \in \{ 1, \ldots, n \}$ be pairwise distinct.
Then for any $r \geq 0$ we have
\[ \big( F(z) \A_{i_1 i_{m}}^{r} \big) R_{i_1 i_2} R_{i_2 i_3} \cdots R_{i_{m-1} i_{m}} \in \Sb_{k+r-(m-1)}. \]
\end{lemma}
\begin{proof}
Let $F(z) \in \Sb_k$ be a monomial in the generators and consider the splitting
\[ F(z) = F_{ab}(z_a - z_b) \cdot \widetilde{F}_{ab}(z), \]
where $F_{ab}$ is the product of all factors in $F$ of the form $\wp^{(s)}(z_a - a_b)$ for some~$s$.
In particular, $\widetilde{F}_{ab}(z)$ is regular at $z_a = z_b$.

Consider the action of $R_{ab}$ on $F(z) \A_{ac}^r$. If $b \neq c$ we have
\begin{equation} \label{524524} (F \A_{ac}^r) R_{a b} =
\sum_{\ell \geq 1}
\left[ F_{ab} \right]_{(w_a - w_b)^{-\ell}}
\frac{1}{(\ell - 1)!}
\partial_{z_a}^{\ell-1} \left( \widetilde{F}_{ab} \A^r_{ac}  \right) \Big|_{z_a = z_b},
\end{equation}
where we have used \eqref{RFEREas}. Since
\[ \partial_z \A(z) = - \wp(z) - 2 C_2(\tau) \]
the right hand side of \eqref{524524}
can be written as a sum of terms
\[ F'(z) \cdot \A^{r'}_{bc}, \]
where $F' \in \Sb_{k'}$ with $k' + r' = k+r-1$.
Similarly, if $b = c$ we have
\[ ( F(z) \A_{ac}^r ) R_{a c} \in \Sb_{k+r-1}. \]
The claim follows from the steps above and an induction argument.
\end{proof}

Let $\sigma \in S_n$ be a permutation,
let
\[
g_{ab} =
\begin{cases}
1 & \text{ if }  \sigma(a) > \sigma(b), \\
0 & \text{ otherwise},
\end{cases}
\]
and for all $x \in \BC$ and non-negative integers $a$ define
\[ \binom{x}{a} = \frac{x \cdot (x-1) \cdots (x-a+1)}{a!}. \]

\begin{prop} \label{Proposition_p0_Reduction_with_sigma} Let $F(z) \in \ME$. Then
 \begin{multline*}
\big[ F(z) \big]_{p^0, \sigma} \\
= 
\sum_{\ell \geq 1} \sum_{i_1, i_2, \ldots, i_{\ell}}
\left[ F \cdot \binom{ \A_{1n} + \ell - 2
- g_{i_1 i_2} - \ldots - g_{i_{\ell-1} i_{\ell}}
}{\ell-1}
R_{i_1 i_2} \cdots R_{i_{\ell-1} i_{\ell}} \right]_{p^0, \sigma},
\end{multline*}
where the inner sum is over all non-recurring\footnote{
A sequence $x_1, x_2, x_3, \ldots $ is non-recurring if
$x_i \neq x_j$ for all $i \neq j$.}
sequences $i_1, \ldots, i_{\ell} \in \{ 1, \ldots, n \}$
with endpoints $i_1=1$ and $i_{\ell}=n$.
\end{prop}

\begin{proof}
 We argue by induction on $L$ that for every $L \geq 1$ we have
\begin{multline} \label{aaaaaaa}
\big[ F(z) \big]_{p^0, \sigma}
= \\
\sum_{\ell = 1}^{L} \sum_{i_1, i_2, \ldots, i_{\ell}}
\left[ F \cdot \binom{ \A_{1n} + \ell - 2 - g_{i_1 i_2} - \ldots - g_{i_{\ell-1} i_{\ell}} }{\ell-1}
R_{i_1 i_2} \cdots R_{i_{\ell-1} i_{\ell}} \right]_{p^0, \sigma},
\end{multline}
where the inner sum runs over all non-recurring sequences $(i_1, \ldots, i_{\ell})$ 
such that $i_1 = 1$ and the following holds:
\begin{itemize}
 \item if $\ell < L$ then $i_{\ell} = n$,
 \item if $\ell = L$ and $i_r = n$ then $r=\ell$.
\end{itemize}

If $L = 1$ equality \eqref{aaaaaaa} holds by definition.
Hence we may assume the claim holds for $L \geq 1$ and we show the case $L+1$.
Every summand on the right hand side of \eqref{aaaaaaa} with $i_{\ell} \neq n$
is equal to the $p^0$-coefficient (in $U_{\sigma}$) of
\begin{equation} \label{INTTTT}
\int_{\C_{a}}
F \cdot \binom{ \A_{1n} + \ell - 2 - g_{i_1 i_2} - \ldots - g_{i_{\ell-1} i_{\ell}} }{\ell - 1}
R_{i_1 i_2} \cdots R_{i_{\ell-1} i_{\ell}} \dd{z_{i_{\ell}}}
\end{equation}
for some $a \in \BC$ such that 
\[ (z_1, \ldots, z_{i_{\ell}-1}, a, z_{i_{\ell}+1},  \ldots, z_n) \in U_{\sigma} \]
and $\C_a$ is the line segment from $a$ to $a+1$ in the $z_{i_{\ell}}$-plane.
Define the function
\[ \mathsf{H}(z) =
 F \cdot \binom{ \A_{1n} + \ell - 1 - g_{i_1 i_2} - \ldots - g_{i_{\ell-1} i_{\ell}} }{\ell} R_{i_1 i_2} \cdots R_{i_{\ell-1} i_{\ell}}
\]
Using Lemma~\ref{LemmaA} and 
$\mathrm{Res}_{z=r+s} f(z) = \mathrm{Res}_{z=r}f(z+s)$
repeatedly we find
\begin{multline*}
\mathsf{H}(z_1, \ldots, z_{i_{\ell}} + \tau, \ldots , z_n)
=
\mathsf{H}(z_1, \ldots, z_n) \\
-
F \cdot \binom{ \A_{1n} + \ell - 2 - g_{i_1 i_2} - \ldots - g_{i_{\ell-1} i_{\ell}} }{\ell - 1} R_{i_1 i_2} \cdots R_{i_{\ell-1} i_{\ell}}.
\end{multline*}
Hence arguing as in the proof of Proposition~\ref{4dji9isfg}
we may replace \eqref{INTTTT} by an integral of $\mathsf{H}(z)$
over the box $B_{a}$ depicted in Figure~\ref{FigureBOXB_i}.
The function $\mathsf{H}$ has possible poles inside $B_{a}$ only at the points\footnote{Since $F$ and $\A$ are both $1$-periodic we may assume
there is no shift by an integer.}
\[ z_{i_{\ell}} = z_{i_{\ell+1}} + g_{i_{\ell} i_{\ell+1}} \tau \]
for some $i_{\ell+1} \notin \{ i_1, \ldots, i_{\ell} \}$.
By the residue theorem \eqref{INTTTT} is therefore
\[
2 \pi i \sum_{i_{\ell+1} \notin \{ i_1, \ldots, i_{\ell} \}}
\mathrm{Res}_{z_{i_{\ell}} = z_{i_{\ell+1}} + g_{i_{\ell} i_{\ell+1}} \tau}
\mathsf{H}(z),
\]
which after moving the shift by $g_{i_{\ell} i_{\ell+1}} \tau$ inside simplifies to
\[
\sum_{i_{\ell+1} \notin \{ i_1, \ldots, i_{\ell} \}}
F \cdot \binom{ \A_{1n} + \ell - 1 - \sum_{a=1}^{\ell} g_{i_a i_{a+1}} }{\ell}
R_{i_1 i_2} \cdots R_{i_{\ell-1} i_{\ell}} R_{i_{\ell} i_{\ell+1}}.
\]
Plugging back into \eqref{aaaaaaa} we obtain the case $L+1$. The induction is complete.
\end{proof}

Averaging Proposition~\ref{Proposition_p0_Reduction_with_sigma}
over all permutations $\sigma$ yields the following. 

\begin{prop} \label{Prop3534545} Let $F(z) \in \ME$. Then
\[
\big[ F(z) \big]_{p^0} = \sum_{m \geq 1}
\sum_{i_1=1, i_2, \ldots, i_{m+1}=n}
\left[
\left( F \cdot \frac{\A_{1n}^{m}}{m!}\right) R_{i_1 i_2} R_{i_2 i_3} \cdots R_{i_m i_{m+1}} \right]_{p^0},
\]
where the inner sum runs over all non-recurring
sequences $i_1, \ldots, i_{\ell+1} \in \{ 1, \ldots, n \}$
with endpoints $i_1=1$ and $i_{\ell+1}=n$.
\end{prop}
\begin{proof}
Setting $\ell = m+1$ in Proposition~\ref{Proposition_p0_Reduction_with_sigma} yields
\begin{multline*}
\big[ F(z) \big]_{p^0, \sigma}
=
\sum_{m \geq 1} \sum_{i_1, i_2, \ldots, i_{m+1}} \\
\left[ F \cdot \binom{ \A_{1n} + m - 1 - g_{i_1 i_2} - \ldots - g_{i_{m} i_{m+1}} }{m} R_{i_1 i_2} \cdots R_{i_{m} i_{m+1}} \right]_{p^0, \sigma},
\end{multline*}
where the non-recurring sequence $(i_1, \ldots, i_{m+1})$ satisfies $i_1 = 1$, $i_{m+1} = n$.

We sum the previous equation over all permutations $\sigma \in S_n$.
By Lemmas~\ref{Lemma_cyclic_permutation} and~\ref{erkfprogdf} it is enough to sum over all $\sigma$ with $\sigma(n) = n$.
It follows $g_{i_{m} i_{m+1}} = 0$ above.
We then split the sum over all such $\sigma$ into
a sum over orderings $\rho$ of the variables $z_{i}, i \notin \{ i_1, \ldots, i_{m}, n \}$,
a sum over orderings $\tau \in S_{m}$ of the variables $z_{i_1}, \ldots, z_{i_{m}}$
and the $\binom{n-1}{m}$ refinements of both orderings.
Since
\[ F \cdot \binom{ \A_{1n} + m - 1 - g_{i_1 i_2} - \ldots - g_{i_{m-1} i_{m}} }{m} R_{i_1 i_2} \cdots R_{i_{m} i_{m+1}} \]
depends only on the variables $z_i$ with $i \notin \{ i_1, \ldots, i_{m} \}$
we find
\begin{multline*}
 \big[ F(z) \big]_{p^0}
 =
\sum_{m \geq 1} \sum_{i_1, i_2, \ldots, i_{m+1}}
\sum_{\rho \in S_{n-m-1}} \frac{1}{(n-1)!} \cdot \binom{n-1}{m} \\
\cdot \left[
\sum_{\tau \in S_{m}} F \cdot \binom{ \A_{1n} + m - 1 - g_{i_1 i_2} - \ldots - g_{i_{m-1} i_{m}} }{m} R_{i_1 i_2} \cdots R_{i_{m} i_{m+1}}
\right]_{p^0, \tilde{\rho}} 
\end{multline*}
where $\tilde{\rho}$ is any fixed refinement
of the ordering $\rho$.
The proposition follows now by an application of Worpitzky's identity
\[
\sum_{\tau \in S_{m}} \binom{x+m-1-a_{\tau}}{m} = x^m,
\]
where $a_{\tau}$ is the number of ascents of $\tau$, i.e. the number of $i \in \{ 1, \ldots, \ell-1 \}$ with $\tau(i+1) > \tau(i)$.
\end{proof}

\begin{lemma} \label{Lemma_REsiduerelations}
The action of the residue operators $R_{ab}$
on meromorphic functions of variables $z_1,\ldots,z_n$ with poles only along $z_i-z_j = 0$ for $i<j$ satisfy
\[ R_{ab} R_{cb} = R_{cb} R_{ab} + R_{ca} R_{ab}, \ \quad \ R_{ab} R_{bc} = -R_{ba} R_{ac} \]
for all pairwise distinct $a,b,c$.
\end{lemma}
\begin{proof}
We may assume that
\[ f(z) = \prod_{1 \leq i < j < n} (z_i - z_j)^{m_{ij}} \]
for some $m_{ij} \in \BZ$. The claim follows then from a direct calculation.
\end{proof}

\subsection{Proof of Theorem~\ref{THMAPP2}}
\label{Subsection_Proof_of_Thm_Appendix}
We prove the quasimodularity of $[F]_{p^0, \sigma}$,
the homogeneity of $[F]_{p^0}$, and the formula
\begin{equation} \label{dfsdfsd}
\begin{aligned}
\frac{d}{dC_2} \Big[ F(z) \Big]_{p^0}
=
\left[ \frac{d}{dC_2} F \right]_{p^0}
& - 2\sum_{a < b = n} \left[ (w_a - w_b) F R_{ab} \right]_{p^0} \\
& - 2 \sum_{a < b < n} \left[ (w_b - w_a) F R_{ba} \right]_{p^0},
\end{aligned}
\end{equation}
which implies the formula in the Theorem by symmetrization over $S_n$.
We argue by induction on $n$, the number of variables $z_i$ on which $F$ depends.

If $n = 1$, then $F$ is a quasimodular form
and all three statements hold by inspection.
Assume the statement is known for all functions which depend on
a smaller number of variables.
By Proposition~\ref{Prop3534545}, we have
\[
\Big[ F(z) \Big]_{p^0} = \sum_{m \geq 1}
\sum_{i_1=1, i_2, \ldots, i_{m+1}=n}
\left[
\left( F \frac{\A_{1n}^{m}}{m!}\right) R_{i_1 i_2} R_{i_2 i_3} \cdots R_{i_m i_{m+1}} \right]_{p^0}.
\]
Each summand on the right
side depends on fewer variables than $F$
and is therefore a quasi-modular form of weight $k$
by Lemma~\ref{erkfprogdf} and induction.
To obtain \eqref{dfsdfsd} we apply the $\frac{d}{dC_2}$ operator,
use induction on the right side,
and use Lemma~\ref{Lemma_REsiduerelations} to commute the resulting $R_{ab}$
operators past the $R_{i_k i_{k+1}}$ operators.
This yields \eqref{dfsdfsd} also for $F$.
The quasimodularity of $[F]_{p^0, \sigma}$
(and the weight bound)
follows similarly from Lemma~\ref{erkfprogdf} and Proposition~\ref{Proposition_p0_Reduction_with_sigma}. \qed

\section{Elliptic fibrations} \label{Section_Elliptic_fibrations}
\subsection{Overview}
We present a refinement of Conjecture~\ref{Conj_HAE}
by weight, and give evidence in the case of
elliptic Calabi--Yau threefolds in fiber classes.

\subsection{Weight refinement} \label{Subsection_WeightE}
Let $\pi : X \to B$ be an elliptic fibration with a section
and integral fibers.
The holomorphic anomaly equation of Conjecture~\ref{Conj_HAE}
and the argument used in the proof of Corollary~\ref{CorE}
yield a refinement of Conjecture~\ref{Conj_Quasimodularity} by weight as follows.

Recall the divisor class $W$ defined in Section~\ref{Subsection_Elliptic_fibrations}.
The endomorphisms of $H^{\ast}(X)$ defined by
\[
T_+(\alpha) = ( \pi^{\ast} \pi_{\ast} \alpha ) \cup W, \quad
T_-(\alpha) = \pi^{\ast} \pi_{\ast} ( \alpha \cup W )
\]
satisfy
$T_+^2 = T_+$ and $T_-^2 = T_-$ as well as $T_+ T_- = T_- T_+ = 0$. Hence
the cohomology of $X$ splits as
\[ H^{\ast}(X) = \mathrm{Im}(T_+) \oplus \mathrm{Im}(T_-)
\oplus \big( \mathrm{Ker}( T_+ ) \cap \mathrm{Ker}(T_-) \big) .
\]
Define a modified degree function $\underline{\deg}(\gamma)$ by
the assignment
\begin{equation*} \label{modified_degree_function}
\underline{\deg}(\gamma) =
\begin{cases}
2 & \text{if } \gamma \in \mathrm{Im}(T_+) \\
1 & \text{if } \gamma \in \mathrm{Ker}( T_+ ) \cap \mathrm{Ker}(T_-) \\
0 & \text{if } \gamma \in \mathrm{Im}(T_-). \\
\end{cases}
\end{equation*}
If $X$ is an elliptic curve and $B$ is a point then
$\underline{\deg}$ specializes to the real cohomological degree $\deg_{\BR}$.

\begin{corstar} \label{Cors}
Assume Conjectures~\ref{Conj_Quasimodularity} and \ref{Conj_HAE}
hold.
Then for any $\underline{\deg}$-homogeneous classes
$\gamma_1, \ldots, \gamma_n \in H^{\ast}(X)$
we have
\[
\CC^{\pi}_{g, \kk}( \gamma_1, \ldots, \gamma_n )
\in 
H_{\ast}(\Mbar_{g,n}(B, \kk))
\otimes
\frac{1}{\Delta(q)^{m}}
\QMod_{\ell},
\]
where $m = -\frac{1}{2} c_1(N_{\iota}) \cdot \mathsf{k}$
and $\ell = 2g - 2 + 12m + \sum_i \underline{\deg}(\gamma_i)$.
\end{corstar}

\subsection{An example} \label{Subsection_CY3example} 
Let $X$ be a Calabi--Yau threefold and let
$\pi : X \to B$ be an elliptic fibration
with section and integral fibers over a Fano surface $B$.
We consider the genus $g$ Gromov--Witten potentials in fiber classes
\[ F_g(q) = \sum_{d = 0}^{\infty} q^d \int_{[ \Mbar_{g,0}(X,dF) ]^{\text{vir}}} 1 \]
with the convention that the summation starts at $d=1$ if $g \in \{ 0, 1 \}$.
By Toda's calculation \cite[Thm 6.9]{T12}, the Pandharipande--Thomas invariants $\mathsf{P}_{n,\beta}$
of $X$ in fiber classes form the generating series
\[
\sum_{d = 0}^{\infty} \sum_{n \in \BZ}
\mathsf{P}_{n, dF} y^n q^d
=
\prod_{\ell, m \geq 1} (1- (-y)^{\ell} q^{m} )^{-\ell \cdot e(X)}
\cdot \prod_{m \geq 1} (1-q^m)^{-e(B)} \,.
\]
Assuming $X$ satisfies the Gromov--Witten/Pairs correspondence \cite{PaPix1, PaPix2},
we therefore obtain
\begin{align*}
F_0(q) & = - e(X) \sum_{m, a \geq 1} \frac{1}{a^3} q^{ma} \\
F_1(q) & = \left( e(B) - \frac{1}{12} e(X) \right) \sum_{m,a \geq 1} \frac{1}{a} q^{ma} \\
F_g(q) & = e(X) \frac{(-1)^g B_{2g}}{4g} C_{2g-2}(q), \quad g \geq 2.
\end{align*}
If $g \geq 2$ the series
\[ \int \CC^\pi_{g,0}() = F_g(q) \]
is quasimodular of weight $2g-2$ in agreement with Corollary*~\ref{Cors}.

In genus $g \leq 1$ the series $F_0$ and $F_1$ are not quasimodular forms.
However, this does not contradict Corollary*~\ref{Cors}
since the moduli spaces $\Mbar_{g,0}(\p^1,0)$ are unstable here and $\CC^\pi_g()$ is not defined.
Instead, we need to add additional insertions to stabilize the moduli space. In genus $0$ we obtain
\begin{align*}
\int \CC^\pi_{0,0}(W,W,W) & = \int_X W^3 + \left( q \frac{d}{dq} \right)^3 F_0(q) = -12 e(X) C_4(q), \\
\int \CC^\pi_{0,0}(\pi^{\ast}D,W,W) & = \int_X \pi^{\ast} D \cup W^2 = 0,\\
\int \CC^\pi_{0,0}(\pi^{\ast}D, \pi^{\ast} D',W) & = \int_X \pi^{\ast} D \cup \pi^{\ast} D' \cup W = \int_B D \cdot D',
\end{align*}
for any $D, D' \in H^{2}(B)$, where in the first equality we used
\[ e(X) = -60 \int_{B} K_B^2. \]
All three evaluations are in perfect agreement with Corollary*~\ref{Cors}.

In genus $1$ we obtain agreement with Corollary*~\ref{Cors} by
\begin{align*}
\int \CC^{\pi}_{1,0}(W)
& = \int_{\Mbar_{1,1}(X,0)} \ev_1^{\ast}(W) + \left( q \frac{d}{dq} \right) F_1(q) \\
& = \left( e(B) -\frac{1}{12} e(X) \right) C_2(q),
\end{align*}
where we used
\[ c_2(X) = \pi^{\ast} c_2(B) + 11 \pi^{\ast} c_1(B)^2 + 12 \iota_{\ast} c_1(B). \]

A direct check shows that all evaluations above are also compatible with the conjectured holomorphic anomaly equation.
For example, in genus $1$ Conjecture~\ref{Conj_HAE} predicts correctly
\begin{align*}
\frac{d}{dC_2} \int \CC^{\pi}_{1,0}(W)
& = \int \CC_{0,0}^{\pi}(W, \Delta_B) - 2 \int \CC_{1,0}(\1) \psi_1 \\
& = e(B) - 2 \int_{\Mbar_{1,1}} \psi_1 \int_X c_3(X) \\
& = e(B) - \frac{1}{12} e(X).
\end{align*}


\begin{thebibliography}{99}
\bibitem{AS}
M.~Alim and E.~Scheidegger,
{\em Topological Strings on Elliptic Fibrations},
Commun. Number Theory Phys. {\bf 8} (2014), no. 4, 729--800. 

\bibitem{Beauville} A.~Beauville, {\em Counting rational curves on $K3$ surfaces}, Duke Math. J. {\bf 97} (1999), no. 1, 99--108. 

\bibitem{B2} K.~Behrend,
{\em The product formula for Gromov--Witten invariants},
J. Algebraic Geom. 8 (1999), no. {\bf 3}, 529--541. 

\bibitem{B} K.~Behrend,
{\em Donaldson--Thomas type invariants via microlocal geometry},
Ann. of Math. (2) {\bf 170} (2009), no. 3, 1307--1338. 

\bibitem{Bryan-K3xE} J.~Bryan, {\em The Donaldson--Thomas theory of $K3 \times E$ via the topological vertex}, \href{http://arxiv.org/abs/1504.02920}{arXiv:1504.02920}.

\bibitem{BL} J.~Bryan and N.~C.~Leung, {\em The enumerative geometry of $K3$ surfaces
and modular forms}, J. Amer. Math. Soc. {\bf 13} (2000), no. 2, 371--410.

\bibitem{BCOV}
M. Bershadsky, S. Cecotti, H. Ooguri and C. Vafa,
{\em Kodaira-Spencer theory of gravity and exact results for quantum string amplitudes}, Comm. Math. Phys. {\bf 165} (1994), no. 2, 311--427.

\bibitem{BOPY}
J.~Bryan, G.~Oberdieck, R.~Pandharipande, Q.~Yin, \emph{Curve counting on abelian surfaces and threefolds},
Algebr. Geom., to appear, \href{https://arxiv.org/abs/1506.00841}{arXiv:1506.00841}.

\bibitem{BSSZ}
A.~Buryak, S.~Shadrin, L.~Spitz, and D.~Zvonkine,
{\em Integrals of {$\psi$}-classes over double ramification cycles},
Amer. J. Math. {\bf 137} (2015), no. 3, 699--737.

\bibitem{EZ} 
M.~Eichler and D.~Zagier,
\newblock \textsl{ The theory of {J}acobi forms}, volume~55 of \textsl{
  Progress in Mathematics},
\newblock Birkh\"auser Boston Inc., Boston, MA, 1985.

\bibitem{FPrel}
C.~Faber and R.~Pandharipande, 
{\em Relative maps and tautological classes}, 
J. Eur. Math. Soc. {\bf 7} (2005), no. 1, 13--49.

\bibitem{FP13}
C.~Faber and R.~Pandharipande,
\newblock {\em Tautological and non-tautological cohomology of the moduli space of curves},
\newblock in {\em Handbook of moduli}, Vol. I, 293--330, Adv. Lect. Math. (ALM), {\bf 24}, Int. Press, Somerville, MA, 2013.


\bibitem{GM}
E.~Goujard and M.~M{\"o}ller,
{\em Counting Feynman-like graphs: Quasimodularity and Siegel-Veech weight},
to appear in J. European Math. Soc,
\href{https://arxiv.org/abs/1609.01658}{arXiv:1609.01658}.


\bibitem{GP}
T.~Graber and R.~Pandharipande,
{\em Constructions of nontautological classes on moduli spaces of curves},
Michigan Math. J. {\bf 51} (2003), no. 1, 93--109. 

\bibitem{Gr2} F.~Greer, in preparation.

\bibitem{GN} 
V.~A. Gritsenko and V.~V. Nikulin, \textsl{Siegel automorphic form corrections
  of some {L}orentzian {K}ac-{M}oody {L}ie algebras},
\newblock Amer. J. Math. \textbf{ 119} (1997), 181--224.

\bibitem{Jtaut}
F.~Janda,
{\em Gromov-Witten theory of target curves and the tautological ring},
Michigan Math. J. {\bf 66} (2017), no. 4, 683--698. 

\bibitem{JPPZ}
F. Janda, R. Pandharipande, A. Pixton and D. Zvonkine,
{\em Double ramification cycles on the moduli spaces of curves},
Publ. Math. Inst. Hautes \'Etudes Sci. {\bf 125} (2017), 221--266. 

 
\bibitem{KZ}
M.~Kaneko and D.~Zagier,
{\em A generalized Jacobi theta function and quasimodular forms}
in \emph{The moduli space of curves, (Texel Island, 1994)}, 165--172,
Progr. Math., {\bf 129}, Birkh\"auser Boston, Boston, MA, 1995. 

\bibitem{KKV} S. Katz, A. Klemm, and C. Vafa, {\em M-theory, topological
strings, and spinning black holes}, Adv. Theor. Math. Phys. {\bf 3} (1999),
1445--1537.

\bibitem{KMW}
A.~Klemm, K.~Manschot, T.~Wotschke,
{\em Quantum geometry of elliptic Calabi-Yau manifolds},
Commun. Number Theory Phys. {\bf 6} (2012), no. 4, 849--917. 

\bibitem{KT}
M.~Kool and R.~Thomas,
{\em Reduced classes and curve counting on surfaces I: theory},
Algebr. Geom. {\bf 1} (2014), no. 3, 334--383. 

\bibitem{LQ}
Y.-P. Lee and F. Qu,
{\em A product formula for log Gromov-Witten invariants},
\href{https://arxiv.org/abs/1701.04527}{arXiv:1701.04527}.

\bibitem{LP}
H.~Lho and R.~Pandharipande,
{\em Stable quotients and the holomorphic anomaly equation},
\href{https://arxiv.org/abs/1702.06096}{1702.06096}.

\bibitem{Junli1} J.~Li, {\em Stable morphisms to singular schemes and relative
stable morphisms}, J. Differential Geom. {\bf 57} (2001), no. 3, 509--578.

\bibitem{Junli2} J.~Li, {\em
A degeneration formula for Gromov-Witten invariants}, J. Differential Geom. {\bf 60} (2002), no. 2, 199--293.

\bibitem{M_An} D. Maulik, {\em Gromov-Witten theory of
$A_n$-resolutions}, Geom. and Top. {\bf 13} (2009), 1729--1773. 

\bibitem{MP} D.~Maulik and R.~Pandharipande, {\em A topological view of Gromov-Witten theory}, Topology {\bf 45} (2006), no. 5, 887--918.

\bibitem{GWNL} D. Maulik and R. Pandharipande, {\em 
Gromov--Witten theory and Noether--Lefschetz theory}, in {\em A 
celebration of algebraic geometry}, Clay Mathematics Proceedings {\bf 18},
469--507,
AMS (2010).

\bibitem{MPT}
D.~Maulik, R.~Pandharipande, and R.~P.~Thomas, {\em Curves on $K3$ surfaces and modular forms}, with an appendix by A.~Pixton, J. Topol. {\bf 3} (2010), no. 4, 937--996.

\bibitem{MRS}
T.~Milanov, Y.~Ruan and Y.~Shen,
{\em Gromov--Witten theory and cycle-valued modular forms},
Journal fur die Reine und Angewandte Mathematik, vol 2015, 2015.


\bibitem{HilbK3} G. Oberdieck, \emph{Gromov--Witten invariants of the Hilbert scheme of points of a K3 surface},
Geom. Topol. {\bf 22} (2018), no. 1, 323--437.

\bibitem{K3xP1} G.~Oberdieck, {\em Gromov--Witten theory of $\text{K3} \times \p^1$ and quasi-Jacobi forms},
Int. Math. Res. Not., to appear,
\href{http://arxiv.org/abs/1605.05238}{arXiv:1605.05238}

\bibitem{O1} G.~Oberdieck, {\em On reduced stable pair invariants},
Math. Z., to appear,
\href{http://arxiv.org/abs/1605.04631}{arXiv:1605.04631}.

\bibitem{K3xE}
G.~Oberdieck and R.~Pandharipande,
{\em Curve counting on $K3\times E$, the
Igusa cusp form $\chi_{10}$, and descendent integration},
in K3 surfaces and their moduli, C.~Faber, G.~Farkas, and G.~van der Geer, eds., Birkhauser Prog. in Math. 315 (2016), 245--278.

\bibitem{OS1} G.~Oberdieck and J.~Shen,
{\em Curve counting on elliptic Calabi--Yau threefolds
via derived categories},
JEMS, to appear,
\href{https://arxiv.org/abs/1608.07073}{arXiv:1608.07073}.

\bibitem{OSK0}
G.~Oberdieck and J.~Shen,
{\em Reduced Donaldson-Thomas invariants and the ring of dual numbers},
\href{https://arxiv.org/abs/1612.03102}{arXiv:1612.03102}.

\bibitem{OPix2}
G.~Oberdieck and A.~Pixton,
{\em  Gromov-Witten theory of elliptic fibrations: Jacobi forms and holomorphic anomaly equations},
\href{https://arxiv.org/abs/1709.01481}{arXiv:1709.01481}.

\bibitem{OP1}
A.~Okounkov and R.~Pandharipande, {\em Gromov-{W}itten theory, {H}urwitz theory, and completed cycles},
\newblock Ann. of Math. (2) {\bf 163} (2006), no. 2, 517--560.

\bibitem{OP3}
A.~Okounkov and R.~Pandharipande, {\em Virasoro constraints for target curves}, Invent. Math. {\bf 163} (2006), no. 1, 47--108.

\bibitem{PaPix1} R. Pandharipande and A. Pixton,
{\em Gromov--Witten/pairs descendent correspondence for toric 3-folds},
Geom. Topol. {\bf 18} (2014), no. 5, 2747--2821. 

\bibitem{PaPix2} R. Pandharipande and A. Pixton, {\em Gromov--Witten/Pairs
correspondence for the quintic 3-fold}, J. Amer. Math. Soc. {\bf 30} (2017), no. 2, 389--449.

\bibitem{PT1}
R.~Pandharipande and R.~P.~Thomas,
\newblock {\em Curve counting via stable pairs in the derived category},
Invent. Math. {\bf 178} (2009), no. 2, 407--447.

\bibitem{PT2}
R.~Pandharipande and R.~P.~Thomas,
{\em The Katz-Klemm-Vafa conjecture for K3 surfaces},
Forum Math. Pi {\bf 4} (2016), e4, 111 pp. 

\bibitem{PZ}
A.~Pixton and D.~Zagier,
in preparation.

\bibitem{T12} Y.~Toda,
{\em Stability conditions and curve counting invariants on Calabi--Yau 3-folds},
Kyoto J. Math. {\bf 52} (2012), no. 1, 1--50. 
\end{thebibliography}
\end{document}